\documentclass[11pt]{article}
\usepackage{amsmath, amsbsy, amsthm, amsfonts, array}
 \newtheorem{thm}{Theorem}[section]
 \newtheorem{cor}[thm]{Corollary}
 \newtheorem{lem}[thm]{Lemma}
 \newtheorem{prop}[thm]{Proposition}
 \newtheorem{claim}[thm]{Claim}
 \theoremstyle{definition}
 \newtheorem{defn}[thm]{Definition}
 \newtheorem{exam}[thm]{Example}

 \newtheorem{rem}[thm]{Remark}
 
 \newtheorem{addendum}[thm]{Addendum}
 \numberwithin{equation}{section}

 \newcommand{\A}{\mathcal{A}}

\def\Zee{\mathbb{Z}}

\def\Ar{\mathbb{R}}

\def\be{\begin{equation}}
\def\ee{\end{equation}}
\def\ba{\begin{eqnarray}}
\def\ea{\end{eqnarray}}
\def\tilde{\widetilde}

\def\e1{\epsilon}
\def\AAl{\mathcal{A}_{\lambda}}
\def\A0{\stackrel{\circ}{\AAl}}

\def\o1{\omega}
\def\01{\Omega}
\def\c1{\gamma}

\def\g1{\Sigma}

\def\bigcup{\cup}

\def\l1{\Lambda}

\def\v1{\varphi}

\def\d1{\delta}
\def\part{\partial}

\def\f1{\frac}
\def\t1{\theta}
\def\b1{\beta}
\def\bar{\overline}
\def\bs{\begin{eqnarray*}}
\def\es{\end{eqnarray*}}
\def\m1{\Theta}
\def\w1{\wedge}

\begin{document}

\author{John Morgan and Gang Tian \thanks{Supported partially by
NSF grants DMS 0706815 (Morgan), DMS 0703985 (Tian), and DMS 0735963(Tian)}}
\date{\today}
\title{Completion of the proof of the Geometrization Conjecture}
\maketitle

This paper builds upon and is an extension of \cite{MT}.
In this paper, we complete a proof of the following:

\medskip
\noindent {\bf Geometrization Conjecture: Any
closed, orientable, prime $3$-manifold $M$ contains a disjoint
union of embedded $2$-tori and Klein bottles such that each connected component of the
complement admits a locally homogeneous Riemannian metric of finite volume.}

Recall that a Riemannian manifold is {\em homogeneous} if its isometry group acts transitively on the underlying manifold; a {\em locally homogeneous} Riemannian manifold is the quotient of a homogeneous Riemannian manifold by a discrete group of isometries acting freely. Recall also that a {\em prime} $3$-manifold is one which is not diffeomorphic to $S^3$ and which is not a connected sum of two manifolds neither of which is diffeomorphic to $S^3$. It is a classic result in $3$-manifold topology, see \cite{Milnor} that every $3$-manifold is a connected sum of a finite number of prime $3$-manifolds, and this decomposition is unique up to the order of the factors.

The main part of this paper is devoted to giving a proof of Theorem 7.4 stated in \cite{P2} on locally volume collapsed $3$-manifolds with curvature bounded from below which is the last step  in the proof of the Geometrization Conjecture. In the introduction we will summarize the major results on Ricci flow contained in \cite{MT} and also briefly discuss results on Ricci flow in dimension $3$ beyond those contained in \cite{MT} that are needed in proving the Geometrization conjecture.

The Geometrization Conjecture was proposed by W. Thurston in early 1980's. It includes the Poincar\'e conjecture as a special case.
Thurston himself solved this conjecture for a large class of 3-manifolds, namely those containing an incompressible surface; i.e., an embedded surface of genus $\ge 1$ whose fundamental group injects into the fundamental group of the $3$-manifold.

Thurston's Geometrization Conjecture suggests the existence of especially nice metrics on 3-manifolds and consequently,
a more analytic approach to the problem of classifying 3-manifolds. Richard Hamilton formalized
one such approach in \cite{H} by introducing the Ricci flow equation on the space of Riemannian
metrics:
$$\frac{\partial g(t)}{\partial t}=-2{\rm Ric}(g(t)),$$
where ${\rm Ric}(g(t))$ is the Ricci curvature of the metric $g(t)$. In dimension 3, the fixed
points (up to rescaling) of this equation are the Riemannian metrics of
constant Ricci curvature. Beginning with any Riemannian manifold $(M,g_0)$, there is a solution $g(t)$ of this Ricci flow on $M$
for $t$ in some interval such that $g(0)=g_0$. It is hoped that if $M$ is a closed 3-manifold, then $g(t)$ exists for all $t>0$ after
appropriate scaling and converges to a nice metric outside a part with well-understood topology.

In \cite{HamiltonNSRF3M}, R. Hamilton showed that if the Ricci flow exists for all time and if there is an
appropriate curvature bound together with another geometric bound, then as
$t\rightarrow\infty$, after rescaling to have a fixed diameter, the metric
converges to a metric of constant negative curvature.

However, the general situation is much more complicated to
formulate and much more difficult to establish. There are many technical issues with
this program: One knows that in general the Ricci flow will develop singularities
in finite time, and thus a method for analyzing these singularities and continuing
the flow past them must be found. Furthermore, even if the flow continues for all
time, there remain complicated issues about the nature of the limiting object at
time $t = \infty$. For instance, if the topology of $M$ is sufficiently complicated, say it
is a non-trivial connected sum, then no matter what the initial
metric is one must encounter finite-time singularities, forced by
the topology. More seriously, even if $M$ has simple
topology, beginning with an arbitrary metric, one expects to (and
cannot rule out the possibility that one will) encounter finite-time
singularities in the Ricci flow. These singularities may occur along proper subsets of the
manifold, not the entire manifold. Thus, one is led to study a more general evolution process
called {\em Ricci flow with surgery},
first introduced by Hamilton in the context of four-manifolds, \cite{hamilton4}.
This evolution process is still parameterized by an interval in time, so that for each $t$
in the interval of definition there is a compact Riemannian $3$-manifold
$M_t$. But there is a discrete set of times at which the manifolds and metrics undergo topological and
metric discontinuities (surgeries). In each of the complementary
intervals to the singular times, the evolution is the usual Ricci
flow, though, because of the surgeries, the topological type of the
manifold $M_t$ changes as $t$ moves from one complementary interval
to the next. From an analytic point of view, the surgeries at the
discontinuity times are introduced in order to `cut away' a
neighborhood of the singularities as they develop and insert by
hand, in place of the `cut away' regions,  geometrically nice
regions. This allows one to continue  the Ricci flow (or more
precisely, restart the Ricci flow with the new metric constructed at
the discontinuity time). Of course, the surgery process also changes
the topology. To be able to say anything useful topologically about
such a process, one needs results about Ricci flow, and one also
needs to control both the topology and the geometry of the surgery
process at the singular times. For example, it is crucial for the
topological applications that we do surgery along $2$-spheres rather
than surfaces of higher genus. Surgery along $2$-spheres produces
the connected sum decomposition, which is well-understood
topologically, while, for example, Dehn surgeries along tori can
completely destroy the topology, changing any $3$-manifold into any
other.

The change in topology turns out to be completely understandable and
amazingly, the surgery processes produce exactly the topological
operations needed to cut the manifold into pieces on which the Ricci
flow can produce the metrics sufficiently controlled so that the
topology can be recognized.

Following Perelman in \cite{P2}, we gave a detailed proof
in \cite{MT} of the following long-time existence result for Ricci
flow with surgery:

\begin{thm}\label{surgery} Let $(M,g_0)$ be a closed Riemannian $3$-manifold.
Suppose that there is no embedded, locally separating $P^2$
contained\footnote{That is, no embedded $P^2$ in $M$ with trivial
normal bundle. Clearly, all orientable manifolds satisfy this
condition.} in $M$. Then there is a Ricci flow with surgery, say $(M_t,g(t))$, defined
for all $t\in [0,\infty)$ with initial metric $(M,g_0)$. The set of
discontinuity times for this Ricci flow with surgery is a discrete
subset of $[0,\infty)$. The topological change in the $3$-manifold
as one crosses a surgery time is a connected sum decomposition
together with removal of connected components, each of which is
diffeomorphic to one of $S^2\times S^1$, $\Ar P^3\#\Ar P^3$, the
non-orientable $2$-sphere bundle over $S^1$, or a manifold admitting
a metric of constant positive curvature. Furthermore, there are two decreasing functions
$r(t)>0$ and $\kappa(t)>0$ such that (1) $(M_t,g(t))$ is $\kappa(t)$-non-collapsed (see \cite{MT} Definition 9.1) and (2)
any point $x\in M_t$ with $R(g(t))\ge r(t)^{-2}$ satisfies the so called strong canonical neighborhood
assumption (see \cite{MT} Definition 9.78 and Theorem 15.9).
\end{thm}

Theorem~\ref{surgery} is central for all applications of Ricci
flow to the topology of three-dimensional manifolds.
The book \cite{MT} dealt with the case that
$M_t = \emptyset$ for $t$ sufficiently large, that is, the case when the Ricci flow with surgery becomes
extinct at finite time. Then it follows from the above theorem that the initial manifold $M$ is diffeomorphic to
a connected sum of copies of $S^2\times S^1$, the non-orientable $2$-sphere bundle over $S^1$, and
$S^3/ \Gamma$, where $\Gamma \subset O(4)$ is a finite group  acting freely on $S^3$.
It was shown in \cite{MT} that if $M$ is a simply-connected $3$-manifold,  then for any initial metric $g_0$ the corresponding Ricci flow with surgery becomes extinct at finite time, see also
(\cite{P3} and \cite{CM}). Consequently, $M$ is diffeomorphic to $S^3$, thus proving the Poincar\'e conjecture.

If $g_0$ has positive scalar curvature, then the Ricci flow with surgery $g(t)$ becomes extinct in time at most
$\frac{3}{2 a}$, where $a$ is a positive lower bound of the scalar curvature of $g_0$. This follows from the maximum principle
and the induced scalar curvature evolution equation for Ricci flow. By the above theorem, we see that  $ M $
in this case is diffeomorphic to a connected sum of  copies of $2$-sphere bundles over $S^1$ and
metric quotients of the round $S^3$. If the scalar curvature is only nonnegative, then by the
strong maximum principle it instantly becomes positive unless the metric is
(Ricci-)flat; thus in this case, $M$ is a flat manifold.

However, if the scalar curvature is negative somewhere, then one needs to analyze the asymptotic behavior
of $g(t) $ as $t$ goes to $\infty$. For this purpose, one first examines when the sectional curvature can be
bounded at $t=\infty$. This was given in Section 6 of Perelman's second paper \cite{P2} and more details can be found
in \cite{KL}, \cite{morgantian2}.

Roughly speaking, Perelman showed that for any $w > 0$, there are $\tau=\tau(w) >0$, $K = K(w) <
\infty$, $\bar r =\bar r(w) > 0$ and $\theta(w)>0$ such that  if the ball $B(x,t, r)$ of $(M_t,g(t))$
centered at $x$ and of radius $r$ has its sectional curvature bounded from below by
$-r^{-2}$, where $\theta^{-1}(w)h(t/2)\le r\le \bar r \sqrt{t}$ (Here $h(s)$ is the surgery scale at time $s$, see \S 4.4 of \cite{P2} or \S 15.3 of \cite{MT}.) and its volume is bounded from below by $ w r^3$, then  its scalar curvature is bounded by $Kr^{-2}$
on $B(x,t,r/4)\times [t-\tau r^2,t]$ (cf. Corollary 6.8, \cite{P2}). Using this and the strong canonical neighborhood assumption one can conclude that given $w>0$ for all $t$ sufficiently large if $B(x,t,r)$ has volume $\ge wr^3$ and sectional curvatures bounded below by $-r^{-2}$, then provided that $0<r\le \bar r(w)\sqrt{t}$, there are constants $K_m=K_m(w),\   m=0,1,\ldots,$ such that the $m^{th}$ covariant derivative of the Riemannian curvature tensor at $(x,t)$ is bounded by $K_mr^{-(m+2)}$.

By the above discussion, one may assume that our initial manifold does not admit a metric
with nonnegative scalar curvature, and that once we get a component with
nonnegative scalar curvature, it is immediately removed. Hence, we can assume that the scalar curvature of $g(t)$ is negative somewhere on each component and each time $t$.
To see what the limit can be as $t\rightarrow\infty$, using the above curvature estiamtes
Perelman adapted the arguments of R. Hamilton in \cite{hamilton4} to
the Ricci flow with surgery $g(t)$. For the readers' convenience, we outline the arguments following \cite{P2} (For more details, see
\cite{KL}, \cite{morgantian2}).

Recall that for the Ricci flow with surgery $g(t)$, one still has the evolution equation
on its scalar curvature $R$
\begin{equation}
\label{eq:scale-curv-evo}
\frac{d R}{dt} = \Delta  R + 2 |{\rm Ric}^o|^2 +\frac{2}{3}R^2,
\end{equation}
where ${\rm Ric}^o$ is the trace-free part of ${\rm Ric}$. Let $R_{min}(t)$ be the minimum of the scalar curvature $R(g(t))$
of $g(t)$. Then by the usual (scalar) maximum principle we have
\begin{equation}
\label{eq:r-min-evo}
\frac{d R_{min}}{dt} \ge  \frac{2}{3} R_{min}^2.
\end{equation}
It follows that
\begin{equation}
\label{eq:r-min-bound}
R_{min}(t) \ge  -
\frac{3}{2}\left(\frac{1}{t + 1/4}\right).
\end{equation}
Let $V$ be the volume, then
\begin{equation}
\label{eq:vol-evo}
\frac{d V}{dt} \le - R_{min} V.
\end{equation}
It follows that the function $V (t)(t+1/4)^{-\frac{3}{2}}$ is non-increasing in $t$.
denote by $\bar V$ its limit as $t\to \infty$.
Put $\hat{R} = R_{min} V^{\frac{2}{3}}$. It is a scale invariant and satisfies
\begin{equation}
\label{eq:hat-r-evo}
\frac{d\hat{R}}{dt}\ge \frac{2}{3} \hat{R} V^{-1} \int (R_{min} - R) dV.
\end{equation}
The right-handed side is nonnegative since $R_{min }\le 0$ on each component.
So $\hat{R}(t)$ has a unique limit, say $\bar R$, as $t\to \infty$.

If $\bar{ V} > 0$, then it follows from Equations~(\ref{eq:r-min-bound}) and~(\ref{eq:vol-evo}) that
$R_{min}(t)$ is asymptotic to $- 3/2t$, that is, $\bar{R} \bar {V}^{-\frac{2}{3 }} = -\frac{3}{2}$. Now Inequality~(\ref{eq:hat-r-evo}) implies that whenever we have a sequence of parabolic balls\footnote{$P(x_0,t_0,r_0,-\Delta t)$ means the product of the ball $B(x_0,t_0,r_0)$ of radius $r_0$ centered at $x_0$ in the metric $g(t_0)$ with the time interval $[t_0-\Delta t,t_0]$. Implicitly, we assume that this ball exists for all times $t\in [t_0-\Delta t,t_0]$.}
$P(x_a,  t_a, r\sqrt{t_a},-r^2t_a)$ for some fixed small $r >0$, such that the
scalings of $g(t)$ by factor $t_a^{-1}$ converge to some limit flow in the smooth topology,
defined in a certain parabolic ball $P(\bar{x}, 1, r,-r^2)$, then the scalar
curvature of this limit flow is independent of the space variables and equals $-\frac{3}{2t}$ at time $t\in [1 - {r}^2, 1]$. Moreover, applying
the strong maximum principle to Equation~(\ref{eq:scale-curv-evo}), one can easily show that
the sectional curvature of the limit at time t is constant and equals $-\frac{1}{4t}$.
If $\bar{V}=0$, then no such a sequence of parabolic balls can exist, so the above conclusion is automatically valid.
Furthermore, using curvature estimates from Section 6 in \cite{P2}, Perelman showed a more effective estimate on how close $g(t)$ is to a hyperbolic metric:
Given $w > 0$, $r > 0$, $\xi > 0$, one can find $T = T (w, r, \xi) < \infty$ such that if the ball $B(x,t,  r\sqrt{t})$ at some time $t \ge T$ has volume at least $w r^3t^{3/2}$ and sectional curvature at least $-r^{-2} t^{-1}$, then Ricci curvature satisfies
\begin{equation}\label{xieqn}
|2t {\rm Ric}(x,t) + g(x,t) |_{g(x,t)} < \xi.
\end{equation}
If one allows $T$ also to depend on $A\in (1,\infty)$, one can even have the above inequality for all points in the parabolic ball
$P(x, t,Ar\sqrt{t},-Ar^2t)$.

Now one can introduce the thick-thin decomposition of $M_t$ for $t$ sufficiently large.
Let $\rho(x, t)$ denote the radius $\rho$ of the ball $B(x, t,\rho)$ where $\inf {\rm Rm} =-\rho^{-2}$.
Fix $w>0$ a small positive constant.
Let $M_-(w, t)$ denote the thin part of $M_t$ consisting of all $ x \in M_t$ satisfying:
$$Vol(B(x_0,t,\rho(x,t))\le w \rho(x,t)^3.$$
Let $M_+(w, t)$ be its complement.

As Perelman pointed out in Section 7.3 of \cite{P2}, using the curvature pinching inequality along $g(t)$ and curvature estimates
from Section 6 in \cite{P2}, one can show: For any $w > 0$, there are $\bar {r}=\bar{r}(w)$ and $\bar{t} = \bar{t}(w)$ such that if
$t \ge \bar{t}$ and $\rho(x, t) < \bar{r} \sqrt{t}$, then
\begin{equation}\label{voleqn}
{\rm vol}(B(x,t,\rho(x, t))) < w \rho(x, t)^3.\end{equation}
It follows that for any $w>0$, if $x\in M_+(w,t)$ and $t$ is sufficiently large, then
\begin{equation}\label{rhoineq}
\rho(x,t)\ge \bar{r}\sqrt{t}.\end{equation}
Then for any given $w$ and $\xi$, for $t$ sufficiently large,
every point $(x,t)\in M_+(w,t)$ satisfies the estimates for curvature in Inequality~\ref{xieqn}.
Using Inequalties~(\ref{xieqn}) and~(\ref{rhoineq}), one can show that if $\{x_a\in M_+(w,t_a)\}$ is
a sequence of points with $t_a \to\infty$, then based manifolds $(M_{t_a}, t_a^{-1} g(t_a), x_a)$
converge, along a subsequence of $a\to\infty$, to a complete hyperbolic manifold of finite volume.
The limit may depend on choices of $(x_a, t_a)$. If one of these limits is closed, then $M_{t_a}$
is diffeomorphic to this limit when $a$ is sufficiently large and $t_a^{-1}g(t_a)$ converges to a hyperbolic metric as $a\to\infty$.
Then using the rigidity of hyperbolic metrics, one can further show that $(M_t,t^{-1}g(t))$ converges to the same hyperbolic manifold.
So one may assume that none of the limits is closed, let $ H_1$ is such a limit with the least number of cusps. Define
$H_1(w')$ to be the set of points in $H_1$ where the injectivity radius is not less than $w'$. Then by using an
argument in \cite{hamilton4} based on hyperbolic rigidity, Perelman showed that for all sufficiently small
$w$ and for $t$ large enough, $M_+(w/2, t)$ contains an almost isometric copy of $\varphi_t\colon H_1(w)\to M_+(w/2,t)$,
which in turn contains a component of $M_+(w, t)$; Moreover, this embedded
copy $\varphi_t(H_1(w))$ moves by isotopy as $t\to\infty$. If for some $w > 0$ the
complement $M_+(w, t) \backslash \varphi_t(H_1(w))$ is not empty for a sequence of $t\to\infty$, then one
can repeat the above arguments and get other complete hyperbolic manifolds $H_2,\cdots, H_k$. Since each hyperbolic
3-manifold with finite volume has a uniform lower bound, one can conclude that for $k$ sufficiently large and each sufficiently
small $w >0$,  the images of $\varphi_t(H_1(w)),\cdots, \varphi_t(H_k(w))$ in $M_t$ cover $M_+(w, t)$ for all sufficiently large t. Clearly, each boundary component
of $H_j(w)$ is a torus. In fact, those boundary tori of $\varphi_t(H_j(w))$ are incompressible in $M_t$. The proof is identical to that of R. Hamilton
in \cite{hamilton4} and is a minimal surface argument, using a result of Meeks
and Yau. An alternative proof for this incompressibility was given in \cite{KL} by using
volume comparison. This later proof is simpler and more elementary.

For $t$ sufficiently large, we define
$$\tilde M_-(w,t)=\overline{ M_t\backslash \varphi_t(H_1(w))\cup\cdots \cup \varphi_t(H_k(w))}.$$
It is a compact, codimension-$0$ submanifold of $M_-(w,t)$ with incompressible tori as boundary components. Furthermore, if $w$ is sufficiently small and $t$ is sufficiently large,
each boundary component of $(\tilde M_-(w,t),t^{-1}g(t))$ is convex,
has diameter at most $w$ and has a (topologically trivial) collar of length one, where the sectional curvatures
are between $-1/4 - \epsilon$ and $-1/4 + \epsilon$, where $\epsilon>0$ can be as small as one wants so long as $t$ is large enough.
For sufficiently small $w > 0$ and sufficiently large
$t$, $\tilde M_-(w, t)$ is actually diffeomorphic to a graph manifold, and consequently satisfies the Geometrization Conjecture\footnote{The definition of a graph manifold and a discussion of the fact that graph manifolds satisfy the Geometrization Conjecture are both given in the next section.} as does $M_t$. The fact that, for $t$ sufficiently large and $w$ sufficiently small, the $\tilde M_-(w,t)$ are graph manifolds is a consequence of the following theorem on collapsing with local lower curvature bound, applied to the $(\tilde M_-(w,t),t^{-1}g(t))$. The goal of this paper is to give a self-contained, complete proof of this theorem, which is a slight rewording\footnote{The difference is that we have not restricted $\rho(x)$ to be less than the diameter of the manifold. This is taken care of by the argument in Section~\ref{rhosect}.} of the result stated (without proof) as Theorem 7.4 of \cite{P3}.

\begin{thm}\label{7.4}
Suppose that $(M_n,g_n)$ is a sequence of compact, oriented
Riemannian $3$-manifolds, closed or with convex boundary, and that
$w_n$ is a sequence of positive numbers tending to zero as $n$ tends
to $\infty$. Assume that:
\begin{enumerate}
\item For each point $x\in M_n$ there exists a radius
$\rho=\rho_n(x)$  such that the ball $B_{g_n}(x,\rho)$ has volume at
most $w_n\rho^3$ and all the sectional curvatures of the restriction
of $g_n$ to this ball are all at least $-\rho^{-2}$;
\item Each component of the boundary of $M_n$ is an incompressible
torus of diameter at most $w_n$ and with a topologically trivial
collar containing the all points within distance $1$ of the boundary
on which the sectional curvatures are between $-5/16$ and $-3/16$;
\item For every $w'>0$ there exist $\bar r=\bar r(w')>0$ and
constants $K_m=K_m(w')<\infty$ for $m=0,1,2,\ldots,$ such that for all $n$
sufficiently large, and any $0<r\le \bar r$, if the ball
$B_{g_n}(x,r)$ has volume at least $w'r^3$ and sectional curvatures
at least $-r^{-2}$, then the curvature and its $m^{th}$-order
covariant derivatives at $x$, $m=1,2,\ldots,$ are bounded by
$K_0r^{-2}$ and $K_mr^{-m-2}$, respectively.
\end{enumerate}
Then for every $n$ sufficiently large, $M_n$ is a graph manifold.
If, in addition, $M_n$ is aspherical, there is a finite collection
${\mathcal T}_n$ of disjoint copies of $T^2\times I$ and twisted
$I$-bundles over the Klein bottle in $M_n$ such that each component
of  $M_n\setminus {\rm int}\,{\mathcal T}_n$  is a Seifert
fibered $3$-manifold with incompressible boundary. It follows that
each component of $M_n\setminus{\mathcal T}_n$ admits a locally
homogeneous geometry of finite volume.
\end{thm}

From the above discussions we see that, given any $w>0$ for all $t$ sufficiently large the manifolds $\tilde M_-(w,t)$ satisfy the hypotheses of Theorem~\ref{7.4}. Thus, applying this theorem, we see that for all
 $w$ sufficiently small and for all $t$ sufficiently large, the $\tilde M_-(w,t)$ are graph manifolds with incompressible boundary. Thus, fixing $w$ sufficiently small and fixing $t$ sufficiently large, we have the following: There is a decomposition $M_t=M_+(w,t)\cup \tilde M_-(w,t)$ where $M_+(w,t)$ and $\tilde M_-(w,t)$ are compact codimension-$0$ submanifolds meeting along their boundaries, these boundaries being incompressible tori. Furthermore, the interior of $M_+(w,t)$ is diffeomorphic to a complete hyperbolic manifold of finite volume and $\tilde M_-(w,t)$ is a graph manifold with incompressible boundary. Since each connected component of $M_t$ is either prime or diffeomorphic to the $3$-sphere, the same is true for $\tilde M_-(w,t)$. It is known that every connected, orientable, prime graph manifold with incompressible torus boundary either is diffeomorphic to one of $T^2\times I$ or a twisted $I$-bundle over the Klein bottle, or can be decomposed by a finite collection of disjoint incompressible tori and Klein bottles
into manifolds which admit complete, locally homogeneous Riemannian metrics.
 It follows immediately that the same statement is true for $M_t$. {\bf This completes the outline of the proof of the Geometrization Conjecture}.

At the beginning of this introduction we stated that in this paper we are extending and building upon the work of \cite{MT} in order to present a complete proof the Geometrization Conjecture.
Let us clarify that statement. What we present here in detail is a proof of the theorem about locally volume collapsed $3$-manifolds (Theorem 7.4 of \cite{P2}). In addition to that and the material in \cite{MT} one also needs the material in Section 6 and in the first part of section 7 of \cite{P2}.
In this paper we have been content to outline the main results from this material. It is our plan to combine this manuscript with an exposition of the remaining material from \cite{P2} into a sequel to \cite{MT}. Together these two books will give an entirely self-contained proof of the Geometrization Conjecture using Ricci flow and Alexandrov spaces.

There are other approaches to the Geometrization Conjecture which use variations of Theorem 7.4.
As was indicated above, if a $3$-manifold $M$ admits an impressible torus, then it falls into the class of $3$-manifolds for which the Geometrization Conjecture had been established
by Thurston himself. A detailed proof of the Geometrization Conjecture for those $3$-manifolds was given in \cite{Otal1} and \cite{Otal2}. In view of this, it suffices to prove Theorem 7.4 for closed
manifolds (again appealing to the Ricci flow results from \cite{P1} and the material in \cite{P2} preceding Theorem 7.4). A version of Theorem 7.4 for closed
$3$-manifolds has been proved in a series of papers of Shioya-Yamaguchi (\cite{SY2000}, \cite{SY}). They did not make use of Assumption 3 on bounds on derivatives of curvature.\footnote{Their proof was mostly for manifolds with curvature bounded from below, but the extension to the case of curvature locally bounded from below is not difficult as they point out in an appendix to their paper.} So their result can be applied to
$3$-manifolds which do not necessarily arise from Ricci flow. However, because they are not relying on Assumption 3, to prove their result,
Shioya-Yamaguchi need to use a hard stability theorem on Alexandrov spaces with generalized curvature bounded from below. This stability theorem
was due to Perelman and its proof was given in an unpublished manuscript in 1993. Recently, V. Kapovitch posted a preprint, \cite{Kapovitch}, which proposes
a more readable proof for this stability theorem of Perelman. Putting all these together, one has a
Perelman-Shioya-Yamaguchi-Kapovitch proof of Theorem 7.4 for closed manifolds without Assumption 3. As we have indicated, this proof requires a much deeper knowledge on Alexandrov spaces than the proof we present.
Our presentation of the collapsing space theory is motivated by, and to a large extent follows, the Shioya-Yamaguchi paper.

There is another approach to the proof of the Geometriation Conjecture due to  Besson etc. \cite{besson-etc} which avoids using Theorem 7.4. Again, this result relies on Thurston's theorem that $3$-manifolds with incompressible surfaces satisfy the Geometrization Conjecture, so that, as in the previous approach, one only needs to consider the case when the entire closed $3$-manifold is collapsed. Rather than appealing to the theory of Alexandrov spaces, this approach relies on other deep works on geometry and topology, e.g., results on the Gromov norms of $3$-manifolds.

Thus, all other approaches rely on Thurston's result on geometrization of $3$-manifolds containing an incompressible surface. The proof of this result uses completely different techniques than Ricci flow and is highly non-trivial in its own right, relying as it does on delicate and powerful results from hyperbolic geometry. For this reason, we feel that it is worthwhile to have a self-contained argument based on Ricci flow with surgery not making use of Thurston's results on hyperbolic manifolds and the hard stability theorem of Alexandrov spaces.

One can find other, related works on Ricci flow and the Geometrization Conjecture on www.claymath.org and in the long introduction of our previous book \cite{MT}. This paper will eventually be a part of our book project \cite{morgantian2} on the Geometrization Conjecture.

\section{The Collapsing Theorem: First remarks}

From now on in this article $3$-manifolds are implicitly assumed to be
orientable. Recall that a {\em Seifert fibration structure} on a compact
$3$-manifold  is a locally-free circle action on a $2$-sheeted
covering $\widetilde M$ of $M$ such that, denoting the covering
transformation on $\widetilde M$ by $\tau$, we have $\tau(\zeta\cdot
x)=\bar \zeta\cdot x$ for all $x\in \widetilde M$ and all $\zeta\in
S^1$. Seifert fibration structures are classified in terms of their
base orbifolds,  local Seifert invariants, and, when the base is
closed, an `Euler class,' see \cite{Seifert} or \cite{Orlik}.
 A compact $3$-manifold is
said to be {\em Seifert fibered} if it admits a Seifert fibration
structure. A compact, connected, Seifert fibered $3$-manifold is
either diffeomorphic to a solid torus or has boundary consisting of incompressible
tori. Furthermore, the interior of any compact, connected, Seifert
fibered $3$-manifold not diffeomorphic to a solid torus, to $T^2\times
I$, or to a twisted $I$-bundle over a Klein bottle admits a complete,
locally homogeneous Riemann metric of finite volume.
 A {\em graph manifold} is a compact $3$-manifold that is a connected
sum of manifolds each of which is either diffeomorphic to the solid
torus or can be cut apart along a finite collection of
incompressible tori into Seifert fibered $3$-manifolds.  Thus, a
graph manifold with incompressible boundary satisfies Thurston's
geometrization conjecture. One result we need is that the union
along boundary tori of the total space of a locally trivial
$S^1$-fibration and a collection of solid tori is a graph manifold,
see \cite{Wald}. Furthermore, if the fiber of the $S^1$-fibration is
homotopically essential in each of the solid tori, then the result
is a Seifert fibered $3$-manifold.

For the rest of this paper we fix the constants $\bar r(w')$ and the
$K_m(w'),\ m=0,1,\ldots$ as in the statement of Theorem~\ref{7.4}.

\subsection{Outline of the Proof}

According to Theorem 1.17 in Section 1.6 of \cite{apanasov}, a
closed, connected $3$-manifold admitting a flat metric is Seifert
fibered and hence is a graph manifold. If a closed, orientable
$3$-manifold has a metric of non-negative sectional curvature then by \cite{H}
it is diffeomorphic to one of the following:
\begin{enumerate}
\item  a spherical $3$-dimensional space-form,
\item a manifold double covered by $S^2\times S^1$, or
\item a flat $3$-manifold.
\end{enumerate}

Thus, without loss of generality we can make the following
assumption.

\medskip
\noindent {\bf Assumption 1. For each $n$, no connected, closed component of
$M_n$  admits a Riemann metric of non-negative sectional curvature.}

\medskip

The idea of the proof is to consider a sequence of balls of the form
$B_{g'_n(x_n)}(x_n,1)\subset M_n$,  $n =1,2,\ldots$, where by
definition $g'_n(x_n)= \rho_n(x_n)^{-2}g_n$. The hypotheses of the
theorem and Assumption 1 imply that each of these balls is
non-compact, but locally complete and of sectional curvature $\ge
-1$. The general theory of Alexandrov spaces  implies that given any
such sequence there is a subsequence that converges in the sense of
Gromov-Hausdorff to a ball of radius one in an Alexandrov space of
curvature $\ge -1$ and of dimension at least $1$ and at most $3$.
The volume condition implies that the limit is a $1$- or
$2$-dimensional. We then use results on the structure of Alexandrov
spaces of dimension $1$ and $2$ to deduce strong information about
the structure of these balls in $M_n$ for all $n$ sufficiently
large. These local structures can then be pieced together to form a
global result, proving the theorem stated above. In
Section~\ref{secbasics} we recall the general theory on Alexandrov
spaces,  and in Section~\ref{sec2D} we give a more detailed analysis
of $1$- and $2$-dimensional Alexandrov spaces that is necessary to
prove this result. In this introduction we assume that these basic
notions are understood and we formulate the precise structural
results that will be proved. In Section 4 we deduce the local results,
i.e., the possible structures of the balls $B_{g'_n(x)}(x,1)$, and in Section 5
we piece the local results together proving Theorem~\ref{1Dthm'} below.

\subsection{The collapsing theorem}

Let us now state the topological theorem that is established using the
compactness of Alexandrov spaces of curvature $\ge -1$ and the
volume collapsing hypotheses.

\begin{thm}\label{1Dthm'}
Suppose that we have a sequence of compact $3$-manifolds
 satisfying the hypothesis of Theorem~\ref{7.4} and satisfying
Assumption 1. Then, for every $n$ sufficiently large there are
compact, codimension-$0$, smooth submanifolds $V_{n,1}\subset M_n$
and $V_{n,2}\subset M_n$ with $\partial M_n\subset V_{n,1}$
satisfying the following.
\begin{enumerate} \item Each connected component
of $V_{n,1}$ is diffeomorphic to  one of the following:
\begin{enumerate}
\item[(a)]  a
$T^2$-bundle over $S^1$ or a union of two twisted $I$-bundles over
the Klein bottle along their common boundary;
\item[(b)]  $T^2\times I$ or $S^2\times I$, where $I$ is a closed
interval;
\item[(c)]  a compact $3$-ball or the complement of an open  $3$-ball in $\Ar
P^3$;
\item[(d)]  a twisted $I$-bundle over the Klein bottle;
or a solid torus.
\end{enumerate}
In particular, every boundary component of $V_{n,1}$ is either a
$2$-sphere or a $2$-torus.
\item[2.] $V_{n,2}\cap V_{n,1}=\partial V_{n,2}\cap \partial V_{n,1}$.
\item[3.] If $X_0$ is a $2$-torus component of $\partial V_{n,1}$, then $X_0\subset \partial
V_{n,2}$ if and only if $X_0$ is not a boundary component of $M_n$.
\item[4.] If $X_0$ is a $2$-sphere component of $\partial V_{n,1}$,
then $X_0\cap \partial V_{n,2}$ is diffeomorphic to an annulus.
\item[5.] $V_{n,2}$ is the total
space of a locally trivial $S^1$-bundle and $\partial V_{n,1}\cap
\partial V_{n,2}$ is saturated under this fibration.
\item[6.]  $M_n\setminus
{\rm int}\,\left(V_{n,2}\cup V_{n,1}\right)$ is a disjoint union of
solid tori  and
solid cylinders, i.e., copies of $D^2\times I$.
The boundary of each solid torus is a boundary component of $V_{n,2}$, and
 each solid cylinder $D^2\times I$ meets
$V_{n,1}$ exactly in $D^2\times\partial I$.
\end{enumerate}
\end{thm}

\subsection{Proof that Theorem~\protect{\ref{1Dthm'}} implies Theorem~\protect{\ref{7.4}}}

 In deducing Theorem~\ref{7.4} from Theorem~\ref{1Dthm'} we shall introduce several
topological simplifications in the decomposition given in the
conclusion of Theorem~\ref{1Dthm'}. While the decomposition given in
Theorem~\ref{1Dthm'} is deduced from the collapsing theory (in
particular, $V_{n,1}$ is the  part of $M_n$ close to a
$1$-dimensional space and  $V_{n,2}$ is the part close to a
$2$-dimensional space), as we modify the decomposition we  work
purely topologically and do not try to keep the connection with the
collapsing geometry.

\begin{claim}
It suffices to establish Theorem~\ref{7.4} under the assumption that
we have a decomposition as given in Theorem~\ref{1Dthm'} that
satisfies the following additional properties:
\begin{enumerate}
\item $V_{n,1}$ has no closed components.
\item Each $2$-sphere component of $\partial V_{n,1}$ bounds a
$3$-ball component of $V_{n,1}$.
\item Each $2$-torus component of
$\partial V_{n,1}$ that is compressible in $M_n$ bounds a solid
torus component of $V_{n,1}$.
\end{enumerate}
\end{claim}

\begin{proof}
By assumption, each closed component of $V_{n,1}$  can be decomposed
along a single incompressible $T^2$ into Seifert fibered manifolds,
and hence these satisfy the conclusion of Theorem~\ref{7.4}. Thus,
without loss of generality we can assume that there are no closed
components of $V_{n,1}$. In the similar
way, we can suppose that no component of $M_n$ is the union of two
solid tori, the union of a solid torus and a twisted $I$-bundle over
the Klein bottle, or the union of two twisted $I$-bundles over the
Klein bottle along a common boundary torus, since manifolds of the
first two types admit Riemannian metrics of non-negative sectional
curvature and those of the third decompose along an incompressible
torus into pieces that are Seifert fibered.

Let $C$ be a $2$-sphere component of $\partial V_{n,1}$. If $C$
bounds a component $\hat C$ of $V_{n,1}$ diffeomorphic to $\Ar
P^3\setminus B^3$, then we remove $\hat C$ from $M_n$ and from
$V_{n,1}$ and replace it in each with a $3$-ball in each. This has
the effect of removing a prime factor diffeomorphic to $\Ar P^3$
from $M_n$.
 This allows us to assume
that there are no components of $V_{n,1}$ diffeomorphic to $\Ar
P^3\setminus B^3$ and hence that the only components of $V_{n,1}$
with boundary $2$-spheres are either $3$-balls or diffeomorphic to $S^2\times I$.

Now let $C$ be a $2$-sphere component of $\partial V_{n,1}$, but not
bounding a  $3$-ball component of $V_{n,1}$. We cut $M_n$ open along
$C$ and cap off  the resulting two copies of $C$ with $3$-balls. We
add these balls to $V_{n,1}$ forming $V_{n,1}'$, and we leave
$V_{n,2}$ unchanged.
 The resulting subsets $V'_{n,1}$
and $V_{n,2}$ satisfy all the conclusions of Theorem~\ref{1Dthm'}.
If we can show that the result is a  graph manifold, then the same
is true for $M_n$. Induction then allows us to assume that every
$S^2$-boundary component of $V_{n,1}$ bounds a $3$-ball component of
$V_{n,1}$.

Next, we consider a $2$-torus component $T$ of $\partial V_{n,1}$
that is a compressible $2$-torus in $M_n$, but one that does not
bound a solid torus component of $V_{n,1}$. By Dehn's lemma there is
an embedded disk in $M_n$ meeting $T$ only along its boundary, that
intersection being homotopically non-trivial in $T$. First, suppose
that $T$ separates $M_n$. We write $M_n=P\cup _TN$. A thickening of
$T\cup D$ has a $2$-sphere boundary component $S$, which we can
suppose lies in $P$. Let $R$ be the region between $T$ and $S$; it
is diffeomorphic to the complement in a solid torus of a $3$-ball.
We form $A=P\cup_TF$ where $F$, is a solid torus, glued in such a
way that $R\cup_TF$ is diffeomorphic to a $3$-ball. We set
$V_{n,2}(A)=V_{n,2}\cap P$ and $V_{n,1}(A)=\left(V_{n,1}\cap
A\right)\cup F$. We also form $B=\widehat R\cup _TN$ where $\widehat
R$ is the solid torus obtained from $R$ by attaching a $3$-ball to
its $S^2$-boundary. We set $V_{n,2}(B)=V_{n,2}\cap N$ and
$V_{n,1}(B)=\left(V_{n,1}\cap N\right)\cup \widehat R$. It is easy
to see that $M_n$ is diffeomorphic to  $A\# B$ and that the given
decompositions of $A$ and $B$ satisfy all the conclusions of
 Theorem~\ref{1Dthm'} unless $T$ bounds a component of
$V_{n,1}$ that is a twisted $I$-bundle over the Klein bottle. In
this case, that component of $V_{n,1}$ is $N$ and $\widehat
R\cup_TN$ is Seifert fibered, whereas the conclusions of
Theorem~\ref{1Dthm'} hold for $A$. Thus, by a straightforward
induction argument, allows us to assume that every compressible
$2$-torus component of $\partial V_{n,1}$ that separates $M_n$
bounds a solid torus component of $V_{n,1}$. If $T$ does not
separate $M_n$ we cut $M_n$ open along $T$, add a solid torus $F$ as
before to the copy of $T$ bounding $R$ and add a copy of $\widehat
R$ to the other copy of $T$. Then $M_n$ is diffeomorphic to the
connected sum of the resulting manifold, $M'_n$, and $S^2\times
S^1$. Furthermore, adding $\widehat R\coprod F$ and to $V_{n,1}$ and
leaving $V_{n,2}$ unchanged produces a new decomposition satisfying
the hypotheses of Theorem~\ref{1Dthm'}. Again a simple induction
argument shows that repeated application of this operation removes
all non-separating compressing tori boundary components of $V_{n,1}$
without creating any new compressing tori boundary components that
do not bound solid torus components of $V_{n,1}$. This completes the
proof of the claim.
\end{proof}

With all these simplifying assumptions in place, we are ready to
complete the proof that Theorem~\ref{1Dthm'} implies
Theorem~\ref{7.4}. Let us consider the union, $X$, of the $D^2\times
I$ components of the closure of $M_n\setminus \left(V_{n,1}\cup
V_{n,2}\right)$ and the $3$-ball components of $V_{n,1}$. Since
every $2$-sphere boundary component of $V_{n,1}$ bounds a $3$-ball
component of $V_{n,1}$, each $D^2\times I$ meets the disjoint union
of the $3$-balls exactly in $D^2\times
\partial I$ and the boundary of each $3$-ball contains exactly
two disks in common with $\coprod D^2\times \partial I$. It then
follows from the fact that $M_n$ is orientable that $X$ is
diffeomorphic to a disjoint union of a finite number of solid tori.
Hence, the closure of $M_n\setminus V_{n,2}$ is a finite collection
of solid tori, components diffeomorphic to $T^2\times I$, and
components diffeomorphic to twisted $I$-bundles over the Klein
bottle. Furthermore, all boundary components of the $T^2\times I$
and twisted $I$-bundles over the Klein bottle are incompressible in
$M_n$. We remove from $M_n$ all components of $M_n\setminus V_{n,2}$
diffeomorphic to either $T^2\times I$ or to a twisted $I$-bundle
over the Klein bottle. The result, $W_n$, is a manifold that is the
union of $V_{n,2}$ and a collection of solid tori glued in along
boundary components. According to \cite{Wald}, since $V_{n,2}$ is an
$S^1$-fibration, $W_n$ is a graph manifold. Since the tori boundary
components that we cut along are incompressible, $\partial W_n$
consists of incompressible boundary tori. It follows that each prime
factor of $W_n$ has the property that removing a disjoint union of
submanifolds diffeomorphic to $T^2\times I$ and twisted $I$-bundles
over the Klein bottle results in an open manifold each component of
which admits complete homogeneous metrics of finite volume. The same
is then true of $M_n$.

If $M_n$ is aspherical, then it is not a non-trivial connected sum.
Removing from $M_n$ copies of $T^2\times I$ and twisted $I$-bundles
over the Klein bottle yields a manifold each component of which is
aspherical. But an aspherical graph manifold with incompressible
boundary has the property that taking out further copies of
$T^2\times I$ and twisted $I$-bundles over the Klein bottle results
in a manifold each component of which is Seifert fibered with
incompressible boundary. This completes the proof that
Theorem~\ref{1Dthm'} implies Theorem~\ref{7.4}. The rest of this
paper is devoted to the proof of Theorem~\ref{1Dthm'}.

\subsection{First reductions in the proof of Theorem~\protect{\ref{1Dthm'}}}

\subsubsection{A smooth limit result}

As we have already indicated, the entire argument revolves around
considering sequences $\{x_n\in M_n\}_{n=1}^\infty$,  rescaling the
metrics $g_n$, and, after passing to a subsequence, extracting a
limit (usually a Gromov-Hausdorff limit) of the metric unit balls in
the rescaled metrics. In general, a limit like this can be of
dimension $1$, $2$, or $3$ (although when we use $\rho^{-2}_n(x_n)$
to rescale the limit, the volume collapsing hypothesis implies that
the limit has dimension $1$ or $2$) and depending on which it is we
get a different structure for balls. The easiest case to treat is
when the limit is $3$-dimensional. As the next theorem shows,
because of the assumption on bounds on the curvature and its
derivatives in the statement of Theorem~\ref{7.4} such limits are
automatically smooth limits, rather than the more general
Gromov-Hausdorff limits that occur in the other two cases.

\begin{prop}\label{smlimits}
Let $(M_n,g_n)$ and $w_n$ be as in the statement of
Theorem~\ref{7.4}. Suppose that we have a sequence of points $x_n\in
M_n$  such that  $B_n=B_{g_n}(x_n,\rho_n(x_n))$ is disjoint from
$\partial M_n$ and a sequence of constants $\lambda_n$ with a
Gromov-Hausdorff limit of a subsequence of $(B_n,\lambda_ng_n,x_n)$,
which is a $3$-dimensional Alexandrov space. Then, passing to a
further subsequence, there is a smooth limit of the
$(B_n,\lambda_ng_n,x_n)$, which is a complete, non-compact manifold
of non-negative curvature.
\end{prop}

\begin{proof}
{\bf First step:}

\begin{claim}\label{vollimit} If $(B_n,\lambda_ng_n,x_n)$ converges to a 3-dimensional
Alexandrov space, then  there is a sequence of points $y_n\in M_n$
converging to a point $y$ in the limit and constants $r>0$ and
$\kappa>0$ such that for all $n$ sufficiently large ${\rm Vol}\,
B_{\lambda_ng_n}(y_n,r)\ge \kappa r^3$.
\end{claim}

\begin{proof} Fix $\delta>0$ sufficiently small.
Let $X$ be the limiting $3$-dimensional Alexandrov space. By Corollary 6.7 of
\cite{BGP}  the subset $R_\delta(X)$ consisting of
points with a $(3,\delta)$-strainer is dense. Choose $y\in
R_\delta(X)$
 and let $y_n\in M_n$ be a sequence converging to $y$.
 Then there is a $(3,\delta)$-strainer $\{a_1,b_1,a_2,b_2,a_3,b_3\}$
 at $y$. Let $d$ be the size of this strainer. Hence for all $n$
 sufficiently large, there is a $(3,\delta)$-strainer of size $d/2$
 at $y_n$ in $\lambda_nB_n$. According to Proposition~\ref{goodregion} this means that for
 some $r<< d/2$, but depending only on $d$, there is an almost bilipschitz homeomorphism from
 $B_{\lambda_ng_n}(y_n,r)$ to the ball of radius $r$ in Euclidean
 space, where the error estimate goes to zero with $\delta$.
 Hence, for any $\epsilon>0$ sufficiently small, the cardinality of a maximal
 $\epsilon$-net in $B_{\lambda_ng_n}(y_n,r)$ is at least $\alpha
 \epsilon^{-3}r^3$ for a universal constant $\alpha>0$. If we choose
 $\epsilon>0$ sufficiently small then the volume in $\lambda_ng_n$
 of any ball of radius $\epsilon/2$ centered at a point of
 $B_{\lambda_ng}(y_n,r)$ is at least $\omega_0(\epsilon/2)^3$ where
 $\omega_0$ is the volume of the unit ball in Euclidean $3$-space.
 Hence, ${\rm Vol}\,B_{\lambda_ng_n}(y_n,(r+\epsilon))\ge \alpha\omega_0 r^3/8$.
Taking the limit as $\epsilon\rightarrow 0$ gives the uniform lower
bound to the volume of the ball of radius $B_{\lambda_ng_n}(y_n,r)$.
\end{proof}

{\bf Second Step:} Suppose that $y_n\in B_n$ is as in the previous
claim. Then, we see that the $(B_n,\lambda_ng_n)$ are uniformly
volume non-collapsed at $y_n$. That is to say for some $r>0$ and
$w'>0$, for all $n$ the volume of $B_{\lambda_ng_n}(y_n,r)$ is at
least $w'r^3$.
 Since
there is $\rho(y_n)$ such that the  ball $B(y_n,\rho(y_n))$ has
volume is at most $w_n\rho(y_n)^3$ where $w_n\rightarrow 0$ as
$n\rightarrow \infty$, it follows from Bishop-Gromov volume
comparison that $\rho(y_n)\sqrt{\lambda_n}\mapsto\infty$ as $n$
tends to infinity. Hence, for any $A<\infty$, for all $n$
sufficiently large, we have $A<\rho(y_n)\sqrt{\lambda_n}$. Thus, by
our assumption, for all $n$ sufficiently large, the curvature of
$\lambda_ng_n$ on $B_{\lambda_ng_n}(y_n,A)\ge
-\lambda_n^{-1}\rho(y_n)^{-2}>-A^{-2}$. Also, for all $n$
sufficiently large, $A/\sqrt{\lambda_n}<\bar r(w')$. Hence, for any
$A<\infty$, for all $n$ sufficiently large, we have uniform bounds
on the curvature and all its derivatives in
$B_{\lambda_ng_n}(y_n,A)$. Since we also have uniform volume
non-collapsing at the base point, we can pass to a subsequence,
which has a smooth complete limit. Since
$\rho(y_n)^{-2}\lambda_n^{-1}$ tends to zero, the curvature of the
limit manifold is $\ge 0$.
\end{proof}

This result about the $3$-dimensional limits will be important as we
study the $1$- and $2$-dimensional limits.

\subsubsection{Adjusting $\rho_n$}\label{rhosect}

\begin{lem}\label{rholem} Let $M_n$, $w_n$ and $\rho_n$ satisfy the
hypotheses of Theorem~\ref{7.4} and suppose that the $M_n$ satisfy
Assumption 1. After passing to a subsequence, there are constants
$w'_n$ and functions
 $\rho_n\colon M_n\to (0,\infty)$ satisfying the hypothesis of
 Theorem~\ref{7.4} such that in addition the following hold:
\begin{enumerate}
\item For any connected component $M_n^0$ of $M_n$ and for any $x\in M_n^0$
we have $$\rho_n(x)\le \frac{1}{2}{\rm diam}\, M_n^0,$$  and
\item if $y\in B(x,\rho_n(x)/2)$ then $\rho_n(y)/2\le \rho_n(x)\le
2\rho_n(x)$.
\end{enumerate}
\end{lem}

\begin{proof}
Without loss of generality we can assume that $M_n$ is connected. If
$M_n$ is closed, then by assumption it is not the case that ${\rm
Rm}\ge 0$ on all of $M_n$. If $M_n$ has non-empty boundary, then
also by assumption ${\rm Rm}$ is not everywhere positive. Thus, for
each $x\in M_n$, there is a maximum $r=r_n(x)\ge \rho_n(x)$ such
that the ${\rm Rm}\ge -r^{-2}$ on $B(x,r)$. Furthermore,  by volume
comparison (the Bishop-Gromov theorem)
$${\rm vol}\, B(x,r)\le
\frac{V_{\rm hyp}(1)}{V_{\rm Eucl}(1)}w_nr^3,$$ where $V_{\rm
hyp}(1)$, resp. $V_{\rm Eucl}(1)$, is the volume of the unit ball in
hyperbolic, resp. Euclidean, $3$-space. Thus, at the expense of
changing the $w_n$ by a factor independent of $n$, we can define the
function $\rho_n$ so that $\rho_n(x)$ is this maximum $r(x)$. It
follows immediately that if $y\in B(x,\rho_n(x)/2)$ then
$$\frac{1}{2}\rho_n(x)\le \rho_n(y)\le 2\rho_n(x).$$
If, for all $n$, we have $\rho_n(x)\le {\rm diam}\, M_n$
for all $x\in M_n$, then we simply replace $\rho_n(x)$ by $\rho_n(x)/2$
and $w_n$ by $8w_n$ and we have established the claim in this case.

Now suppose (after passing to a subsequence) that for each $n$ there
is $x_n\in M_n$ with $\rho_n(x_n)>{\rm diam}\, M_n$. This implies
that ${\rm Rm}(x)\ge -({\rm diam}\, M_n)^{-2}$ for all $x\in M_n$
and hence that $\rho_n$ is a constant function; we denote its value
by $\rho_n$. Passing to a subsequence we can assume that ${\rm
vol}(M_n)/({\rm diam}\,M_n)^3$ tends to a limit (possibly $+\infty$)
as $n\rightarrow \infty$. First, we consider the case when this
limit is non-zero. The fact that the volume divided by the cube of
the diameter is bounded away from zero and the volume inequality
assumed in Theorem~\ref{7.4} imply that ${\rm diam}\,M_n/\rho_n$
tends to $0$ as $n\rightarrow\infty$. By the hypothesis about the
boundary of $M_n$, this implies that $M_n$ is closed. Rescaling
$M_n$ to make its diameter $1$ yields a manifold whose sectional
curvatures are bounded below by $-({\rm diam}\, M_n)^2/\rho^2_n$ and
whose volume is bounded away from zero. By
Proposition~\ref{smlimits} we see that passing to a subsequence
there is a smooth limit which has non-negative sectional curvature.
This is contrary to Assumption 1. Thus, we can suppose that ${\rm
vol}(M_n)/({\rm diam}\,M_n)^3$ tends to zero as $n$ goes to
infinity. We take $w_n'=8{\rm vol}(M_n)/({\rm diam}\,M_n)^3$ and we
take $\rho_n$ to be the constant ${\rm diam}\, M_n/2$.
\end{proof}

\noindent
{\bf Assumption 2 and notation: Now we fix the constants $w_n$ and
the functions $\rho_n\colon M_n\to (0,\infty)$ satisfying
Lemma~\ref{rholem}. For any $n$ and any $x\in M_n$ we denote by
$g'_n(x)$ the metric $\rho_n(x)^{-2}g_n$.
Thus, $B_{g_n}(x,\rho_n(x))= B_{g'_n(x)}(x,1)$.}

\section{Gromov-Hausdorff convergence of Alexandrov spaces}\label{secbasics}

\subsection{Basics about Alexandrov Spaces}

Let $X$ be  locally compact metric space. Then $X$ is a {\em local
length space} if $X$ is covered by open subsets $U_i$ such that for
each $i$ and for every two points $x,y$ in $U_i$ there is a closed
interval $I\subset \Ar$ and an isometric embedding $I\to X$ whose
endpoints are $x$ and $y$. In particular, the image of $I$ is
rectifiable and its length is the distance between its endpoints.
Fix a number $k$. We denote by $H_{k}$ the complete, simply
connected surface of constant curvature $k$.  Given three points
$p,q,r$ in a metric space and a real number $k$, then a {\em
comparison triangle} $\tilde p\tilde q\tilde r$ in $H_k$ is a
triangle  whose side lengths are equal to the distances between the
corresponding points of $X$, e.g., $d_{H_{k}}(\tilde p,\tilde
q)=d_X(p,q)$. Such a comparison triangle exists and is unique up to
isometries of $H_k$, provided only that if $k>0$, then the sum of
the three pair-wise distances is at most $2\pi/\sqrt{ k}$. We define
the {\em $k$-comparison angle} (or {\em the comparison angle} if $k$
is clear from the context) $\tilde \angle rpq$ to be the angle
$\angle \tilde r\tilde p\tilde q$ in $H_{k}$. By definition, a {\em
local Alexandrov space with curvature bounded below by $k$} is a
locally compact, local length space with the property that for every
$p\in X$ there is a neighborhood $U\subset X$ of $p$ such that for
any three points $q,r,s$ in $U$ the $k$-comparison angles satisfy
\begin{equation}\label{triangle}
\tilde\angle qpr+\tilde\angle rps+\tilde \angle spr\le 2\pi,
\end{equation}
 see
\cite{BGP}. In addition, if the dimension of $X$ is one and $k>0$,
then we require the diameter of $X$ to be at most $\pi/\sqrt{k}$. A {\em local
Alexandrov space} is a local Alexandrov space of curvature bounded
below by some $k$. If $X$ is a local  Alexandrov space with curvature $\ge
k$ and if $\lambda>0$ then  the metric space, denoted $\lambda X$,
obtained by multiplying the given metric by $\lambda^2$ is a local
Alexandrov space with curvature bounded below by $\lambda^{-2}k$.
An {\em Alexandrov space} is a complete metric space which is a length space
in the sense that any two points are jointed by an isometric embedding of an
arc, and that is also a local Alexandrov space. In such spaces Inequality~\ref{triangle}
holds globally, i.e., for all $4$-tuples of points in the space. (See
\S 3 of \cite{BGP}.)

Let $A, B$ be disjoint compact subsets of a local Alexandrov space $X$.
By a {\em geodesic} from $A$ to $B$ we mean an isometric embedding
of an interval into $X$ with one endpoint contained in $A$ and the
other contained in $B$ and such that the length of this interval is
equal to the distance from $A$ to $B$. By a {\em geodesic} we mean a
geodesic between the compact sets which are its endpoints. It is an
easy exercise to show that if $\gamma$ is a geodesic from $A$ to $B$
and $q$ is an interior point of this geodesic, then the sub-interval
of $\gamma$ from $A$ to $q$ is the unique geodesic from $A$ to $q$.

An elementary application of Toponogov \cite{Top}  theory shows the following.

\begin{exam}
Let $M$ be a  Riemannian manifold with locally convex boundary of with all sectional curvatures
$\ge k$. Then $M$ is a local Alexandrov space of curvature bounded below
by $k$. If $M$ is complete, then it is an Alexandrov space.
\end{exam}

Some of the basic properties of   Alexandrov spaces $X$ that follow
from this definition are:
\begin{prop}\label{ASbasics}
Let $X$ be a local Alexandrov space with curvature bounded
below by $k$ and let $U\subset X$ be an open subset with the property that
every pair of points in $U$ is connected by a geodesic and with the property
that Inequality~\ref{triangle}
holds for all quadruples of points in $U$. \begin{enumerate}
\item Let $p,q,r$ be three points in $U$
 and let $\tilde p\tilde q\tilde r$ be a comparison triangle in
$H_{k}$. Let $\gamma$ be a geodesic in $U$ with endpoints $q$ and
$r$ and let $x\in \gamma$ be at distance $s$ from $q$. Let $\tilde
x\in \tilde q\tilde r$ be a point on the corresponding geodesic in
$H_{k}$ at distance $s$ from $\tilde q$. Then $d(\tilde p,\tilde
x)\le d(p,x)$.
\item Let $\gamma$ and $\mu$ be geodesics in $U$ emanating from $p$. Set $a$, respectively $b$,
equal to the length of $\gamma$, respectively $\mu$. For
$0<s\le a$ and $0<t\le b$ let $\gamma(s)$ be the point on $\gamma$ at
distance $s$ from $p$ and let $\mu(t)$ be the point on $\mu$ at
distance $t$ from $p$. Define the function
$$f(s,t)=\tilde \angle \gamma(s)p\mu(t).$$
Then $f(s,t)$ is a monotone non-increasing function of either
variable $s$ and $t$ when the other is held fixed. In particular,
the limit as $s$ and $t$ both go to zero $f(s,t)$ exists and is
denoted $\angle qpr$.
\end{enumerate}
In particular, if $X$ is complete, then all these properties hold with $U=X$.

\end{prop}

These are all proved in \S 3 for \cite{BGP}.

An important result is the following splitting theorem for complete
Alexandrov spaces.

\begin{prop}\label{splitting}
Suppose that $X$ is an Alexandrov space of curvature $\ge 0$
and that $X$ contains a geodesic line $L\subset X$ (i.e., a copy of
$\Ar$ isometrically embedded in $X$) parameterized as $\zeta(s)$.
Then there is an Alexandrov space $Y$ and an isometry
$\Ar\times Y\to X$ so that $L$ is the image of $\Ar\times\{y\}$ for
some $y\in Y$. The parallel copies of $Y$ in the product are the
level sets of the function $f={\rm lim}_{s\rightarrow
-\infty}(d(\zeta(s),\cdot)-s$.
\end{prop}

This leads immediately by induction to:

\begin{cor}\label{prod}
Suppose that $X$ is an Alexandrov space of curvature $\ge 0$
containing an isometric copy of $\Ar^m$ for some $m>0$. Then there
is an Alexandrov space $Y$ and an isometric product decomposition
$X=\Ar^m\times Y$ with the property that the given copy of $\Ar^m$
is identified with $\Ar^m\times \{y\}$ for some $y\in Y$.
\end{cor}

The proofs given in \cite{Top} in the case of smooth manifolds work mutatis-mutandis
for Alexandrov spaces.

\begin{defn}
The {\em dimension} of a local Alexandrov space is its Hausdorff
dimension.
 Later, we shall see there is a much more
precise statement about dimension using strainers (or burst points).
\end{defn}

\subsection{Gromov-Hausdorff convergence}

Hausdorff convergence is defined for  metric spaces. The Hausdorff
distance between two metric spaces $X$ and $Y$ is less $\epsilon$ if
there is a metric on $X\coprod Y$ extending the given metrics on $X$
and $Y$ with $X$ contained in the $\epsilon$-neighborhood of $Y$ and
$Y$ contained in the $\epsilon$-neighborhood of $X$.

\begin{defn}
Let $X$ be a metric space. For any compact subset $A\subset X$ and
any $r>0$ we define the metric ball $B(A,r)=\{x\in X\left|\right.
d(x,A)<r\}$ and we denote the metric sphere $S(A,r)=\{x\in
X\left|\right.d(A,x)=r\}$. Notice that balls are open subsets.
\end{defn}

\begin{defn}
Let $(X_n,p_n)$ be a sequence of pointed, locally compact metric
spaces and that $(X,p)$ is a complete pointed metric space. We say
that the  $(X_n,p_n))$ converge in the {\em Gromov-Hausdorff sense}
to a ball $(X,p)$ if for every $R<\infty$ the balls $B(p_n,R)$
converge in the Hausdorff topology to $B(p,R)$.
\end{defn}

A crucial property (and indeed one of the main reasons for
introducing Alexandrov spaces) is the following compactness result
(see \S 8.5 of \cite{BGP}).

\begin{prop}
Fix an integer $N$ and a real number $k$. The collection of
complete, pointed Alexandrov spaces of dimension $\le N$ with
curvature bounded below by $k$ is sequentially  compact in the {\em
Gromov-Hausdorff sense}, meaning that if $(X_n,p_n)$ is a sequence
of such Alexandrov spaces, then there is a subsequence converging in
the Gromov-Hausdorff sense to a complete pointed Alexandrov space
$(X,p)$ which itself has dimension at most $N$ and curvature bounded
below by $k$.
\end{prop}

It is important to recognize that even if the $X_n$ all have the
same dimension $N$, it may well be the case that the limit $X$ has
lower dimension.

We shall also need a local version of convergence.

\begin{defn}
An {\em Alexandrov ball} is a local Alexandrov space $B=B(p,r)$ that
is a metric ball and has the property for any $y,z\in B$
with $d(y,z)+d(p,y)+d(p,z)<2r$ are joined by a geodesic in $B$.
\end{defn}

The local version that we need is the following.

\begin{prop}\label{ballscon}
Suppose that $B_n=B(p_n,r_n)$ are Alexandrov balls  of dimension $N$
and of curvature bounded below by $k$. Suppose that the
$r_n\rightarrow r$ as $n\rightarrow \infty$ with $0<r\le \infty$.
Then, after passing to a subsequence, there is an Alexandrov ball
$(B,p)$ of dimension at most $N$ and of radius $r$ such that for
every $r'<r$  the balls $ B(p_n,r')$ converge in the
Gromov-Hausdorff topology to $B(p,r')$.
\end{prop}

The essential point here is that for every $s<r$ for all $n$
sufficiently large there are balls of radius $(r-s)/3$ centered at
each point of $B(p_n,s)$ on which the Alexandrov property holds.
Given this uniformity, the argument in the complete case adapts to
establish this result.




If spaces $X_n$ converge in the Gromov-Hausdorff sense to $X$, then we say that a sequence $x_n\in
X_n$ {\em converges to} $x\in X$ (or that $x$ is the {\em limit} of the
$x_n$) if, setting $d_n$ equal to the distance between $x_n$ and $x$
in the given metric on $X_n\coprod X$, the $d_n(x_n,x)$ tend to zero as $n$
tends to infinity. Suppose that $(X_n,p_n)$ is a sequence of pointed
local Alexandrov spaces of dimension $\le N$ with curvature bounded
below by $k$ which are either complete or are  Alexandrov balls of
radius $r_n\rightarrow r>0$. If the $X_n$ are complete, then we set
$r=\infty$. Suppose that the $(X_n,p_n)$ are of dimension $\le N$
and that they converge in the Gromov-Hausdorff sense to $(X,p)$. Then all of the following are
immediate consequences of the definitions and usual compactness
arguments.
\begin{enumerate}
\item[(1)] Suppose that  $s<r$ and that for each $n$ we have a point $x_n\in
B(p_n,s)$. Then, after passing to a subsequence, there is a limit
point $x\in B(p,r)\subset X$ for the sequence.
\item[(2)] Suppose that we have a sequence of geodesics $\gamma_n$
in $X_n$ whose lengths converge to a  non-zero (but possibly
infinite) limit as $n$ tends to $\infty$ and suppose that the
initial points of the $\gamma_n$ converge to a point of $X$. Then
after passing to a further subsequence the geodesics converge,
uniformly on compact sets, to a geodesic in $X$.
\item[(3)]
 Suppose that we have
distinct points $q,r,s$ in $X$ and sequences $q_n,r_n,s_n$ in $X_n$
such that ${\rm lim}\,q_n=q$, ${\rm lim}\,r_n=r$ and ${\rm
lim}\,s_n=s$. Then
$${\rm lim}_{n\rightarrow\infty}\tilde \angle q_nr_ns_n=\tilde\angle qrs.$$
\item[(4)]  Suppose that for each $n$ we have geodesics $\gamma_n$
and $\mu_n$ in  $X_n$ both emanating from a point $p_n\in X_n$ with
${\rm lim}\,\gamma_n=\gamma$, ${\rm lim}\, \mu_n=\mu$. Then
$${\rm liminf}_{n\rightarrow\infty}\angle \gamma_n\mu_n\ge \angle \gamma\mu.$$
\end{enumerate}

\subsubsection{Limits that are products}

\begin{prop}\label{1limitprod} Fix $r>0$. Let $\lambda_n\rightarrow\infty$
and $\delta_n\rightarrow 0$ as $n\rightarrow\infty$. Suppose that
$(X_n,x_n)$ is a sequence of local Alexandrov spaces of dimension
$N$ and curvature $\ge -1$. Suppose that for each $n$ there are
compact sets $\{A^+_n,A^-_n\}$ with $d(x_n,A^+_n),d(x_n,A^-_n)\ge
r$. We suppose that for each point $z_n$ in the ball $B(x_n,r)$
there are geodesics from $z_n$ to $A^\pm$. We also suppose that the
comparison angle $\tilde\angle A^-_nx_nA^+_n>\pi-\delta_n$. Suppose
that the $(\lambda_nX_n,x_n) $ converge to $(X,x)$. Then there is a
based Alexandrov space $(Y,y)$ of dimension $\le N-1$ and isometry
$(X,x)\cong (Y,y)\times (\Ar,0)$ with the property that for any
sequence of points $z_n\in X_n$ converging to a point $z\in X$ and
geodesics $\gamma_n^\pm$ from $x_n$ to $A_n^\pm$, the $\gamma_n^\pm$
converge (uniformly on compact sets containing $x_n$) to the
geodesic rays from $z$ in the positive and negative $\Ar$-directions
in the product.
\end{prop}

\begin{proof}
Denote by $g_n$ the metrics on $X_n$; the rescaled metrics are
$\lambda_n^2g_n$. Let $\zeta_n^\pm$ be geodesics from $x_n$ to
$A^\pm_n$. Since the comparison angle $\tilde\angle A^-_nx_nA^+_n$
is greater that $\pi-\delta_n$, by monotonicity for any points
$u_n^\pm$ on $\zeta_n^\pm$ the comparison angle $\tilde\angle u_n^-x_nu_n^+$
is greater than
$\pi-\delta_n$. Hence, rescaling by the $\lambda_n$ and taking
limits we see that for points $u^\pm$ on the limiting geodesic rays
$\zeta^\pm$ the comparison angle $\tilde\angle u^-xu^+=0$, meaning that
$\zeta=\zeta^-\cup\zeta^+$ is a geodesic line. Since the $X_n$ have
curvature $\ge -1$ and the $\lambda_n\rightarrow \infty$, the limit
$X$ has curvature $\ge 0$. Hence, by Proposition~\ref{splitting} it
splits as a product $Y\times \Ar$ in such a way that $\zeta$ is the
factor in the $\Ar$-direction through the base point. Furthermore,
it also follows from this proposition that, letting $f_n$ be the
function
$d_{\lambda_n^2g_n}(A_n^-,\cdot)-d_{\lambda_n^2g_n}(A^{-1}_n,x_n)$,
the $f_n$ converge to a function $f\colon X\to \Ar$ whose level sets
are the parallel copies of $Y$ in the product structure. Let $z_n\in
B(x_n,r)$ be a sequence of points converging to $z\in X$, and let
$\gamma^\pm$ be a geodesic from $z_n$ to $A_n^\pm$. The directional
derivatives of the $f_n$ at the $x_n$ in the directions of the
$\gamma_n^+$ converge to $1$ as $n$ goes to infinity. Hence, the
$\gamma_n^+$ converge to rays in the positive $\Ar$-direction. It
follows that the $\gamma_n^-$ converge to rays in the negative
$\Ar$-direction.
\end{proof}

\subsection{Local geometry of Alexandrov spaces}

In this section we introduce the notion of the tangent cone for an
Alexandrov space. Using this we define directional derivatives for
Lipschitz functions and also the boundary of an Alexandrov space

\subsubsection{The space of directions and the tangent cone}
Let $X$ be a local Alexandrov space of dimension $\le N$, and let
$p\in X$ be a point.  Let $\Sigma_p'$ be the set of equivalences
classes of geodesics with $p$ as one endpoint where, by definition,
two geodesics are {\em equivalent} if their intersection contains a
geodesic of positive length emanating from $p$. The set of
equivalence classes has a metric: $d([\gamma_1],[\gamma_2])$ is
equal to the angle at $p$ of representatives $\gamma_1$ and
$\gamma_2$ of the equivalence classes, which is clearly independent
of the choice of representatives. The metric completion of
$\Sigma_p'$ is the {\em space of directions at $p$}, denoted
$\Sigma_p$. The dense subset $\Sigma_p'$ in $\Sigma_p$ is called the
{\em set of directions realizable by geodesics}.

\begin{prop}
The space $\Sigma_p$ is a compact Alexandrov space of dimension one
less than the dimension of $X$. If the dimension of $X$ is at least
three, then $\Sigma_p$ is an Alexandrov space of curvature $\ge 1$.
If $X$ has dimension $2$, then either $\Sigma_p$ is isometric to a
circle or an interval and has diameter at most $\pi$.
\end{prop}

For a proof of this result see Section 7 of \cite{BGP}.

Notice that the length of a metric circle is twice its diameter.

\begin{defn}
The tangent cone $T_pX$ is the cone over $\Sigma_p$ to the point
$p$. It is an Alexandrov space of curvature $\ge 0$ of the same
dimension as $X$.
\end{defn}

\begin{lem}\label{conelimit}
Let $\lambda_n$ be a sequence of positive real numbers tending
$+\infty$. Then the based local Alexandrov spaces $(\lambda_nX,p)$
converge in the Gromov-Hausdorff topology to
$(T_pX,\{p\})$.\end{lem}

For a proof, see Theorem 7.8.1 of \cite{BGP}.

\subsubsection{The boundary of an Alexandrov space}

\begin{defn}
The boundary of a local Alexandrov space is defined inductively on
dimension. Let $X$ be a one-dimensional local Alexandrov space. Then
it is either isometric to either an interval or a circle. Its
boundary as an Alexandrov space is its topological boundary.
 More
generally, we define the boundary of a higher dimensional local Alexandrov
space by induction. For $X$ an $n$-dimensional local Alexandrov space, we
define $\partial X$ to be the subset of $X$ consisting of points $p$
for which $\Sigma_p$ is an $(n-1)$-dimensional Alexandrov space with
non-empty boundary. Then $\partial X$ is a closed subset. Its
complement is denoted ${\rm int}\, X$.
\end{defn}

\subsection{Regular functions}

For the material in this section see Section 11 of \cite{BGP}.

\begin{defn}
Let $X$ be a local Alexandrov space. We say that a Lipschitz
function $f\colon X\to \Ar$ {\em has a directional derivative at
$q\in X$} if there is a continuous function $f_q'\colon \Sigma_q\to
\Ar$ such that for any geodesic $\gamma$ starting at $q$ and
parameterized as $\gamma(s)$ where $s$ is the distance from $q$, the
function $s\mapsto f(\gamma(s))$ has a one-sided derivative at
$s=0$, and this derivative is $f'_q([\gamma])$. We say that a
Lipschitz function $f$ is {\em regular} at $q$ if it has a
directional derivative at $q$ and if the directional derivative in
some direction at $q$ is positive. We say that a $1$-Lipschitz
function $f$ is {\em $a$-strongly regular} at $q$, if  $f$ has a
directional derivative and if there is a direction $\tau$ such that
$f_q'(\tau)>a$. The function $f$ is {\em strongly regular} if it is
$a$-strongly regular for every $a<1$.
\end{defn}

\begin{defn}
Let $A$ be a compact subset of a local Alexandrov space $X$ and let
$q$ be a point of $X\setminus A$.  We denote by  $A'\subset
\Sigma_q$ the set of directions at $q$ to all geodesics $\gamma$
from $A$ to $q$. It is a closed subset of $\Sigma_q$.
\end{defn}

\begin{lem}\label{cosformula}
Let $A\subset X$ be a compact subset, and set $f=d(A,\cdot)$. Then
at any point $q\in X\setminus A$ for which there is at least one
geodesic from $A$ to every point in a neighborhood of $q$, the
function $f$ is a $1$-Lipschitz function and $f$ has a directional
derivative at $q$. The derivative $f'_q$ is given by
$$f'_q(\xi)=-{\rm cos}(|\xi A'|),$$
where  $|\xi A'|$ denotes the distance in $\Sigma_q$ from $\xi$ to
the closed set $A'$ of all tangent directions at $p$ to geodesics
from $A$ to $p$. In particular, $f$ is $a$-strongly regular if and
only if there is a direction $\tau$ such that $d(A',\tau)>{\rm
cos}^{-1}(-a)$.
\end{lem}

\begin{proof}
See \S 11.4 of \cite{BGP}.
\end{proof}

The following is an elementary consequence of this lemma.

\begin{cor}\label{goodreg} Let $A$ a compact subset of $X$
and let $V\subset X\setminus A$ be an open subset. Suppose that there is a geodesic from
$A$ to each point of $V$. Then the subset of $V$ where $d(A,\cdot)$
is regular, respectively $a$-strongly regular, is an open set of
$V$ and includes any point $q\in V$ with the property
that there is a geodesic from $A$ to a point $w\not= q$ and passing
through $q$.
\end{cor}

Similarly, one shows:

\begin{cor}\label{regularopen}
Suppose that we have a sequence of pointed local Alexandrov spaces
$(X_n,p_n)$ with curvature $\ge k$ converging in the
Gromov-Hausdorff topology to a limit $(X,p)$. Suppose that there are
compact subsets $A_n\subset X_n$ converging  to $A\subset X$ and
open subsets $V_n$ converging to $V$. Suppose that there is a
geodesic from $A_n$ to each point of $V_n$, and suppose that $q\in
V$ and $q_n\in V_n$ is a sequence converging to $q$. Then if
$d(A,\cdot)$ is regular at $q$, resp. $a$-strongly regular at $q$,
then for all $n$ sufficiently large, $d(A_n,\cdot)$ is regular at
$q_n$, resp. $a$-strongly regular at $q_n$.
\end{cor}

\begin{lem}\label{levelsetlimit}
Suppose that $f\colon X\to \Ar$ is a Lipschitz function with
directional derivatives and that $q_n\in f^{-1}(f(q))$ is a sequence
converging to $q$. Let $\gamma_n$ be a geodesic from $q$ to $q_n$.
Suppose that the unit tangent vectors to the $\gamma_n$ at $q$
converge to a tangent direction $\tau$. Then $f_q'(\tau)=0$.
\end{lem}

\begin{proof}
This is elementary from the comparison results, see \S 11.3 of
\cite{BGP}.
\end{proof}

\subsubsection{Regular functions on smooth manifolds}

We shall need information about level sets of regular functions on
smooth manifolds.

\begin{lem}\label{smreg}
Suppose that $X$ is a locally complete Riemannian manifold and that
$f$ is the distance function from a compact set $A$ and that $f$ is
regular (in the Alexandrov sense) at $q_0\in X\setminus A$. Then there is a
neighborhood $U$ of $q_0$ and a smooth unit vector field $\tau$ on
$U$  with the property that $f_q'(\tau)>0$ for all $q\in U$.
Furthermore, there is an open interval $J$, an open subset $U'$ of
$\Ar^{n-1}$, and a bi-Lipschitz homeomorphism $U\cong U'\times J$
with the property that the level sets of $f|_U$ are identified with
the subsets $U'\times \{j\}$ for $j\in J$. In particular, the level
sets of $f$ are topologically locally flat, codimension-$1$
submanifolds near $q$.
\end{lem}

\begin{proof}
Consider the subset of the unit tangent bundle of $X$ consisting of
directions $\chi_q\in T_qX$ with the property that $f_q'(\chi_q)>0$
as $q$ varies over an open neighborhood $U$ of $q_0$. Arguments
similar to the above show that this is an open subset of $TX$. If we
take $U$ small enough, the fiber over every $q\in U$ is non-empty.
Hence, there is a smooth vector field $\tau$ defined
in a neighborhood $U$ of $q$ and $\alpha>0$ such that
$f_q'(\tau(q))\ge \alpha$ for all  $q\in U$. Now we integrate $\tau$
to define a smooth local coordinate system $(x^1,\ldots,x^n)$ near
$q_0$ such that $\tau=\partial/\partial x^1$. We replace $U$ be a
smaller open set which is the product of an open ball in
$(x^2,\ldots,x^n)$-space with an interval in the $x^1$-direction.
Since $f'(\partial/\partial x^1)>0$ everywhere, we see that the
level sets of $f$ meet each interval in the $x^1$-direction in at
most one point. That is to say, near $q_0$ these level sets are
given by the graphs of functions $x^1=\varphi(x^2,\ldots,x^n)$.
Elementary arguments show that the map $(x^1,\ldots,x^n)\mapsto
(f(x^1,\ldots,x^n),x^2,\ldots,x^n)$ is the required bi-Lipschitz
homeomorphism.
\end{proof}

We also need a fairly restricted version of an analogous result for
maps to the plane. The following is an elementary lemma.

\begin{lem}\label{2dregular}
Given $\epsilon'>0$, the following holds for all $\epsilon>0$
sufficiently small. Let $B(0,\epsilon^{-1})$ be the ball of
radius $\epsilon^{-1}$ in the Euclidean plane centered at the origin. We
denote by $(x,y)$ the Euclidean coordinates on this ball and by
$\theta$ the usual coordinate along the circle. Let $g$ be a Riemann
metric  on $U=B(0,\epsilon^{-1})\times S^1$ that is within
$\epsilon$ in the $C^N$-topology (where $N=[\epsilon^{-1}]$) of the
product of the usual Euclidean metric on $B(0,\epsilon^{-1})$ and
the Riemannian metric of length $1$ on the circle.  Suppose that
$F=(f_1,f_2)\colon U\to \Ar^2$ is a map with the property that $f_1$
and $f_2$ are $1$-Lipschitz with respect to $g$ with directional
derivatives at all points of $U$. Suppose further that the
directional derivatives of $f_i$ with respect to $g$ satisfy:
$$|f_1'(\partial_x)-1|<\epsilon$$
$$|f_2'(\partial_y)-1|<\epsilon$$
$${\rm max}(|f_1'(\pm\partial_y)|,|f_2'(\pm\partial_x)|,|f_1'(\pm\partial_\theta)|,
|f_2'(\pm\partial_\theta)|)<\epsilon.$$
Then any fiber $F^{-1}(p)$ that meets $B(0,\epsilon^{-1}/2)$ is a
circle that is $\epsilon'$-orthogonal to the family of horizontal
spaces $B(0,\epsilon^{-1})\times \{\theta\}$ in the sense that,
fixing $a\in F^{-1}(p)$, as $b\in F^{-1}(p)$ approaches $a$ the
angle (measured with respect to product metric) of the geodesic (in
the product metric) from $a$ to $b$ with the horizontal space
through $a$ is within $\epsilon'$ of $\pi/2$. Furthermore, any fiber
$F^{-1}(p)$ that meets $B(0,\epsilon^{-1}/2)$ intersects each
horizontal space $\{\theta\}\times B(0,\epsilon^{-1})$ in a single
point.
\end{lem}

\subsection{Almost manifold regions in Alexandrov spaces}

We introduce an open dense set of `good' points in an
$n$-dimensional Alexandrov space and show that these points have
neighborhoods that are $(1+\epsilon)$-Lipschitz equivalent to $\Ar
^n$.

\subsubsection{$(m,\delta)$-strainers}

Let $X$ be an $n$-dimensional local Alexandrov space of curvature
bounded below by $k$. Fix $x\in X$ and let $U\subset X$ be a neighborhood of $x$
in which the Alexandrov property holds. A {\em $(m,\delta)$-strainer}\footnote{This
notion is called $(m,\delta)$-burst points in \cite{BGP}.} at $x\in
X$ is a set of $2m$ points $\{a_1,b_1,\ldots,a_m,b_m\}$ in $U$ such that:
\begin{enumerate}
\item $\tilde \angle a_ipb_i>\pi-\delta$ for all $i=1,\ldots,m$,
\item $\tilde \angle a_ipa_j>\pi/2-\delta$ for all $i\not= j$,
\item $\tilde \angle b_ipb_j>\pi/2-\delta$ for all $i\not= j$.
\end{enumerate}

Notice that it follows from the defining property of an Alexandrov
space that $\tilde \angle a_ipa_j<\pi/2+2\delta$ for every
$i\not=j$, and similarly for the $b_i,b_j$. The {\em size} of an
$(m,\delta)$-strainer is defined to be the minimum of the $2m$
distances $\{d(p,a_i),d(p,b_i)\}_{i=1}^m$. For any local Alexandrov
space $X$ and any $m\ge 0$ and any $\delta>0$, the subset of points
$x\in X$ that have a $(m,\delta)$-strainer is an open subset.

\begin{defn}
Let $X$ be an $n$-dimensional local Alexandrov space. Then for any
$\delta>0$ denote by $R_\delta(X)$, the {\em $\delta$-regular set},
the subset of $X$ consisting of points with an
$(n,\delta)$-strainer. According to Section 6 of \cite{BGP}
$R_\delta(X)$ is an open dense subset of $X$.
\end{defn}

\begin{prop}\label{goodregion} Fix $n>0$.
Given $\epsilon>0$ there is $\delta>0$ such that the following
holds. Let $X$ be an $n$-dimensional local Alexandrov space and
suppose that $x\in R_\delta(X)$, and suppose that the Alexandrov
property holds on the ball of radius $r$ centered at $x$. Then there
is a neighborhood $U$ of $x$ and a $(1+\epsilon)$-bilipschitz
homeomorphism from $U$ to an open subset of $\Ar^n$. The open set
$U$ contains a metric ball about $x$ whose radius depends only on
the size of the $(n,\delta)$-strainer at $x$ contained in $B(x,r)$.
\end{prop}

\begin{proof}
Let $\{a_1,b_1,\ldots,a_n,b_n\}$ be a $\delta$-strainer of size $s$
at $x$ contained in $B(x,r)$. We define a map $X\to \Ar^n$ by $y\mapsto
(d(y,a_1),\ldots,d(y,a_n))$. According to Theorem 9.4 of \cite{BGP}
this map has the required properties in a  ball about $x$ whose
radius depends only on $s$ and $n$.
\end{proof}

Any compact, connected $1$-dimensional Alexandrov space is isometric
either to a closed interval or to a circle. In each case, the only
invariant up to isometry is the total length of the space.
Furthermore, for any $\delta>0$ the subset of $\delta$-regular
points is the interior of the interval or the entire circle. The
size of a $(1,\delta)$-strainer is the distance from the boundary in
the case of the interval, or one-quarter the length in the case of
the circle.

\subsection{A blow-up argument}

We need a special result about rescaling Alexandrov spaces so as to construct
higher dimensional limits. We need this result in order to handle sequences of
points $x\in M_n$ converging to a singular point of a
$2$-dimensional limit. The following is a reformulation in our context of Lemma 3.6 of \cite{SY2000}.

\begin{prop}\label{blowup}
For any $\delta>0$, the following holds for all $\mu>0$ sufficiently
small. Fix $r>0$. Suppose that $B_n=B(x_n,1))$ is a sequence of
Riemannian balls of radius $1$ with curvature $\ge -1$ in a complete
Riemannian manifolds with convex boundary. Suppose that the $B_n$
are non-compact and converge to an Alexandrov ball $(X,x)$. Suppose
that ${\rm dim}\,X$ is either $1$ or $2$. Suppose also that
rescaling $B=B(x,r)$ by $r^{-2}$ produces a ball  that is within
$\mu$ in the Gromov-Hausdorff distance of a flat cone. If $B$ is
$1$-dimensional, then we require that this flat cone is the cone on
a single point. If $B$ is $2$-dimensional we require that the flat
cone is the cone  either on a circle or on an interval and the
diameter of the base of the cone is  most $\pi-\delta$. Then, after passing to a subsequence, there are
points $\hat x_n\in X_n$ with $d(x_n,\hat x_n)\rightarrow 0$ as
$n\rightarrow \infty$ such that one of the following holds:
\begin{enumerate}
\item   $d(\hat x_n,\cdot)$ has no critical points in $B(\hat
x_n,r)\setminus\{\hat x_n\}$. In this case $B(\hat x_n,r')$ is
diffeomorphic to a ball for every $0<r'< r$.
\item There is a sequence of positive constants $\delta_n\rightarrow 0$ as $n\rightarrow
\infty$ such that:
\begin{enumerate}
\item Every critical point of $d(\hat x_n,\cdot)$ in $B(\hat x_n,r)$ is within
distance $\delta_n$ of $\hat x_n$, and
\item there is a critical
point $q_n$ for $d(\hat x_n,\cdot)$ at distance $\delta_n$ from
$\hat x_n$.
\end{enumerate}
In this case, passing to a subsequence there is a  limit of the
$\frac{1}{\delta_n}B(\hat x_n,r)$. This limit is a complete
Alexandrov space of curvature $\ge 0$ and of dimension at least one
more that the dimension of $X$.
\end{enumerate}
\end{prop}

\subsection{Gromov-Hausdorff limits of balls in the $M_n$}

Now we turn from generalities about Alexandrov spaces to special
properties of Gromov-Hausdorff limits of balls in the $M_n$.
Recall that we have a sequence of constants $w_n\rightarrow 0$ as
$n\rightarrow \infty$ and functions $\rho_n\colon M_n\to [0,\infty)$
with the property that $\rho_n(x)\le {\rm diam}(M^0_n)/2$ for every
$n$ and every $x$ in the connected component $M_n^0$ of $M_n$.
Thus, for every $n$ and every $x\in M_n$,
the ball $B_{g_n}(x,\rho_n(x))$ is non-compact. Since $M_n$ is
itself compact, it follows that for every $0<r<\rho_n(x)$, the ball
$B_{g_n}(x,r)$ has compact closure in $B(x,\rho_n(x))$. It then follows that
the $B_{g_n}(x,\rho_n(x))$ are non-compact Alexandrov balls.
Rescaling the metric by $\rho_n(x)^{-2}$, that is to say
replacing the metric $g_n$ on this ball by the metric
$g'_n(x)=\rho_n(x)^{-2}g_n$ we obtain non-compact Alexandrov
balls $B_{g'_n(x)}(x,1)$ of radius $1$ with the property that their
sectional curvatures are bounded below by $-1$, and their volumes are bounded above by $w_n$.
Since $w_n\rightarrow 0$ as $n\rightarrow \infty$,
the following is then immediate from Lemma~\ref{ballscon} and Proposition~\ref{goodregion}.

\begin{lem}
Let $x_n\in M_n$ be given for every $n\ge 1$. Then, after passing to
a subsequence, the $B_{g'_n(x)}(x_n,1)$ converge to  a non-compact
Alexandrov ball $B= B(\bar x,1)$ of curvature $\ge
-1$  and of dimension $1$ or $2$. The limiting ball contains points at every distance $<1$ from $\bar x$.
\end{lem}

This leads immediately to the following corollary.

\begin{cor}\label{epsilonn}
There is a decreasing sequence of constants $\epsilon_n>0$ tending
to zero as $n\rightarrow\infty$ such that for every $n$ and  for any
$x\in M_n$ there is a non-compact Alexandrov ball $B$ of radius $1$,
of curvature $\ge -1$, and of dimension $1$ or $2$, such that
$B_{g'_n(x)}(x,1)$ is within $\epsilon_n$ in the Gromov-Hausdorff
distance of $B$.
\end{cor}

\noindent
{\bf Assumption 3: We fix a sequence of $\epsilon_n\rightarrow 0$ as
in the corollary.}

\section{$2$-dimensional Alexandrov spaces}\label{sec2D}

 In order get enough information about the structure of
balls in the $M_n$ limiting (after rescaling) to a $2$-dimensional
Alexandrov space, we need fairly delicate information about
$2$-dimensional Alexandrov balls. We will cover these
$2$-dimensional spaces by four types of neighborhoods -- those near
flat balls in $\Ar^2$, those near flat circular cones on a circle of length $<2\pi$,
those near flat cones in $\Ar^2$ of angle $<\pi$,
and those near flat boundary points. Establishing this is the subject of this section.

\subsection{Basics}

\begin{claim}
A $2$-dimensional local Alexandrov space $X$ is a topological
$2$-manifold, possibly with boundary. The
topological boundary of $X$ is Alexandrov boundary $\partial X$.
\end{claim}

For a proof, see \S 12.9.3 of \cite{BGP}.

 Let $X$ be a $2$-dimensional local Alexandrov
space. We define the {\em cone angle} at any point $p\in X$ to be
the total length of  $\Sigma_p$. It follows from the Alexandrov
space axioms that if $p\in {\rm int}\, X$ then the cone angle at $p$
is at most $2\pi$  and the tangent cone is a flat circular cone of
this cone angle. If $p\in
\partial X$, then the cone angle  at $p$ is at most $\pi$, and the tangent cone is
a subcone of $\Ar^2$ of this cone angle.

\begin{lem}\label{decconeangle}
Suppose that $(X_n,x_n)$ is a sequence of $2$-dimensional
Alexandrov balls converging to a $2$-dimensional local Alexandrov
ball $(X,x)$. Suppose that sequence $y_n\in  X_n$ converges to
$y\in X$. Then: \begin{enumerate}
\item If $y_n\in
\partial X_n$ for all $n$ then $y\in
\partial X$.
\item
 Conversely, if $y\in \partial X$, then there is a sequence
$y_n\in \partial X_n$ converging to $y$. \end{enumerate}
\end{lem}

\begin{proof}
Let us suppose that $y_n\in \partial X_n$ for all $n$ and show that
$y\in \partial X$. Suppose to the contrary that $y\in {\rm int}\,
X$. Let $d_n$ be the Gromov-Hausdorff distance from $(X_n,y_n)$ to
$(X,y)$. Choose constants $\lambda_n\rightarrow\infty$ such that
$\lambda_nd_n\rightarrow 0$. Then the Gromov-Hausdorff distance from
$(\lambda_nX_n,y_n)$ to $(\lambda_nX,y)$ goes to zero and the
$(\lambda_nX,y)$ converge to the tangent cone to $X$ at $y$. This
allows us to assume that the $(X_n,y_n)$ converge to a flat circular
cone $C$ with cone point $y$.  For any compact subset $K\subset C\setminus \{y\}$, for
any sequence $z_n\in X_n$ converging to a point in $K$, and for any
$\delta>0$, for all $n$ sufficiently large, there is a $(2,\delta)$
strainer at $z_n$. Hence, $z_n\in {\rm int}\, X_n$. In particular,
the boundary component of $X_n$ containing $y_n$ converges to the
cone point $y$, and thus  the diameter of the boundary
component converges to $0$ as $n\rightarrow \infty$. We take a
piecewise geodesic approximation $\gamma_1,\ldots,\gamma_k$ to the
metric sphere $S(y,1)$, with the endpoints of $\gamma_i$ being $z_i$
and $z_{i+1}$ in $S(y,1)$. Let $z_{n,i}$ be points in $X_n$
converging to $z_i$ and let $\gamma_{n,i}$ be a geodesic with
endpoints $z_{n,i}$ and $z_{n,i+1}$. Then for all $n$ sufficiently
large, the union of the $\gamma_{n,i}$ is a simple closed curve in
${\rm int}\, X_n$ separating $X_n$ into two pieces, one of which,
$X_n'$, contains $y_n$ and hence the entire component of $\partial
X_n$ containing $y_n$. This component has convex boundary and hence
is  a compact $2$-dimensional Alexandrov space $X'_n$. It has at least two boundary components -- the
boundary component $\partial _0$ of $X_n$ containing $y_n$ and the component $\partial _1$
that is the union of the $\gamma_{n,i}$. There is an infinite cyclic
$\widetilde X'_n$ covering of $X'_n$ that unwraps both $\partial _0$ and $\partial _1$.
We give $\widetilde X'_n$ the length space
metric with the property that the projection mapping to $X'_n$ is a
local isometry. This metric makes the total space of the infinite cyclic covering
a complete local Alexandrov space and hence by Theorem 3.2 of\cite{BGP}, an Alexandrov space.
Since the diameter of the boundary component of
$X'_n$ containing $y_n$ is going to zero as $n\rightarrow \infty$,
the generator of the covering group of $\widetilde X'_n$ moves any
point in the preimage of $\partial_0$ a distance $d_n$ that
goes to zero as $n\rightarrow \infty$. Fix a
metric ball $B$ in the $\delta$-regular subset of the cone which is a
topological ball. For all $n$ sufficiently large, there are metric balls
$B_n\subset X'_n$ converging to $B$. For $n$ sufficiently large the $B_n$ are topological balls.
The area of $B_n$ is bounded away from $0$ as
$n\rightarrow \infty$ and the distance from $B_n$ to $y_n$ is also
bounded, say  by $\epsilon$, as $n\rightarrow\infty$. The preimage of
$B_n$ in the covering $\widetilde X'_n$ is a disjoint union of balls
freely permuted by the infinite cyclic covering group. Fix a point $\tilde y_n\in
\widetilde X'_n$ covering $y_n$. Given any $N$, there are distinct
pre-images of $y_n$ all within distance $Nd_n$ of $\tilde y_n$ and
hence $Nd_n$ distinct preimages of $B_n$ within $\epsilon+Nd_n$ of
$\tilde y_n$. Thus, for every $n$ there are $[\epsilon/d_n]$
distinct preimages of $B_n$ within $2\epsilon$ of $\tilde y_n$.
Since the areas of these preimages are bounded away from zero and
the $d_n\rightarrow 0$ as $n\rightarrow \infty$, and since the
$\widetilde X'_n$ all have curvature $\ge -1$, for $n$ sufficiently
large this violates the Bishop-Gromov inequality. This proves that
$y\in \partial X$.

Conversely, suppose that $y\in \partial X$ and let $d_n$ be the
distance from $y_n$ to $\partial X_n$. (We interpret $d_n=\infty$ if
$\partial X_n=\emptyset$.) We suppose that the $d_n$ are bounded
away from $0$. Then by Lemma~\ref{conelimit} there is $0<r<d_n$ for
every $n$ such that the distance function from $y$ is $a$-strongly
regular on $B(y,r)\setminus \{y\}$ and the corresponding level sets
are arcs. Consequently, for  all $n$ sufficiently large, the
distance function $f_n$ from $y_n$ is $a$-strongly regular on
$B(y_n,r)\setminus B(y_n,r/2)$. This implies that this difference is
a topological product of the level set $f_n^{-1}(3r/4)$ with an
interval. Furthermore, furthermore $f_n^{-1}(3r/4)$ is a
$1$-manifold with boundary in $\partial X_n$. Because $d_n>r$, this
level set is in fact unions of circles, and since $B(y,r)\setminus
B(y,r/2)$ is connected, it follows that $f_n^{-1}(3r/4)$ is a single
circle. But this is a contradiction. On the one hand, the
$f_n^{-1}(3r/4)$ are converging in the Gromov-Hausdorff sense to the
corresponding level set of the distance function from $y$ which is
an arc, and on the other hand, a  circle is not close in the
Gromov-Hausdorff distance to an arc. This contradiction proves that
$d_n\rightarrow 0$ as $n\rightarrow\infty$. Replacing $y_n$ by a
nearest point on the boundary of $\partial X_n$ gives us a sequence
of boundary points converging to $y$.
\end{proof}

\begin{cor}\label{disttobdry}
Suppose that $X_n$ are $2$-dimensional local Alexandrov spaces
converging to a $2$-dimensional local Alexandrov space $X$. Suppose
that $x_n\in X_n$ converge to $x\in X$. Let $d_n$ be the distance
from $x_n$ to $\partial X_n$ and let $d$ be the distance from $x$ to
$\partial X$. Then $d={\rm lim}_{n\rightarrow\infty}d_n$.
\end{cor}

\begin{defn}
 We say that a local Alexandrov space $(X,g)$ of dimension $1$ or $2$
 is a {\em standard ball} if:
\begin{enumerate}
\item there is $x\in X$ such that $X=B_g(x,R)$ for some $0<R\le 1$
is an Alexandrov ball of radius  $R$,
\item the curvature of $X$ is bounded below by $-1$, and
\item $X$ is non-compact, so that in particular, $X$ contains points at any distance $<R$
from $x$.
\end{enumerate}
\end{defn}

\subsection{The Interior}

We approximate interior points by cones, including flat cones.

\begin{defn}
Fix $\mu>0$. Let $(X,g)$ be an $2$-dimensional local Alexandrov
space. Then $X$ is {\em interior $\mu$-good at a point $y\in {\rm
int}\, X$ on scale $r$ and of angle $\alpha$} if $B_{r^{-2}g}(y,1)$
is a standard ball and is within $\mu$ in the Gromov-Hausdorff
distance of the circular cone of
 cone angle $\alpha$. We say that $X$ is {\em interior
$\mu$-flat at $y$ on scale $s$} if $B_{s^{-2}g}(y,1)$ is a standard
ball and is within $\mu$ in the Gromov-Hausdorff distance of the
unit ball in $\Ar^2$. We say that $X$ is {\em flat at $y$ on scale
$s$} if for every $\mu>0$ is $\mu$-flat at $y$ on scale $s$.
\end{defn}

The first thing to notice is that being interior flat at one scale
implies interior flatness at all smaller scales.

\begin{lem}\label{flatinherits}
Given $\mu>0$ there is $\nu>0$ such that the following holds. If
$X=B(x,1)$ is interior $\nu$-flat at $x$ on scale $1$, then for any
$0<s\le 1$, the ball  $X$ is interior $\mu$-flat at $x$ on
scale $s$.
\end{lem}

\begin{proof}
Suppose that the result does not hold for some $\mu$. Then there is
 a sequence $\nu_n$ tending to $0$ as $n\rightarrow \infty$ and
 be a $\nu_n$ counter-example, $X_n=B(x_n,1)$, at scale $s_n$ with
$0<s_n\le 1$. Passing to a subsequence we can suppose that the
$X_n$ converge to $X=B(\bar x,1)$ and that the $s_n$ converge to
$0\le s\le 1$. Since the $\nu_n\rightarrow 0$, the limit $X$ is isometric to the unit
ball in $\Ar ^2$. Clearly, then $X$ is $\mu$-flat at $\bar
y$ on any scale $s'$ with $0<s'\le 1$. Thus, if $s\not= 0$,
rescaling the $X_n$ by $1/s_n$ we obtain a contradiction.
Suppose now that $s=0$. Now we rescale $(X_n,x_n)$ by $1/s_n$,
and pass to a subsequence that has a limit. After rescaling there is
a $(2,\delta_n)$-strainer at $x_n$ of size $1/s_n$, where $\delta_n$
depends only on $\nu_n$ and goes to zero as it does. It follows
from Corollary~\ref{prod} that the resulting limit is $\Ar^2$. This
leads immediately to a contradiction as before.
\end{proof}

Next, we see that interior good at a point implies interior flat in
a nearby annular region where the constants depend on the area.

\begin{prop}\label{intcone}
Given $\mu>0$ and $a>0$ then for all $\mu'>0$ sufficiently small
there is $s_0>0$, depending only on $a$,  such that the following
holds. Suppose that $X=B(x,1)$ is interior $\mu'$-good at $x$ on
scale $1$, and suppose that the area of $X$ is $\ge a$. Then $X$ is
interior $\mu$-flat at every point $y\in B(x,7/8)\setminus B(x,1/8)$
on all scales $\le s_0$. Furthermore, for every $b\in (1/8,7/8)$ the
metric sphere $S(x,b)$ is a simple closed curve. (See Figure 1.)
\end{prop}

\begin{proof}
Given $\mu>0$, let $\nu>0$ be as in Lemma~\ref{flatinherits}.
Without loss of generality we can assume that $a\le \pi$. Then for
any $a\le a'\le \pi$ there is a unique flat circular cone $C$ such
that the unit ball about the cone point $p$ has area $a'$. Since
every point of $C\setminus \{p\}$ is interior flat, there is
$s_0>0$, depending only on $a$, be such that every $y$ in the
closure of if the annular sub-region $B(p,7/8)\setminus B(p,1/8)$ of
$C$ is interior flat on all scales $\le s_0$. The first statement
follows immediately by taking limits and using
Lemma~\ref{flatinherits}.

Since $d(p,\cdot)$ is strongly regular on the
annular region $B(p,7/8\setminus B(p,1/8)$ in $C$, for any
$\delta>0$, provided that $\mu'$ sufficiently small, the distance
$d(x,\cdot)$ is $(1-\delta)$-strongly regular on $B(x,7/8)\setminus
B(x,1/8)$. It then follows from  \S 11 of \cite{BGP} that $S(x,b)$
is a simple closed curve provided that $\mu'$ is sufficiently small.
\end{proof}

\begin{defn}
If $B(x,7/8)\setminus B(x,1/8)$ satisfies the statement in the above
proposition then we say that it is a {\em $(\mu,s_0)$-good annular
region.}
\end{defn}

\subsection{The boundary}

 We turn to the analogues for the boundary of interior flatness and interior
goodness.

\begin{defn}
Fix  $\mu>0$. Let $(X,g)$ be a $2$-dimensional Alexandrov space and
let $y\in \partial X$. We say that $X$ is {\em boundary $\mu$-good
on scale $r$ at  $y$ of angle $\alpha$} if the rescaled ball
$B_{r^{-2}g}(y,1)$ is a standard ball and is within $\mu$ in the
Gromov-Hausdorff distance of the unit ball with center the cone
point in a flat $2$-dimensional cone in $\Ar^2$ of cone angle
$\alpha$. We say that $X$ is {\em boundary $\mu$-good at $x$ on
scale $r$}, if it is boundary $\mu$-good on scale $r$ at $x$ of some
angle $\alpha$. We say that $X$ is {\em boundary $\mu$-flat at
$y\in\partial X$} on scale $s$ if $B_{s^{-2}g}(y,1)$ is within $\mu$
in the Gromov-Hausdorff distance to the unit ball centered at a
boundary point of $\Ar\times [0,\infty)$, and we say that $X$ is
{\em boundary flat on scale $s$ at $y$} if it is boundary $\mu$-flat
at $y$ on scale $s$ for every $\mu>0$.
\end{defn}

The next observation is that boundary flatness at a point at one
scale implies boundary flatness at that point at all smaller scales
and interior flatness at nearby points.

\begin{lem}\label{bdryflatinh}
Given $\mu>0$ for all $\nu'>0$ sufficiently small the following
holds. If $X=B(x,1)$ is boundary $\nu'$-flat at $x$ on scale $1$,
then for any $0<s\le 1$, the ball $X$ is boundary $\mu$-flat at $x$
on all scales $s$, $0<s\le 1$. Also, for any $z\in {\rm int}\, X\cap
B(x,1/2)$ the ball $X$ is interior $\mu$-flat at $z$ on scales $\le
d/2$ where $d$ is the distance from $z$ to $\partial X$.
\end{lem}

\begin{proof}
We begin the proof with a claim.

\begin{claim}\label{firstclaim}
Fix $\beta>0$. The following holds for all  $\nu'>0$ sufficiently
small. Suppose $X=B(x,1)$ is boundary $\nu'$-flat at $x$ on scale
$1$. Let $y\in \partial X\cap B(x,7/8)$ and let $z_-$ and $z_+$ be
points on $\partial X$ at distance  $1/8$ from $y$, lying on
opposite sides of $y$ in $\partial X$. Then the comparison angle
$\tilde\angle z_-yz_+$ is at least $\pi-\beta$. There is also a
point $w\in B(x,1)$ at distance $1/8$ from $y$ such that the
comparison angles $\tilde\angle z_-yw$ and $\tilde\angle wyz_+$ are
both at least $\pi/2-\beta$.
\end{claim}

\begin{proof}
These statements hold for $\beta=0$ for be the unit ball in $\Ar\times [0,\infty)$
centered at a boundary point $\bar x$. Thus, given $\beta>0$, the result follows for all
$\nu'>0$ sufficiently small by taking limits.
\end{proof}

Given this claim, the proof of the first statement of this result is
analogous to the proof of Lemma~\ref{flatinherits} using
Corollary~\ref{disttobdry}.
\end{proof}

Now we need the analogue of Proposition~\ref{intcone} producing good
annular regions.

\begin{prop}\label{bdrycone}
Given $\mu>0$ and $a>0$ there is $\mu''>0$ such that the following
holds for some positive constants $s_1$ and $s_2$ depending only on $a$. Suppose that
$X=B(x,1)$ is boundary $\mu''$-good at $x$ on scale $1$ and of area
$\ge a$. Then at every point $y'\in
\partial X\cap \left(B(x,7/8)\setminus B(x,1/8)\right)$ the ball $X$
is boundary $\mu$-flat on scale $s_1$, and for any $b\in [1/8,7/8]$
the metric sphere $S(x,b)$ is an arc with endpoints in $\partial X$.
Furthermore, for any $y\in B(x,7/8)\setminus B(x,1/8)$ one of the
following holds.
\begin{enumerate}
\item $X$ is interior $\mu$-flat at $y$ on all scales $\le s_2$.
\item There is $y'\in \partial X\cap
\left(B(x,7/8)\setminus B(x,1/8)\right)$ with $y\in
B(y',s_1/4)$, and the ball $X$ is interior $\mu$-flat at $y$ on all
scales $\le d/2$ where $d$ is the distance of $y$ to $\partial X$.
\end{enumerate}
(See Figure 2.)
\end{prop}

\begin{proof}
Fix $\mu>0$ and $a>0$. Without loss of generality we can assume that $a\le \pi$.
Let $\nu$ and $\nu'$ be  the constants
associated to $\mu$ by Lemma~\ref{flatinherits} and
Lemma~\ref{bdryflatinh}, respectively. First we show that for all
$\mu''>0$ sufficiently small there is an $s_1$ so that at every
point  $y'\in \partial X\cap \left(B(x,7/8)\setminus
B(x,1/8)\right)$ the ball $X$ is boundary $\nu'$-flat on all scales
$\le s_1$. Suppose not. Then there are a sequence of
$\mu''_k\rightarrow 0$ and examples $X_k=B(x_k,1)$ of area $\ge a$
that are boundary $\mu''_k$-good at $x$ on scale $1$, for which
there are points $y'_k\in  \partial X_k\cap
\left(B(x_k,7/8)\setminus B(x_k,1/8)\right)$  at which $X_k$ is not
boundary $\nu'$-flat on some scales $s'_k\rightarrow 0$. Passing to
a subsequence, we can suppose that the $X_k$ converge to a limit
$X=B(\bar x,1)$ which is a flat cone of area  $\ge a$. We can also
assume that the $y'_k$ converge to $y'$ in the closure of $B(\bar
x,7/8)\setminus B(\bar x,1/8)$, and by Lemma~\ref{decconeangle}, we
have $y'\in
\partial X$. Thus, $X$ is boundary flat at $y'$ on a scale $\hat s_1>0$ that depends only on $a$,
and hence for all $k$ sufficiently large, $X_k$ is boundary
$\nu'$-flat at $y'_k$ on scale $\hat s_1$. This is a contradiction.
Hence, there is $s_1>0$ such that assuming that $\mu''$ sufficiently
small, $X$ is boundary $\nu'$-flat at every $y\in
\partial X\cap\left(B(x,7/8)\setminus B(x,1/8)\right)$ on scale $s_1$. By
Lemma~\ref{bdryflatinh} this proves that for every such $y$, the
ball $X$ is boundary $\mu$-flat on all scales $\le s_1$ and for
every point  of $\partial X\cap\left(B(x_k,7/8)\setminus
B(x_k,1/8)\right)$ at distance $d\le s_1/4$ of $\partial X$ is
interior $\mu$-flat of all scales $\le d/2$.

Now, provided that $\mu''>0$ is sufficiently small, we establish the
existence of $s_2$ such that every point of $B(x,7/8)\setminus
B(x,1/8)$ that is not within $s_1/4$ of $\partial X$ is interior
$\nu$-flat on scale $s_2$. Suppose not. Then there are a sequence
$\mu''_k\rightarrow 0$ and examples $X_k=B(x_k,1)$ of area $\ge a$
that are boundary $\mu''_k$-good at $x_k$ on scale $1$ and points
$z_k\in B(x_k,7/8)\setminus B(x_k,1/8)$ at distance at least $s_1/4$
from $\partial X_k$ at which $X_k$ is not interior $\nu$-flat on
some scale $s'_k\rightarrow 0$. Passing to a subsequence we have a
limit $X$ which is a flat cone of area $\ge a$ and a limit $z\in X$
of the $z_k$. By Corollary~\ref{disttobdry} this is a point at
distance at least $s_1/4$ from $\partial X$. Thus, by
Lemma~\ref{bdryflatinh}, $X$ is interior flat at $z$ on scale
$s_1/8$. It follows that for all $k$ sufficiently large that $X_k$
is interior $\nu$-flat on scale $s_1/8$. This contradiction together
with Lemma~\ref{flatinherits} proves the existence of $s_2$ as
required.

Lastly, since the distance from the cone point in a flat cone is
strongly regular on the corresponding annular region, given any
$\delta>0$, then assuming that $\mu''$ is sufficiently small, the
distance from $x$ is $(1-\delta)$-regular on $B(x,7/8)\setminus
B(x,1/8)$. It follows from \S 11 of \cite{BGP} that, provided that
$\mu''$ is sufficiently small, for any $b\in [1/8,7/8]$ the metric
sphere $S(x,b)$ is an arc with endpoints in $\partial X$.
\end{proof}

\subsection{Geodesics approximating the boundary}

It turns out that near the flat part of the boundary it is better to
take neighborhoods centered around geodesics near the boundary
rather than balls centered around boundary points. Here, we follow
\cite{SY} closely.

\begin{defn}
Fix a $2$-dimensional local Alexandrov space $X$ with curvature $\ge
-1$. Suppose that $\gamma$ is an oriented
 geodesic in $X$ with initial point $e_-$ and final point $e_+$
 and of length $\ell=\ell(\gamma)$. We define
$$f_\gamma =\frac{1}{2}(d(e_-,\cdot)-d(e_+,\cdot))\ \ \ {\rm and} \ \ \
h_\gamma =d(\gamma,\cdot).$$ These  are $1$-Lipschitz
functions. Further, for any $\alpha>0$ we define
$$\nu_\alpha(\gamma)=f_\gamma^{-1}\left([-\ell/4,\ell/4]\right)\cap
h_\gamma^{-1}([0,\alpha\ell)),$$ and
$$\bar\nu_\alpha(\gamma)=f_\gamma^{-1}\left([-\ell/4,\ell/4]\right)\cap
h_\gamma^{-1}([0,\alpha\ell]).$$ We denote
$\nu_\xi^0(\gamma)=\nu_\xi(\gamma)\setminus
\bar\nu_{\xi^2}(\gamma)$. The {\em ends} of $\nu_\alpha(\gamma)$ are
their intersections with $f_\gamma^{-1}(\pm\ell/4)$, and the {\em
side} of $\bar \nu_\alpha(\gamma)$ is its intersection with
$h_\gamma^{-1}(\alpha\ell)$. For any $-\ell/4\le a<b\le \ell/4$ we
set
$$\nu_{\alpha,[a,b]}(\gamma)=f_\gamma^{-1}([a,b])\cap
h_\gamma^{-1}([0,\alpha\ell))$$ and we denote by
$\bar\nu_{\alpha,[a,b]}(\gamma)$ its closure. As before, the
boundary of $\bar\nu_{\alpha,[a,b]}(\gamma)$ is made up of the side,
given by $h_\gamma^{-1}(\alpha\ell)$, and the two ends, given by
$f_\gamma^{-1}(a)$ and $f_\gamma^{-1}(b)$. We say that $\alpha\ell$
is the {\em width} of the neighborhood and $(b-a)\ell$ is its {\em
length}. The level set $f_\gamma^{-1}(0)$ is the {\em center}
 of $\nu_\alpha(\gamma)$.
\end{defn}

\begin{lem}\label{goodrect}
The following hold for any $\xi>0$ sufficiently small and, given
$\xi$, for all $\mu>0$ sufficiently small and for any $s>0$. Suppose
that $X$ is a standard $2$-dimensional ball and  that $\gamma$ is a
geodesic of length between $s/10$ and $s$ with endpoints in
$\partial X$. Then if there is a point $x\in \partial X$ such that
$X$ is boundary $\mu$-flat at $x$ on all scales $\le 5s$ and if
$\gamma\subset B(x,s)$ then the following hold.
\begin{enumerate}
\item The arcs $\nu_\xi(\gamma)\cap \partial X$ and $\gamma\cap
\nu_\xi(\gamma)$ are within $\xi^2\ell(\gamma)$ of each other in
$X$.
\item For each
$y\in \bar\nu_\xi(\gamma)$ the comparison angle $\tilde \angle e_-ye_+$
is greater than $\pi-10\xi$.
\item For each point
$y\in \nu^0_\xi(\gamma)$ there is a geodesic $\zeta$ from $y$ to a
point $z$, with $d(y,z)>10\xi\ell$ such that for any $w\in \zeta$ at
distance at most $5\xi\ell$ from $y$ the comparison angle
$\tilde\angle\gamma wz\ge \pi-\xi$, and the comparison angles
$\tilde\angle e_\pm wz$ are greater than $\pi/2-10\xi$ and less than
$\pi/2+10\xi$.
\item For any level set $L$ of $f_\gamma$ in $\bar\nu_\xi(\gamma)$ and
for any $c\in [\xi^2,\xi]$ the distance from $L\cap \gamma$ to
$L\cap h_\gamma^{-1}(c\ell(\gamma))$ is less than
$(1+2\xi)c\ell(\gamma)$.
\end{enumerate}
(See Figure 3.)
\end{lem}

\begin{proof}
Direct computation shows that the result holds for $\xi>0$
sufficiently small, say $\xi\le \xi_0$ for some $\xi_0>0$, for $X$
being a ball of radius $1$ centered at a boundary point of
$\Ar\times [0,\infty)$ and $\gamma$ being a geodesic contained in
$\partial X$ of length between $1/50$ and $1/5$. Fix any $0<\xi\le
\xi_0$. Now let $B=B(x,5s)$ be a ball that is boundary $\mu$-flat at
$x$ on scale $5s$. Rescaling the metric by $(1/5s)^2$, we can suppose that
$s=1/5$ and that $B$ is a ball of radius $1$. The result is now
immediate by fixing $\xi$ and taking limits as $\mu$ tends to zero.
\end{proof}

The exact same proof as in the above lemma
shows the following result.

\begin{cor}\label{flatint}
The following
holds for all $\xi>0$ sufficiently small, and given $\xi$ for all
$\mu>0$ sufficiently small. Let $X$ be a standard $2$-dimensional
ball and suppose we have a geodesic $\gamma\subset X$, a constant  $s$,
 and a point $x\in X$ satisfying the hypotheses of the previous lemma.
 Suppose that $\zeta\subset B(x,5s)$
is a geodesic of length between $s/20$ and $2s$ with endpoints in
$\partial X$. Fix a direction along $\partial X\cap B(x,5s)$ and let
endpoints of $\gamma$ and $\zeta$, denoted $e_\pm(\gamma)$ and
$e_\pm(\zeta)$, be chosen so that in the given direction along
$\partial X$ we have $e_-(\gamma)<e_+(\gamma)$ and
 $e_-(\zeta)<e_+(\zeta)$. Suppose that there are constants $c,c'$
 with $\xi\le c,c'\le 1$. Then the following hold:
\begin{enumerate}
\item If $c\ell(\gamma)\le(0.9)c'\ell(\zeta)$, then the
side of $\bar\nu_{c'\xi}(\zeta)$
is at distance at least $(0.05)c\xi\ell(\zeta)$ from
$\bar\nu_{c\xi}(\gamma)$.
\item If $c\ell(\gamma)\ge (1.1)c'\ell(\zeta)$, then the
side of $\bar\nu_{c\xi}(\gamma)$ is at distance at least
$(0.05)c\xi\ell(\gamma)$ from $\bar\nu_{c'\xi}(\zeta)$.
\item For any point $y\in
\bar\nu_\xi(\gamma)\cap \bar\nu_\xi(\zeta)$, the comparison angles satisfy:
$$\tilde \angle e_-(\gamma)ye_+(\zeta)>\pi-10\xi$$ and
$$\tilde\angle e_-(\zeta)ye_+(\gamma)>\pi-10\xi.$$
\item Suppose that a level set $L\subset \bar\nu_\xi(\zeta)$
for $f_\zeta$ meets $\nu_\xi(\gamma)$. Then for any $y_1,y_2\in
L\cap \nu_\xi(\gamma)$ we have
$$|f_\gamma(y_1)-f_\gamma(y_2)|<\xi^2\ell(\gamma).$$
\end{enumerate}
(See Figure 3.)
\end{cor}

\begin{defn}
 We say that
a geodesic $\gamma\subset X$
 is a {\em $\xi$-approximation to
$\partial X$ on scale $s$ with $\mu$-control} if $\gamma$ is a
geodesic of length between $s/10$ and $s$ and if there is a point
$x\in \partial X$ at which $X$ is boundary $\mu$-flat on scales $\le
5s$ with $\gamma\subset B(x,s)$ such that:
\begin{enumerate}
\item the conclusions of Lemma~\ref{goodrect} hold for $\xi$, and
\item the conclusions of Corollary~\ref{flatint} hold for
$\xi$ and any
geodesic $\zeta\subset B(x,5s)$  of length between $s/20$ and $2s$.
\end{enumerate}
The point $x$ is a {\em $\mu$-control point} for $\gamma$.
\end{defn}

We must also compare flat regions near the boundary with balls
around boundary points.

\begin{cor}\label{flat/ball}
The  following holds for all
$\xi>0$ sufficiently small and, given $\xi$, for all $\mu>0$
sufficiently small. Suppose that $X$ is a standard $2$-dimensional
ball that is boundary $\mu$-good at $y$ on scale $r$. Suppose that
$\gamma\subset X$ is  a $\xi$-approximation to $\partial X$ on scale
$s\le r/20$ that is contained in $B(y,7r/8)\setminus B(y,r/8)$. We
orient $\gamma$ so that $e_+$ is separated along $\partial X\cap
B(y,r)$ by $e_-$ from $y$. Then for any $z\in \nu_\xi(\gamma)$ the
comparison angle $\tilde\angle yze_+$ is greater than $(.99)\pi$.
Furthermore, for any level set $L$ of $d(y,\cdot)$ that meets
$\bar\nu_{\xi,[-(.24)\ell,(.24)\ell]}(\gamma)$ intersection
$L\cap\bar \nu_\xi(\gamma)$ is an interval with on endpoint in
$\partial X$ and the other in the side of $\bar \nu_\xi(\gamma)$.
Furthermore, the function $f_\gamma$ varies by at most
$\xi^2\ell(\gamma)$ on this intersection. (See Figure 4.)
\end{cor}

\begin{proof}  First we show the following.

\begin{claim}\label{omega}
There is $\omega$ depending only on $\mu$ and going to zero with
$\mu$ such that the following holds. For any point $x\in \partial
X\cap B(y,15r/16)$ with $d(x,y)\ge r/16$ and any $w\in
\partial X\cap B(y,r)$ that is separated from $y$ by $x$ along
$\partial X\cap B(y,r)$, the comparison angle $\tilde\angle
wxy>\pi-\omega$.
\end{claim}

\begin{proof}
First consider the point $w'\in \partial X\cap B(y,r)$ that is
separated from $y$ by $x$ and lies at distance $\ge r/16$ from $x$.
Since $B(y,r)$ is boundary $\mu$-good at $z$ on scale $r$, it
follows that there is $\omega$ depending only on $\mu$ and going to
zero as $\mu\rightarrow 0$ so that the comparison angle $\tilde
\angle w'xy>\pi-\omega$. Next, we claim that for any $w\in \partial
X\cap B(y,r)$ lying in the interior of the sub-interval of $\partial
X$ with endpoints $x$ and $w'$ we have $\tilde\angle wxy\ge
\tilde\angle w'xy$. To see this  consider a geodesics $\alpha$ from
$y$ to $w$ and $\beta$ from $x$ to $w'$. These must cross at exactly
one point, say $u$. Also, let $\alpha'$ be a geodesic from $y$ to
$x$ and $\beta'$ be a geodesic from $x$ to $w$. Then by the
monotonicity of comparison angles $\tilde\angle uxy\ge \tilde \angle
w'xy$. Also by monotonicity we have $\tilde\angle wxu+\tilde\angle
uxy$ is at most the sum of the angle between $\alpha'$ and $\beta$
at $x$ and the angle between $\beta$ and $\beta'$ at $x$. But since
$x\in\partial X$, the sum of the angles between these geodesics is
at most $\pi$. It follows that $\tilde\angle yxu+\tilde\angle uxw\le
\pi$. Using this we see that $\tilde\angle yxw\ge \tilde \angle
yxu+\tilde\angle uxw$. Thus, $\tilde \angle yxw\ge \tilde\angle
yxu\ge \tilde\angle yxw',$ completing the proof of the claim.
\end{proof}

Let $z\in \nu_\xi(\gamma)$. Then $d(z,e_+)>\ell(\gamma)/4$ and
$d(z,y)\ge r/4$. According to Lemma~\ref{goodrect}, $z$ is within
$(1+2\xi)\xi\ell(\gamma)$ of a point $q\in \gamma\cap
\nu_\xi(\gamma)$ and every point of $\gamma\cap\nu_\xi(\gamma)$ is
within $\xi^2\ell(\gamma)$ of a point in the boundary arc with
endpoints $e_+$ and $e_-$. It follows that $z$ is within
$(1+3\xi)\xi\ell(\gamma)$ of a point $w$ in the arc with endpoints
$e_+$ and $e_-$ of $\partial B(y,r)$. Since $\ell(\gamma)<r$, it
follows from the law of cosines that there is $\omega'$ that goes to
zero as $\omega$ and $\xi$ both go to zero, so that the comparison
angle $\tilde\angle e_+zy$ is at least $\pi-\omega'$. This proves
that given any $\omega>0$, the angle
$\tilde\angle yze_+<\pi-\omega$ provided that $\xi$ and $\mu$ are
sufficiently small.

With this estimate, the result now follows immediately by rescaling by to make
 $\ell(\gamma)=1$ taking limits.

\end{proof}

\begin{cor}\label{levelset}
Let $\xi>0$ be given; fix $\mu>0$ such that Lemma~\ref{goodrect} and
Corollaries~\ref{flatint} and~\ref{flat/ball} hold. Fix $a>0$ and
let $\mu''>0$ be such that Proposition~\ref{bdrycone} for these
values of $\mu,a$ and $\mu''$. Then there are $s'_1$ and $s'_2$ such
that the following holds. Suppose that $X=B(x,1)$ is boundary
$\mu''$-good on scale $1$ at $x$ and of area $\ge a$. Then for any
$b\in(1/8,7/8)$, the metric sphere $S(x,b)$ is an interval with
endpoints $y_1,y_2$ in $\partial X$. The space $X$ is boundary
$\mu$-flat at $y_i$ on all scales $\le s_1'$ and there are geodesics
$\gamma_i$ that are $\xi$-approximations to $\partial X$ on scale
$s_1'$ such that $S(x,b)$ is contained in the union of: (i) the open
subset of points at which $X$ is interior $\mu$-flat on all scales
$\le s'_2$ and
$\nu_{\xi^2,[-\ell(\gamma_1)/8,\ell(\gamma_1)/8]}(\gamma_1)\cup
\nu_{\xi^2,[-\ell(\gamma_2)/8,\ell(\gamma_2)/8]}(\gamma_2)$.
\end{cor}

\begin{proof}
Fixing $\xi,\mu,a,\mu''$ so that Proposition~\ref{bdrycone} holds,
we take $s_1$ and $s_2$ as in Proposition~\ref{bdrycone}. Then by \S
11 of \cite{BGP} for any $b\in (1/8,7/8)$ the metric sphere $S(x,b)$
is an interval with endpoints in $\partial X$. We denote these
endpoints by $y_1$ and $y_2$. Let $\gamma_i$ be a geodesic of length
$s_1$ with endpoints in $\partial X$ equidistance from $y_i$.
\begin{claim}
Every point of $S(x,b)\setminus \left( \nu_{\xi^2}(\gamma_1)\cup
\nu_{\xi^2}(\gamma_2)\right)$ is distance at least $\xi^2s_1$ from
$\partial X$.
\end{claim}

\begin{proof}
Let $z\in S(x,b)\setminus \left( \nu_{\xi^2}(\gamma_1)\cup
\nu_{\xi^2}(\gamma_2)\right)$ and suppose that $\zeta$ is a geodesic
from $z$ to a point $w\in \partial X$ with the length of $\zeta$
being less than $\xi^2s_1$. Then we have $|d(w,x)-b|<\xi^2s_1$. By
the first statement in Corollary~\ref{flat/ball}, provided that
$\xi$ is sufficiently small, the value of the function $d(x,\cdot)$
at the endpoints of $\nu_{\xi^2}(\gamma_i)$ differs by at least
$3s_1/8$. Since the endpoints of $\gamma_i$ are equidistance from
$y_i$, it follows that the distance of each of these the endpoints
from $x$ differs from $b$ by at least $s_1/8$. Since by
Claim~\ref{omega}, the distance from $x$ is monotone along $\partial
X\cap \left(B(x,7/8)\setminus B(x,1/8)\right)$, it follows that on
$\partial X\setminus\left(\nu_{\xi^2}(\gamma_1)\cup
\nu_{\xi^2}(\gamma_2)\right)$ the distance from $x$ takes no value
within $s_1/8$ of $b$. Hence, the point $w\in
\nu_{\xi^2}(\gamma_1)\cup \nu_{\xi^2}(\gamma_2)$. By symmetry we can
suppose that $w\in \nu_{\xi^2}(\gamma_1)$. This argument shows that
the distance from $w$ to the endpoints of $\partial X\cap
\nu_{\xi^2}(\gamma_1)$ is at least $s_1/16$. Since $w\in
\nu_{\xi^2}(\gamma_1)$, the geodesic $\zeta$ must cross the frontier
of $\nu_{\xi^2}(\gamma_1)$.  But the distance from $x$ of any point
in the ends of $\nu_{\xi^2}(\gamma_1)$ is within $200\xi^2 s_1$ of
the distance from $x$ of the corresponding  endpoint of $\partial
X\cap \nu_{\xi^2}(\gamma_1)$. This means that the distance from $w$
to the ends of $\nu_{\xi^2}(\gamma_1)$ is greater than $s_1/32$.
Provided that $\xi$ is sufficiently small, this means that $\zeta$
cannot cross the ends of $\nu_{\xi^2}(\gamma_1)$, and hence it must
cross the side of this neighborhood. It then also crosses the
geodesic $\gamma_1$ and hence its length is at least $\xi^2s_1$,
which is a contradiction. This contradiction proves the claim.
\end{proof}

Now every point $z\in
S(x,b)\setminus\left(\nu_{\xi^2}(\gamma_1)\cup\nu_{\xi^2}(\gamma_2)\right)$
is either not within $s_1/4$ of $\partial X$, in which case by
Lemma~\ref{bdrycone} the space $X$ is interior $\mu$-flat on all
scales $\le s_2$ at $z$, or $z$ is within $s_1/4$ of $\partial X$,
and $X$ is interior $\mu$-flat at $z$ on all scales less than or
equal to $d(z,\partial X)/2$, and by the previous claim
$d(z,\partial X)\ge \xi^2s_1$. Taking $s_1'=s_1$ and $s_2'={\rm
min}(s_2,\xi^2s_1/2)$ then gives the result.
\end{proof}

\begin{defn}
If $B(x,7/8)\setminus B(x,1/8)$ satisfies the conclusions of
Corollary~\ref{levelset}, then we say that it is a {\em
$(\xi,\mu,s'_1,s'_2)$-good strip.}
\end{defn}

\subsection{The covering}

Now we assemble all the local results to give a covering of a
standard $2$-dimensional ball.

\begin{thm}\label{summary}
The following holds for all $\xi>0$  sufficiently small, for all
$\mu>0$ less than a positive constant\footnote{This means that
$\mu_1(\xi)$ is a constant that depends on $\xi$. We shall use a
similar convention throughout.}  $\mu_1(\xi)$, and for all $a>0$.
There are positive constants $\delta$ and $r_0$,  depending on
$\xi,\mu$, and $a$,  with $r_0\le 10^{-3}$, such that for all
$s_1>0$ less than a positive constant $\tilde s_1(\xi,\mu,a,r_0)\le
r_0/20$ and for all $s_2>0$ less than a positive constant $\tilde
s_2(\xi,\mu,a,s_1)\le s_1$, for any standard $2$-dimensional ball
$X=B(x,1)$ of area $\ge a$, the ball $B(x,1/2)$ is covered by open
subsets of the following four types:
\begin{enumerate}
\item The open subset of points $y\in B(x,1/2)$ with the property
that $X$ is interior $\mu$-flat on all scales
$\le s_2$ at $y$.
\item The open subset of points that lie in the center of
neighborhoods $\nu_{\xi^2}(\gamma)$ where $\gamma$ is a $\xi$-approximation to
$\partial X$ at scale $s_1$ with $\mu$-control.
\item Open balls  $B(y,r'/4)$  for some $r_0\le r'=r'(y)\le 10^{-3}$ such
that the open ball $B(y,r')$ is interior $\mu$-good at $y$ at scale
$r'$ and angle $\le 2\pi-\delta$ and
$$\frac{1}{r'}B(y,7r'/8)\setminus \frac{1}{r'}B(y,r'/8)$$
is a $(\mu,s_2)$-good annular region.
\item Open balls $B(y,r'/4)$, for some $r_0\le r'=r'(y)\le 10^{-3}$,
such that $X$ is boundary $\mu$-good at $y$ on scale $r'$ and of
angle $\le \pi-\delta$ and also $\frac{1}{r'}B(y,7r'/8)\setminus
\frac{1}{r'}B(y,r'/8)$ is a $(\xi,\mu,s_1,s_2)$-good strip.
\end{enumerate}
\end{thm}

\begin{proof}
Fix $\xi>0$ sufficiently small and let $\mu>0$ be as in
Corollary~\ref{levelset} for this value of $\xi$. Let $\nu$ and
$\nu'$ be the values associated to $\mu$ by Lemma~\ref{flatinherits}
and Lemma~\ref{bdryflatinh}, respectively. Fix $a>0$. Suppose that
the result is not true. Then take decreasing sequences
$\{\delta_k,r_k,s_{1,k},s_{2,k}\}$ tending to zero, and suppose that
there are examples $X_k=B(x_k,1)$, each satisfying the hypotheses
for the given values of $\xi,\mu$, and $a$ and points $y_k\in
B(x_k,1/2)$ not contained in an open set of any of the four types
for any fixed values of the parameters
$\delta_k,r_{0,k},s_{1,k},s_{2,k}$. Passing to a subsequence we can
suppose that the $X_k$ converge to $X=B(\bar x,1)$ and the $y_k$
converge to $y$ in the closure of $B(\bar x,1/2)$. Since the $X_k$
all have area at least $a$, $X$ is a standard $2$-dimensional ball
of area at least $a$. By Lemma~\ref{conelimit} given any
$\epsilon>0$, there is a constant $r=r(y)>0$ such that
$\frac{1}{r}B(y,r)$ is within $\epsilon$ in the Gromov-Hausdorff
distance to the unit ball in a cone $C$. We can assume that $r\le
10^{-3}$. Because the area of $X$ is $\ge a$, there is $a'>0$,
depending only on $a$, such that the area of $C$ is at least $a'$.
The cone is either a flat cone in $\Ar^2$, if $y\in
\partial X$, or a circular cone, if $y\in {\rm int}\, X$.
We consider the various cases.

\noindent {\bf Case 1: The cone is a circular cone of angle $2\pi$.}
Provided that we have chosen  $\epsilon<\nu$, the ball $X$ is
interior $\nu$-flat at $y$ on scale $s$ and hence for all $k$
sufficiently large $X_k$ is interior $\nu$-flat at $y_k$ on scale
$s$ and hence by Lemma~\ref{flatinherits} for all $k$ sufficiently
large $X_k$ is $\mu$-flat at $y_k$ on all scales $\le s$. This is a
contradiction since $s_{2,k}<s$ for all $k$ sufficiently large.

\noindent {\bf Case 2: The cone is a circular cone of angle
$2\pi-\delta$ for some $\delta>0$.} In this case provided that we
have chosen $\epsilon$ less than the $\mu'$ determined by $\mu$ and
$a'$ by Proposition~\ref{intcone}, it follows  that
$\frac{1}{r}B(y,r)$ is interior $\mu'$-good at $y$ on scale $1$. The
same is true for the $\frac{1}{r}B(y_k,r)$ for all $k$ sufficiently
large. Hence, by Proposition~\ref{intcone} there is $s_0=s_0(a)$
such that for all $k$ sufficiently large, the region
$\frac{1}{r}B(y_k,7r/8)\setminus \frac{1}{r}B(y_k,r/8)$ is a
$(\mu,s_0(a))$-good annular region. This is a contradiction since
$r_{0,k}<r\le 10^{-3}$ and $\delta_k<\delta$ and $s_{s,k}<s_0$ for
all $k$ sufficiently large.

\noindent {\bf Case 3: The cone is a flat cone in $\Ar^2$ of cone
angle $\pi$.} This means that the limit is boundary flat at $y$ on
scale $r$. Provided that we choose $\epsilon$ less than the $\nu'$
determined by $\mu$  by Lemma~\ref{bdryflatinh}, it follows that for
all $k$ sufficiently large, $X_k$ is boundary $\nu'$-flat at $y_k$
on scale $r$. Hence, by Lemma~\ref{bdryflatinh} for all $k$
sufficiently large $X_k$ is boundary $\mu$-flat at $y_k$ on all
scales $\le r$. Choose a geodesic $\gamma_k$ of length $r/20$ with
endpoints on $\partial X_k$ equidistant from $y_k$. According to
Lemma~\ref{goodrect} this geodesic is a $\xi$-approximation to
$\partial X_k$. Clearly, the limit of the $\gamma_k$ is a geodesic
contained in $\partial X$ whose midpoint is $y$. In
particular, $\nu_{\xi^2}(\gamma)$ contains an entire neighborhood of
$y$, and hence for all $k$ sufficiently large the neighborhood
$\nu_{\xi^2}(\gamma_k)$ contains $y_k$.
Since the endpoints of $\gamma_k$ are equidistance for $y_k$, the point
$y_k$ lies in the center of this neighborhood.
This is a contradiction
since for all $k$ sufficiently large $s_{1,k}<r/20$.

\noindent{\bf Case 4: The cone is a flat cone in $\Ar^2$ of cone
angle $\pi-\delta$ for some $\delta>0$.} Provided that we have
chosen $\epsilon$ less than the constant $\mu''$ associated to $\mu$
and $a'$ by Proposition~\ref{bdrycone}, it follows that $X$ is
boundary good at $y$ of angle $\pi-\delta$ on scale $r$. Hence, the
same is true for $X_k$ at $y_k$ for all $k$ sufficiently large. It
then follows from Corollary~\ref{levelset} that there are constants
$s_1',s_2'>0$ such that $\frac{1}{r}B(y_k,7r/8)\setminus
\frac{1}{r}B(y_k,r/8)$ is a $(\xi,\mu,s_1',s_2')$-good product
region. This is a contradiction since $\delta_k<\delta$,
$s_{1,k}<s'_1$ and $s_{2,k}<s'_2$ for all $k$ sufficiently large.

In all cases we have arrived at a contradiction, proving the result.
\end{proof}

\subsection{Transition between the $2$- and  $1$-dimensional part}

We need to understand the passage between the $1$- and
$2$-dimensional parts of the $M_n$. A $1$-dimensional standard ball
$B(x,1)$ is either an open interval of length $2$ or is a half-open
interval of length $\ell$ with $1\le \ell\le 2$.

\begin{lem}\label{smarea1D}
The following hold for all $\beta>0$ and for all $a>0$ less than a
positive constant $a_2(\beta)$. Let $B(x,1)$ be a standard
$2$-dimensional ball  and suppose that there is a point $y\in
B(x,24/25)$ with the area of $B(y,1/100)$ being at most $a$. Then
$B(x,1)$ is within $\beta$ in the Gromov-Hausdorff distance of a
standard $1$-dimensional ball $J$.
\end{lem}

\begin{proof}
Fixing $\beta>0$ suppose that the result does not hold for any
$a>0$. Take a sequence $B(x_k,1)$ of standard $2$-dimensional balls
of area $a_k\rightarrow 0$ as $k\rightarrow \infty$ and points
$y_k\in B(x_k,24/25)$ for which the result does not hold. Passing to
a subsequence we can extract a limit $\bar B$ with the $y_k$
converging to $\bar y\in \bar B$. Because of the area condition, the
neighborhood $B(\bar y,1/100)$ must be  $1$-dimensional, and hence
$\bar B$ is a standard $1$-dimensional ball.
\end{proof}

\section{$3$-dimensional analogues}\label{4}

Now we discuss the structure of balls in a $3$-dimensional
Riemannian manifold that are close to the various $1$- and
$2$-dimensional balls that we have been discussing. Since we shall
need the results for $3$-dimensional balls near $2$-dimensional balls in
our study of $3$-dimensional balls near $1$-dimensional balls, we
start with the $2$-dimensional case. Recall that for any
$x\in M_n$ we denote by $g'_n(x)$ the rescaled metric
$\rho_n(x)^{-2}g_n$. Throughout this section we consider pairs
consisting of a point $x_n\in M_n$ and a constant $\lambda\ge
\rho_n(x_n)^{-2}$ with the property that $B_{\lambda g_n}(x_n,1)$ is
disjoint from $\partial M_n$. Of course, since $\lambda\ge
\rho_n(x_n)^{-2}$ the sectional curvatures of these balls is bounded
below by $-1$. Any time we refer to such $B_{\lambda g_n}(x_n,1)$, unless
we explicitly state the contrary,
 we are implicitly assuming that it is disjoint from the boundary.

\subsection{Generic interior points of $2$-dimensional Alexandrov spaces}

 We begin with a description of the $3$-dimensional part of a
Riemannian $3$-manifold $M$ that is near the `generic' part of a
$2$-dimensional Alexandrov space.

\begin{lem}\label{generic2D}
For all $\epsilon>0$ and any $0<s_2\le 1/2$, the following holds for
all $\mu>0$ less than a positive constant $\mu_2(\epsilon)$
 and for all $\hat\epsilon>0$ less than a positive constant
$\hat\epsilon_0(s_2,\epsilon)$.  Suppose  that the ball $B_{\lambda
g_n}(x_n,1)$ is within $\hat\epsilon$ of a $2$-dimensional
Alexandrov ball $B=B(\bar x,1)$ which is interior $\mu$-flat at
$\bar x$ on scale $s_2$. Then
 there exist an
embedding $\varphi\colon S^1\times B(0,\epsilon^{-1})\to M_n$ with
$x\in \varphi(S^1\times \{0\})$ and a constant
$\lambda'>\epsilon^{-2}\lambda$ such that the metric
$\varphi^*\lambda'g$ is within $\epsilon$ in the
$C^{[1/\epsilon]}$-topology to the product of the metric of length $1$ on the
circle and the restriction of the standard Euclidean
metric to $B(0,\epsilon^{-1})$. Lastly, there is a universal constant $C>0$
such that, measured using
$\lambda g_n$, the lengths of the circles in this product structure
are less than $C\hat\epsilon$.
\end{lem}

\begin{proof}
Let us first show that it suffices to prove the first conclusion for
$s_2=1/2$. For, suppose that we have established the conclusion in
this special case with constants $\mu_1(\epsilon)$ and
$\hat\epsilon_0(\epsilon,1/2)$, and let us consider the statement
for another value $0<s_2\le 1/2$. Then suppose for some
$\mu<\mu_1(\epsilon)$ and
$\hat\epsilon<2s_2\hat\epsilon_0(\epsilon,1/2)$ that we have
$B_{\lambda g_n}(x_n,1)$ within $\hat\epsilon$ of $B(\bar x,1)$, the
latter being interior $\mu$-flat $\bar x$ on scale $s_2$. Then
$B_{(\lambda/4s^2_2)g_n}(x_n,1)$ is within $\hat\epsilon/(2s_2)$ of
$\frac{1}{2s_2}B(\bar x,2s_2)$, and the latter is $\mu$-flat at $\bar x$ on
scale $1/2$. Since, by construction,
$\hat\epsilon/(2s_2)<\hat\epsilon_0(\epsilon,1/2)$, the result for
$s_2=1/2$ implies the existence of a constant $\lambda'$ as
required. (Of course, $\lambda'>(1/2s_2)\lambda$ since
$B_{(1/4s^2_2)\lambda g_n}(x_n,1)$ is close to a $2$-dimensional ball
whereas $B_{\lambda' g_n}(x_n,1)$ has $3$-dimensional volume bounded
away from zero.)

Thus, we can now assume that $s_2=1/2$. Fix $\epsilon>0$ and suppose
that the first conclusion does not hold for this constant. Then
there is are sequences $\mu_k\rightarrow 0$ and  $\hat\epsilon_k>0$
both tending to zero as $k\rightarrow \infty$ such that for each $k$
there is an index $n(k)$ and a point $x_{n(k)}\in M_{n(k)}$ and
constants $\lambda_{k}\ge \rho_{n(k)}(x_{n(k)})^{-2}$ so that the
ball $B_{\lambda_{k}g_{n(k)}}(x_{n(k)},1)$ is within
$\hat\epsilon_{k}$ of a $2$-dimensional Alexandrov ball
$B_{k}=B(\bar x_{k},1)$ that is interior $\mu_k$-flat at $\bar
x_{k}$ on scale $1/2$, yet no $x_{n(k)}$  satisfies the first
conclusion of the lemma. Fix a sequence of positive constants
$c_k\rightarrow \infty$ such that $c^2_k\hat\epsilon_{k}\rightarrow 0$
as $k\rightarrow \infty$. The $2$-dimensional balls $c_kB(\bar
x_k,1)$  are $\mu_k$-flat at $\bar x_k$ on scale $c_k/2$, and since
$c_k\rightarrow \infty$, these rescaled balls converge in the
Gromov-Hausdorff topology to $\Ar^2$. On the other hand, since
$c_k\hat\epsilon_k\rightarrow 0$, the balls
$c_kB_{\lambda_kg_{n(k)}}(x_{n(k)},1)$ also converge to the same
limit, $\Ar^2$, in the Gromov-Hausdorff topology.
 In particular, the volume of
the unit balls centered at $x_{n(k)}$ in
$c_kB_{\lambda_{k}g_{n(k)}}(x_{n(k)},1)$ tend to zero as
$k\rightarrow\infty$.

 Fix
$\omega$ equal to one-half the volume of the $3$-dimensional
Euclidean ball of radius $1$. We rescale, forming $\widetilde
B_k=\alpha_k B_{\lambda_kg_{n(k)}}(x_{n(k)},1))$  in such a way that
the volume of the unit ball in $\widetilde B_k$ centered at
$x_{n(k)}$ is $\omega$. These rescaling factors divided by $c_k$
tend to infinity, so that $\widetilde B_k$ has a $(2,\mu_k)$
strainer at $x_{n(k)}$ of size that tends to infinity as
$k\rightarrow\infty$. Let $(\widetilde B,\bar x)$ be the limit of a
subsequence. By Proposition~\ref{smlimits} this limit is a smooth,
complete manifold of non-negative curvature and without boundary,
and the convergence is a smooth. The existence of the
$(2,\mu_k)$-strainers of size going to infinity in the sequence
implies that there is an isometric copy of $\Ar^2$ through $\bar x$
in $\widetilde B$. Hence, by Corollary~\ref{prod}, $\widetilde B$
splits as a product of $\Ar^2$ with a complete, connected
$1$-manifold without boundary. This $1$-manifold cannot be $\Ar^1$
because the volume of the unit ball in $\widetilde B$ is one-half
the volume of the unit ball in Euclidean space. Thus, $\widetilde B$
is the product of a circle with $\Ar^2$. Rescaling again by a fixed
constant, we can make the limit the product of the circle of length
$1$ with $\Ar^2$. The conclusion of the lemma then holds for all $k$
sufficiently large by taking limits. This is a contradiction and
proves the existence of the map $\varphi$ as required.

From this and the fact that $B_{\lambda g_n}(x_n,1)$ is within
$\hat\epsilon$ of a $2$-dimensional ball, it is easy to see that the
lengths of the fibers are at most $C\hat \epsilon$ for some
universal constant $C$.
\end{proof}

\begin{defn}
Anytime we have an embedding $\varphi\colon S^1\times
B(0,\epsilon^{-1})\to M$ with $x\in \varphi(S^1\times \{0\})$ that
satisfies the conclusion of the previous lemma, we say that $x$ is
{\em the center of a $S^1$-product neighborhood with
$\epsilon$-control.} The {\em horizontal spaces} of an $S^1$-product
neighborhood are the subspaces $\varphi(\{\theta\}\times B(0,\epsilon^{-1}))$
for $\theta\in S^1$.
\end{defn}

We need a semi-local version of this result.

\begin{prop}\label{semilocal}
Fix $\epsilon'>0$ sufficiently small. Then there is
$\epsilon_0(\epsilon')>0$ such that the following hold for all
$\epsilon<\epsilon_0$. Let $0<\mu\le \mu_2(\epsilon)$ as in
Lemma~\ref{generic2D}. For any $d>0$ there is
$\bar\epsilon(\epsilon,\mu,d)>0$ such that the following holds for
all $\hat\epsilon<\bar\epsilon(\epsilon,\mu,d)$. Suppose that
$B_{\lambda g_n}(x_n,1)$ is within $\hat\epsilon$ of a standard $2$-dimensional
ball $B(\bar x,1)$. Suppose that $(a_1,a_2,b_1,b_2)$ is a
$(2,\mu)$-strainer at a point $\bar y\in B(\bar x,1/2)$ of size
$d$. Then there is a constant $d'>0$ depending only on $d$ such that
the following holds. Let $y_n\in B_{\lambda g_n}(x_n,1/2)$ be a
point within distance $\hat \epsilon$ of $\bar y$ and let $\tilde
a_1,\tilde a_2$ be points of $B_{\lambda g_n}(x_n,1)$ within
$\hat\epsilon$ of $a_1,a_2$. Then there is an open subset $U_n$ with
$B_{\lambda g_n}(y_n,d')\subset U_n\subset B_{\lambda g_n}(y_n,2d')$
such that the function $F=(f_1,f_2)$, where $f_i=d(\tilde
a_i,\cdot)$, determines a (topological) fibration of $U_n$ by circles over an open
topological ball. Furthermore, for each $x\in U_n$ there is an
$S^1$-product neighborhood with $\epsilon$-control centered at $x$,
and the fiber through $x$ of this fibration  makes an angle\footnote{By angle we mean the limit as
$q'\in S(q)$ approaches $q$ of angle between the direction of any
geodesic from $q$ to $q'$ and the horizontal subspace at $q$.} within $\epsilon'$ of
$\pi/2$ to the horizontal spaces of this $S^1$-product structure.
Lastly, this fiber is isotopic in that $S^1$-product neighborhood to
the $S^1$-factor.
\end{prop}

\begin{proof}  For any $\mu'>0$, provided that $d'<<d$
and $\hat\epsilon$ are sufficiently small, choosing points $\tilde
b_1,\tilde b_2$ in $B_{\lambda g_n}(x_n,1)$ within $\hat \epsilon$
of $b_1,b_2$, the quadruple $\{\tilde a_1,\tilde a_2,\tilde
b_1,\tilde b_2\}$ is a $(2,\mu+\mu')$-strainer at any point
$B_{\lambda g_n}(y_n,d')$ of size $d/2$. By Lemma~\ref{generic2D}
provided that $\mu+\mu'$ is sufficiently small, for all
$\hat\epsilon$ sufficiently small (depending on $\mu+\mu'$ and $d'$)
every point of $B_{\lambda g_n}(y_n,d')$ is the center of an
$S^1$-product neighborhood with $\epsilon$-control. Clearly, the
geodesics from each of the four points of the strainer to any point
of this ball are almost horizontal in the $S^1$-product structure.
It follows from Section 11 of \cite{BGP} that the fibers of $F$ are
circles and that they are almost orthogonal to all geodesics from
the four points of the strainer, and hence they are almost
orthogonal to the horizontal spaces. All the errors go to zero $\mu\rightarrow 0$,
$d'\rightarrow 0$ and $\hat \epsilon\rightarrow 0$. Lastly, since
the restriction of $F$ to any horizontal space of an $S^1$-product
structure is a homeomorphism into, the fibers of the fibration
structure on $U_n$ `go around' the $S^1$-direction once and hence
are isotopic to the $S^1$-factor.
\end{proof}

This is a semi-local result: it is not small or the order of the
fiber but it is small on the order of the base. But there is a truly
global result obtained by piecing together the $S^1$-product
structures to form a global $S^1$-fibration.

\subsection{The global $S^1$-fibration}

\begin{prop}\label{s1glue}
Given $\epsilon'>0$,  the following holds for all  $\epsilon>0$ less
than a positive constant $\epsilon_1(\epsilon')$. Suppose that
$K\subset M$ is a compact subset and each $x\in K$ is the center of
an $S^1$-product neighborhood with $\epsilon$-control. Then there is
an open subset $V$ containing $K$ and  a smooth $S^1$-fibration structure on $V$. Furthermore, if $U$ is an $S^1$-product structure with $\epsilon$-control that contains a fiber $F$ of the fibration on $V$, then $F$ is within $\epsilon'$ of vertical
in  $U$ and $F$ generates the fundamental group $U$. In particular, the diameter of $F$  is at most twice the length of any circle in the $S^1$-product structure centered at any point of $F$.
\end{prop}

The proof of this proposition takes up this entire subsection.
For $\epsilon>0$ sufficiently small, we set $N=[1/\epsilon]$. Recall
that an $S^1$-product neighborhood $U\subset M$ is the image
$\varphi(S^1\times B(0,\epsilon^{-1}))$ with the property that
there is $\lambda_{U}>0$ such that
$\varphi^*(\lambda_{U} g)$ is within $\epsilon$ in the $C^N$-topology of
$g_0$,  the product of the Riemannian metric of length $1$ on $S^1$
and the usual Euclidean metric on the ball $B(0,\epsilon^{-1})$ in
the plane.

\subsubsection{Comparing the standard metrics on the overlap}

The first thing to do is to show that on the overlap of $S^1$-product neighborhoods
the standard metrics are close.

\begin{claim}
Given $\epsilon'>0$ there is $\epsilon>0$ such that the following
holds. Suppose that $U_1=\varphi_1(S^1\times B(0,\epsilon^{-1}))$
and $U_2=\varphi_2(S^1\times B(0,\epsilon^{-1}))$ are $S^1$-product
neighborhoods with $\epsilon$-control in
$M$. Suppose that there is a point
$$x\in\varphi_1(S^1\times B(0,\epsilon^{-1}/2))\cap \varphi_2(S^1\times B(0,\epsilon^{-1}/2)).$$
Then for $i=1,2$ the circle factor $F_i$ though $x$ in the product structure on
$U_i$ is within $\epsilon'$ of vertical in the product structure of
$U_{3-i}$. The length of this fiber is between $1-\epsilon'$ and
$1+\epsilon'$ times the length of any circle factor in the product structure
of $U_{3-i}$ as is the ratio $\lambda_{U_1}/\lambda_{U_2}$.
The homotopy class of $F_i$ generates $\pi_1(U_{3-i})$.
\end{claim}

\begin{proof}
Without loss of generality we can assume that $\lambda_{U_2}\ge
\lambda_{U_1}$. Let $\zeta$ be the $g$-shortest homotopically non-trivial
loop through $x$ in $U_2$. Its $g$-length is close to
$\lambda_{U_2}^{-1/2}$. Hence, it is contained in $U_1$ and its length
with respect to the product metric $g_0$ on $U_1$ is close to
$(\lambda_{U_1}/\lambda_{U_2})^{1/2}\le 1$. Let us suppose that it is homotopically
trivial in $U_1$. Then it bounds a disk contained in the
$g$-neighborhood of size $2\lambda_{U_2}^{-1/2}$ of $x$. This disk is then
contained in $U_2$, which is a contradiction. It follows that
$\zeta$ is a homotopically non-trivial loop in $U_1$ through $x$.
Since its length in the metric $g_0$ on $U_1$ is close $(\lambda_{U_1}/\lambda_{U_2})^{1/2}\le 1$, the loop
$\zeta$ generates the fundamental group of $U_1$. It follows that $\lambda_{U_1}/\lambda_{U_2}$ must be close to one. The errors
in these estimates go to zero as $\epsilon$ tends to zero.
\end{proof}

\begin{cor}\label{metclo}
We continue with the notation of the previous claim. Given $\epsilon'>0$
 if $\epsilon>0$ is sufficiently small then the restrictions of $(\varphi_1^{-1})^*g_0$ and $(\varphi_2^{-1})^*g_0$  to $\varphi_1(S^1\times B(0,\epsilon^{-1}/2))\cap \varphi_2(S^1\times B(0,\epsilon^{-1}/2))$ are within $\epsilon'$ in the $C^N$-topology.
 \end{cor}

\subsubsection{Bounding the intersections}

Now we turn to constructing a finite cover with a uniformly bounded number of neighborhoods meeting any given neighborhood.

\begin{claim}\label{Rcover} Fix $R<\infty$ and $\epsilon'>0$. Then for all $\epsilon>0$ sufficiently small,
there is a  finite collection of $S^1$-product neighborhoods with
$\epsilon$-control
$$\varphi_1(S^1\times B(0,\epsilon^{-1})),\ldots,
\varphi_T(S^1\times B(0,\epsilon^{-1}))$$
such that the union of
the images $U'_i=\varphi_i(S^1\times B(0,R))$ cover $K$, and the $\varphi_i(S^1\times B(0,R/3))$ are disjoint. Furthermore for every $i,j$,  $(\varphi_i^{-1})^*g_0$ and $(\varphi_j^{-1})^*g_0$ are within $\epsilon'$ in the $C^N$-topology for Riemannian metrics $$\varphi_i(S^1\times B(0,\epsilon^{-1}/2))\cap \varphi_j(S^1\times B(0,\epsilon^{-1}/2)).$$
\end{claim}

\begin{proof} Fix $\epsilon>0$ sufficiently small.
If $\varphi_i(S^1\times B(0,R/3))\cap \varphi_j(S^1\times
B(0,R/3))\not=\emptyset$, then, by the previous result, the standard metrics on the two
images almost agree, and in particular,  their union is contained in
$\varphi_i(S^1\times B(0,R))$. Take a collection $\{\widehat U_i=\varphi_i(S^1\times B(0,\epsilon^{-1}))\}$ of
$S^1$-product neighborhoods with $\epsilon$-control centered at
points of $K$, maximal with respect to the property that the
$\varphi_i(S^1\times B(0,R/3))$ are disjoint. Then the
$U'_i=\varphi_i(S^1\times B(0,R))$ cover $K$. If we have chosen $\epsilon>0$ sufficiently small, the last statement follows from the previous result.
\end{proof}

\begin{claim}\label{Cnumber}
Given $R>4$, there is an integer $C=C(R)$ such that following holds
for all $\epsilon>0$ sufficiently small. Let $(M,g)$ be a Riemannian
$3$-manifold with curvature $\ge -1$. Suppose that we have a
collection $\{\widehat U_i=\varphi_i(S^1\times B(0,\epsilon^{-1}))\}_i$
of $S^1$-product neighborhoods with $\epsilon$-control. Let $U_i$ be
the image of $\varphi_i(S^1\times B(0,R+1))$. Suppose also that
$\varphi_i(S^1\times B(0,R/3))\cap \varphi_j(S^1\times
B(0,R/3))=\emptyset$ for all $i\not=j$. Then for each $i$ the number
of $j$ for which $U_i\cap U_j\not=\emptyset$ is at most $C$.
\end{claim}

\begin{proof}
This is immediate from the fact that the standard metrics almost
agree on the overlaps of the $U_i$.
\end{proof}

For $R<\epsilon^{-1}$ we define a {\em reduced $S^1$-product structure
with $\epsilon$-control of size $R$} to be an embedding
$\varphi\colon S^1\times B(0,R)\to M$ with the property that there is
$\lambda>0$ such that $\varphi^*\lambda g$ is within $\epsilon$ in the $C^{N}$-topology
to the standard product metric $g_0$ on this product.

Fix $R$ and a covering $\{U_a\}_{a\in A}$ of $K$ as in Claim~\ref{Rcover}.
It follows directly from  Claim~\ref{Cnumber} that we can divide the open sets
 $\{U_a\}$ into  $C$ groups ${\mathcal
U}_1,\ldots,{\mathcal U}_C$ with the following properties:
\begin{enumerate}
\item Each ${\mathcal U}_i$
consists of a disjoint union of finite number of the $U_a$, denoted $U_{i,1},\ldots, U_{i,j_0(i)}$.
\item Each $U_a$ in the original collection occurs as exactly one of the $U_{i,j}$, so that in particular, setting ${\mathcal U}'_i$ equal to the images $\varphi_{i,j}(S^1\times B(0,R))$ for $1\le j\le j_0(i)$, the union
$\cup_{i=1}^C{\mathcal U}'_i$ covers $K$.
\end{enumerate}

\begin{defn}
 For each $0\le D\le 1$
we define ${\mathcal U}_i^{[D]}$ to be the union of the images
$\varphi_{i,j}(S^1\times B(0,R+1-D))$. Notice that ${\mathcal U}_i'={\mathcal U}_i^{[1]}$.
\end{defn}

\subsubsection{The Gluing}

Given a smooth fibration of a Riemannian manifold $M$ by circles,
then on each fiber $F$ there is a unique
measure $d\mu_F$ that is conformally equivalent to the measure
induced by the restriction of the Riemannian metric to $F$
and in which $F$ has total length $1$.

Suppose that we have an open subset $W\subset M$ that is the union
of restrictions of $S^1$-product neighborhoods with $\alpha$-control
to subsets $U_i=\varphi_i(S^1\times B(0,R'))$ for some $R\le R'\le
R+1$, and suppose that the circle fibrations of the various  $U_i$
are compatible so that they define a circle fibration on $W$.
Suppose also that we have a reduced $S^1$-product
structure with $\epsilon$ control $\varphi\colon S^1\times
B(0,R+2)\to M$. Let $U=\varphi(S^1\times B(0,R+1))$. Assuming that $\alpha$
and $\epsilon$ are sufficiently small, let us define a map
from the saturation of $U\cap W$ under the $S^1$-fibration on $W$
to $S^1\times B(0,R+2)$.
For $\alpha$ and $\epsilon$ sufficiently small  this saturation is contained in $\varphi(S^1\times B(0,R+2))$.
Suppose that $p$ is a point of the saturation of  $U\cap W$, say
$p=\varphi(\theta,x)$. Let $F_p$ be the fiber of the fibration
structure on $W$ through $p$. For each $q\in F_p$ we have
$(\theta(q),x(q))$ defined by
$\varphi^{-1}(q)=(\theta(q),x(q))$, so that $x\colon F_p\to B(0,R+2)$. We form $\hat
x(p)=\int_{F_p}x(q)d\mu_{F_p}$ and define the map
$$\psi(p)=(\theta(p),\hat x(p)).$$

The following is obvious from the definitions

\begin{claim}
If $F$ is an orbit of the $S^1$-fibration on $W$ passing though a point of $U$, then $\hat x\colon
F\to B(0,R+2)$ is constant.
\end{claim}

Denote by ${\rm Sat}(U\cap W)$ the saturation of $U\cap W$ under the fibration structure on $W$.

\begin{cor}\label{close} Given $\epsilon_1>0$, then for all $\alpha,\epsilon>0$ sufficiently small, the map $\hat x\colon {\rm Sat}(U\cap W)\to B(0,R+2)$ is within $\epsilon_1$ in the $C^{N+1}$-topology of the restriction to ${\rm Sat}(U\cap W)\subset U$ of the composition of $\varphi^{-1}$ with the projection in product structure to $B(0,R+2)$.
\end{cor}

\begin{proof}
It follows immediately from Corollary~\ref{metclo} that the fibers of the $S^1$-fibration on ${\rm Sat}(U\cap W)$ induced from the fibration on $W$ are geodesics in a metric that is $C^N$-close to the metric $g_0$ on $U$. From this we see that  the map $p\mapsto \hat x(p)$ is $C^{N+1}$-close to the composition of $\varphi^{-1}$ with the projection to $B(0,R+2)$ with the same error estimate.
\end{proof}

It follows from Corollary~\ref{close} that
given $\epsilon_1>0$, there is a constant $\alpha_0(\epsilon_1)>0$ such that
if $\alpha$ and $\epsilon$ are less than $\alpha_0(\epsilon_1)$,
then we can define a map $\psi\colon {\rm Sat}(U\cap W)\to S^1\times B(0,R+2)$
by sending $p=\varphi(\theta,x)$ to $\psi(p)=(\theta(p),\hat x(p))$.
Again invoking Corollary~\ref{close}, we see that:

\begin{cor}\label{close1} Provided that $\alpha$ and $\epsilon$ are less that $\alpha_0(\epsilon_1)$,
the composition
$${\rm Sat}(U\cap W)\buildrel\psi\over\longrightarrow S^1\times B(0,R+2)\buildrel
\varphi\over\longrightarrow \varphi(S^1\times B(0,R+2))$$  is within $\epsilon_1$ of the inclusion of ${\rm Sat}(U\cap W)\subset \varphi(S^1\times B(0,R+2))$ in the $C^{N+1}$-topology.
\end{cor}

Let $\beta\colon [0,R']\to [0,1]$ be
a function that is identically $1$ near $R'$  and identically zero
on a neighborhood of $[0,R'-1/C]$.
We define $\beta_i\colon U_i\to [0,1]$ by $\beta_i(\varphi_i(\theta,x))=\beta(|x|)$.
For all $i$ such that $U_i\cap U\not=\emptyset$,
the gradients of the $\beta_i$ with respect to $\lambda_Ug$ are bounded independent of $i$.
(Recall that $\lambda_U g$ is the multiple of $g$ which is close to the standard
product metric $g_0$ on $U$.)
We set $\hat\beta\colon W\to [0,1]$ equal to the product over the $i$ of the $\beta_i$. This function is identically $1$ in the complement of $W$ and the restriction to $U$ of $\hat\beta$ has a gradient with respect to $g_0$ that is bounded depending only on $C$.
Define $\Psi\colon U\to S^1\times B(0,R+2)$ by
$$\Psi(p)=\beta(p)\varphi^{-1}(p)+(1-\beta(p))\psi(p).$$

\begin{claim}
Given $\epsilon_1$ there is $\alpha_1=\alpha_1(\epsilon_1)>0$ such that if $\alpha$ and $\epsilon$ are less than $\alpha_1$, then $\Psi$ is within $\epsilon_1$ of $\varphi^{-1}$ in the $C^{N+1}$-topology using the metrics $\lambda_Ug$ on the domain and $g_0$ on the range.
\end{claim}

\begin{proof}
This follows immediately from Corollary~\ref{close1}.
\end{proof}

We set $W'\subset W$ equal to $\beta^{-1}(0)$.

\begin{claim}\label{glres}
$W'$ is the union of $\varphi_i(S^1\times B(0,R''))$ where  $R''=R'-1/C$.
 In particular, $W'$ is saturated under the
$S^1$-fibration structure on $W$.
The image of $\Psi$ contains $S^1\times B(0,R+1-1/C)$. Setting $\varphi'\colon S^1\times B(0,R+1-1/C)\to M$ equal to the restriction of the inverse of $\Psi$, we have
\begin{enumerate}
\item $\varphi'$ is a reduced $S^1$-product neighborhood with $\epsilon'$-control of size $R+1-1/C$.
\item If $\varphi'(\theta,x)\subset W'$, then $\varphi'(S^1\times \{x\})$ is a fiber of the $S^1$-fibration on $W'$.
    \item For any $T\le R+1$, the image $\varphi'(S^1\times B(0,T))$ contains $\varphi(S^1\times B(0,T-1/C))$.
\end{enumerate}
\end{claim}

We denote the image  $\varphi'(S^1\times B(0,R+1-1/C))$ by $U^{[1/C]}$.

\begin{cor}
The  $S^1$-fibration structure on $U^{[1/C]}$
coming from the $S^1$-product structure and the given $S^1$-fibration structure on $W'$ are compatible on the overlap $U^{[1/C]}\cap W'$.
\end{cor}

This claim shows that, at the expense of shrinking $W$ to $W'$ and at the expense of deforming $\varphi$ slightly to a
reduced $S^1$-product structure with $\epsilon'$-control, $\varphi'\colon S^1\times B(0,R+1-1/C)\to M$, we can make the $S^1$-fibrations compatible on the overlap, so that together they define an $S^1$-fibration on the union $W'\cup U^{[1/C]}$.
One more remark is in order. If we have not a single reduced $S^1$-product neighborhood with $\epsilon$-control $U$, but rather a collection of them $U_{i_0,j},\ 1\le j\le j_0(i_0)$, whose images are disjoint, then we can perform this operation simultaneously on all of them, so as to deform them all
to $S^1$-product neighborhoods with $\epsilon_1$-control compatible with the circle fibration on $W'$.

Now we are ready to apply this gluing argument by induction to the ${\mathcal U}_1,\ldots,{\mathcal U}_C$. We begin with ${\mathcal U}_1$. In the inductive step, deforming and gluing in ${\mathcal U}_{i_0}$, we cut down the $S^1$-product neighborhoods in the neighborhoods that make up the previous ${\mathcal U}_{i}$ by $1/C$. The deformation of the maps $\varphi_{i_0,j}$ produces a reduced $S^1$-product neighborhood with $\epsilon_1$-control where the amount of the deformation and $\epsilon_1$ depend only on the control we have at the previous step.
Thus, we can iterate this construction $C$ times keeping a fixed control, $\epsilon'$, on all the $S^1$-product neighborhoods and a given control on the size of the deformations, provided only that we arrange that the original control, $\epsilon$, is sufficiently small given $C$, $\epsilon'$, and the desired control on all deformations.

It follows from the second conclusion of Claim~\ref{glres} that the $S^1$-fibrations induced by the product structures on the deformed ${\mathcal U}_i$ are compatible and hence define a global $S^1$-fibration on the union.
It follows from the third conclusion of Claim~\ref{glres} that the union of the deformed $S^1$-product neighborhoods contains $K$. The last statement in the conclusion of Proposition~\ref{s1glue}
is immediate from the construction.
This completes the proof of Proposition~\ref{s1glue}.

\subsection{Balls centered at points of $\partial M_n$}

The results about the generic behavior over interior points of the base is enough to
establish what the neighborhoods of the boundary of the $M_n$ look like.

\begin{prop}\label{bdrybehavior} Fix $\epsilon>0$.
For all $n$ sufficiently large, for any point $x\in \partial M_n$ the ball
$B_{\rho^{-2}_n(x)g_n}(x,1)$ is within $\epsilon$ of the interval of length $1$, and $x$ is within
$\epsilon$ of the endpoint $0$.
\end{prop}

\begin{proof}
Suppose that the result is not true. Then after passing to a subsequence (in $n$) we can suppose
that for each $n$ we have $x_n\in \partial M_n$ for which the result does not hold. Let $T_n$ be
the component of $\partial M_n$ containing $x_n$ and let $C_n$ be the topologically trivial collar
containing the neighborhood of size $1$ of $T_n$. Since $\partial M_n$ is convex and $\rho_n\le
{\rm diam}\, M_n/2$, the balls $B_{\rho_n^{-2}(x_n)g_n}(x_n,1)$ are Alexandrov balls. Because the
curvatures on the topologically trivial collar which includes the neighborhood of size $1$ about
$\partial M_n$, are bounded above by $-3/16$, it follows that $\rho_n(x_n)\le \sqrt{16/3}$. Hence,
$B_n=B_{\rho_n^{-2}(x_n)g_n}(x_n,1/4)$ is contained in $C_n$. We shall show that, after passing to
a subsequence the $B_{\rho_n^{-2}(x_n)g_n}(x_n,1/4)\subset C_n$ converge to the interval $[0,1/4)$
with the $x_n$ converging to the endpoint $0$. Assuming this, it follows that the
$B_{\rho_n^{-2}(x_n)g_n}(x_n,1)$ also converge to a $1$-dimensional Alexandrov space $\hat J$ and
that the $x_n$ converge to an endpoint of $\hat J$. Since the diameter of $M_n$ is greater than
$2\rho_n(x_n)$, it follows that $\hat J$ has length $1$.

We have already remarked that because of the convexity of $\partial M_n$, the $B_n$ are Alexandrov
balls. Passing to a subsequence, there is a limiting Alexandrov space $J$ which is an Alexandrov
ball of diameter $1/4$ centered at $\bar x={\rm lim}x_n$. Because of the volume collapsing
condition on the $M_n$, it follows that $J$ is either of dimension $1$ or $2$. We rule out the
possibility that ${\rm dim}\, J=2$. Suppose to the contrary that the dimension of $J$ is $2$. Fix
$0<\mu\le \mu(\epsilon)$ from Lemma~\ref{generic2D} and fix $0<\alpha<<1/4$. Then there is a point
$\bar y$ of $J$ within distance $\alpha$ of $\bar x$ that has a $(2,\mu)$-strainer of some size
$d>0$. Fix $d'>0$ as in Proposition~\ref{semilocal} for this value of $d$. For all $n$
sufficiently large $\epsilon_n<\bar\epsilon(\epsilon,\mu,d)$ from Proposition~\ref{semilocal}. By
Proposition~\ref{semilocal}  this means that for all $n$ sufficiently large there is a point
$y_n\in B_{g'_n(x_n)}(x_n,1/4)$ within $\epsilon_n$ of $\bar y$ with a neighborhood
$U_n$ with $B_{g'_n(y_n)}(y_n,d')\subset U_n\subset B_{g'_n(y_n)}(y_n,2d')$
that is fibered by circles over a topological ball.
In particular, $\pi_1(U_n)$ is infinite cyclic and hence the
image of $\pi_1(U_n)$ in $\pi_1(C_n)$ is either trivial or infinite cyclic.

We denote by  $\widetilde C_n$ the universal covering of $C_n$ with its inherited Riemannian
metric. Fix lifts $\tilde x_n$ for $x_n$ and $\tilde y_n$ for $y_n$ that are within distance
$\alpha$ of each other. Since $C_n$ is a topologically trivial collar of $T_n$, the group of
covering transformations of $\widetilde C_n$ over $C_n$ is a free abelian group of rank $2$. We
implicitly use the metric $\rho_n^{-2}(x_n)g_n$ on $C_n$ and the induced Riemannian metric on the
covering. Since the diameter of $T_n$ in the metric $\rho_n^{-2}(x_n)g_n$ is at most $4w_n$, it
follows that the fundamental group of $C_n$ is generated by elements which, acting as covering
transformations on $\widetilde C_n$ move $\tilde x_n$ a distance at most $8w_n$. In particular, we
can choose an element $\gamma_n\in \pi_1(T_n)$ that moves $\tilde x_n$ a distance at most $8w_n$
and which is not of finite order in the quotient of $\pi_1(C_n)$ by the image of
$\pi_1(U_n)\to\pi_1(C_n)$. Of course, $\gamma_n$ moves every point of $B(\tilde x_n,4w_n)$ a
distance at most $16w_n$. Since the translates of $B(\tilde x_n,4w_n)$ by $\pi_1(T_n)$ cover all
of $T_n$, and since the fundamental group is abelian, it follows that $\gamma_n$ moves every point
of $\widetilde T_n$ a distance at most $16w_n$ and consequently $\gamma_n^k$ moves $\tilde x_n$ a
distance at most $16kw_n$. Since $U_n$ contains $B_{\rho_n^{-2}(x_n)g_n}(y_n,d')$, it follows that
each component of the preimage $\widetilde U_n$ of $U_n$ contains the ball of radius $d'$ about
each lift $\tilde y_n$ lying in that component. Since the group generated by $\gamma_n$ freely
permutes the components of $\widetilde U_n$, it follows that every power of $\gamma_n$ moves every
preimage of $y_n$ a distance at least $2d'$.

The induced covering $\widetilde B_n\subset \widetilde C_n$  is a Riemannian manifold with convex
boundary and hence is a local Alexandrov space. Furthermore, for $1\le k\le \alpha/16w_n$ the
image $\gamma_n^k\tilde x_n$ is within $\alpha$ of $\tilde x_n$. This implies that there are at
least $m=\alpha/16w_n$ distinct translates of $\tilde y_n$ all within distance $2\alpha$ of
$\tilde x_n$, and these translates are all at least distance $2d'$ apart. Thus, letting $w$ and $w'$ be any
two such translates, the comparison angle $\tilde \angle w\tilde x_nw'$ is bounded away from zero.
Since the distances are much smaller than $1/4$, there are geodesics from $\tilde x_n$ to each of
these $m$ translates of $\tilde y_n$. According to Remark 3.5 of \cite{BGP} monotonicity of angles
as in Part 2 of Proposition~\ref{ASbasics} holds in the region in which we are working. Thus, the angles that these geodesics make with
each other at $\tilde x_n$ are bounded away from zero. As $n$ goes to infinity the number of these
translates goes to infinity. But there is there is a fixed upper bound to the number of geodesics
emanating from a point in a $3$-manifold with the property that the angles between any two
distinct ones is bounded below by a fixed positive constant. This contradiction proves that $J$ is
$1$-dimensional.

Take any point $y_n\in C_n$ at distance $1/2$ from $T_n$ and join it
to $T_n$ by a minimal geodesic $\gamma_n$. Let $x'_n$ be its other
endpoint. This geodesic makes angle at most $\pi/2$ with any tangent
vector at $x'_n$. Taking limits we see that there is a geodesic,
$\bar \gamma$  in the limit $J$ with one endpoint being the limit,
$\bar x,$ of the $x_n$ such that $\bar \gamma$ makes angle at most
$\pi/2$ with any tangent direction at $\bar x$. It follows that
$\bar x$ is an endpoint of $J$.
\end{proof}

\subsection{The interior cone points}

\begin{prop}\label{intcone1} For any $\epsilon>0$ and $a>0$,
the following holds for all $\mu>0$ less than a positive constant
$\mu_3(\epsilon,a)$, for any $0<r_0\le 10^{-3}$,  and for all
$\hat\epsilon>0$ less than a positive constant
$\hat\epsilon_1(\epsilon,a,r_0)$. Suppose that, for some $n$, there
is a point $x_n\in M_n$ with the property that the ball $B_{\lambda
g_n}(x_n,1)$  is within $\hat\epsilon$ of a $2$-dimensional
Alexandrov ball $B(\bar x,1)$  of area $\ge a$ that is interior
 $\mu$-good at $\bar x$ on scale $r'$,
 where $r_0\le r'\le 10^{-3}$. Then there is a compact solid torus
$S$ contained in $ B_{\lambda g_n}(x_n,3r'/4)$ and containing $
B_{\lambda g_n}(x_n,r'/2)$. Furthermore, every point of
$U=B_{\lambda g_n}(x_n,3r'/4)\setminus B_{\lambda g_n}(x_n,r'/4)$ is
the center of an $S^1$-product neighborhood with $\epsilon$-control.
\end{prop}

\begin{proof}
First notice that it follows from the Bishop-Gromov inequality that
there is $a'>0$ depending only on $a$ such that if  $B(\bar x,1)$ is
a standard $2$-dimensional ball of area $\ge a$ then for any $0<r\le
1$, the area of $B(\bar x,r)$ is at least $a'r^2$. First let us show
that it suffices to prove the result when $r_0=10^{-3}$. For suppose
that for every $\epsilon>0$ and $a>0$ we have positive constants
$\mu'_3(\epsilon,a)$and $\hat\epsilon'_1(\epsilon,a)$ so that the
proposition holds for $r_0=10^{-3}$. Fix $\epsilon>0$, $a>0$, and
$r_0>0$. Suppose we have $\mu<\mu_3(\epsilon,a')$ and
$\hat\epsilon<(r_0/10^{-3})\hat\epsilon'_1(\epsilon,a')$. Given
balls $B_{\lambda g_n}(x_n,1)$ and $B(\bar x,1)$ as in the statement
for these values of $\mu$ and $\hat \epsilon$ and $a$, and some $r'$
with $r_0\le r'\le 10^{-3}$. Then $(10^{-3}/r')B(\bar x,1)$ is
interior $\mu$-good at scale $10^{-3}$ at $\bar x$. The unit subball
centered at $\bar x$ has area $\ge a'$. On the other hand
$B_{(10^{-3}/r')\lambda g_n}(x_n,1)$ is within
$(10^{-3})/r')\hat\epsilon<\hat\epsilon'_1(\epsilon,a')$ of
$(10^{-3}/r')B(\bar x,1)$. By our assumption that the result holds
in the special case when $r_0=10^{-3}$, we see that the conclusion
holds for $B_{(10^{-3}/r')\lambda g_n}(x_n,1)$ with $r'$ replaced by
$10^{-3}$. Hence, rescaling it holds for $B_{\lambda g_n}(x_n,1)$
with the given value of $r'$.

This allows us to assume that $r_0=10^{-3}$. Suppose that there are
sequences $\mu_k\rightarrow 0$ and $\hat\epsilon_k\rightarrow 0$ as
$k\rightarrow \infty$ and balls $B_{\lambda_kg_{n(k)}}(x_k,1)$
within $\hat\epsilon_k$ of standard $2$-dimensional balls $B(\bar
x_k,1)$ of area $\ge a$ that are interior $\mu_k$-good at $\bar x_k$
on scale $10^{-3}$ and yet the conclusion of the proposition does
not hold for $r'=10^{-3}$. Passing to a subsequence, we can suppose
that the $B(\bar x_k,1)$ converge to a standard $2$-dimensional ball
$B(\bar x_\infty,1)$. Because the $\mu_k\rightarrow 0$, it follows
that $B(\bar x_\infty,10^{-3})$ is a circular cone of some cone
angle $\alpha\le 2\pi$ which is bounded away from zero because $a$
is greater than zero. Since the $\hat\epsilon_k\rightarrow 0$, the
$B_{\lambda_k g_{n(k)}}(x_k,1)$ also converge to $B(\bar
x_\infty,1)$.

Let us first consider the case when $\alpha=2\pi$ so that $B(\bar x_\infty,10^{-3})$ is isometric
to a ball in $\Ar^2$. It follows from Proposition~\ref{semilocal} that there is $d'>0$, a
$(2,\mu)$-strainer $\{a_1,a_2,b_1,b_2\}$ for $x_{n(k)}$, and an open subset $U_{n(k)}$ containing
$B_{\lambda_kg_{n(k)}}(x_k,d')$ and contained in $B_{\lambda_kg_{n(k)}}(x_{n(k)},2d')$ with the
property that the function $F=(f_1,f_2)$ where $f_i=d(a_i,\cdot)$ determines a fibration of
$U_{n(k)}$ by circles over a disk in the plane. Furthermore, by Lemma~\ref{generic2D}, for all $k$
sufficiently large, there is an $S^1$-product neighborhood $V$ with $\epsilon$ control centered at
$x_{n(k)}$. Also, according to Proposition~\ref{semilocal} the circle of the fibration structure
on $U_{n(k)}$ passing through $x_{n(k)}$ is almost orthogonal to the horizontal spaces of the
$S^1$-product structure centered at that point and this circle is isotopic in $V$ to the
$S^1$-factor. This means that the closure of $V$ is a solid torus contained in $U_{n(k)}$ whose
core is isotopic to the fiber of the fibration structure on $U_{n(k)}$. It follows that the
inclusion of $V\subset U_{n(k)}$ induces an isomorphism on fundamental groups, both groups being
isomorphic to $\Zee$. Also, it follows that the region between $V$ and the closure of $U_{n(k)}$
is homeomorphic to $T^2\times I$. We have inclusions $V\subset B_{\lambda_k
g_{n(k)}}(x_{n(k)},d')\subset U_{n(k)}\subset B_{\lambda_k g_{n(k)}}(x_{n(k)},2d')$. For all $k$
sufficiently large, the distance function from $x_{n(k)}$ is regular on
$B_{\lambda_kg_{n(k)}}(x_{n(k)},2d')\setminus B_{\lambda_kg_{n(k)}}(x_{n(k)},d')$, and
consequently, the inclusion of the smaller ball into the larger induces an isomorphism on the
fundamental group. It then follows from the sequence of inclusions that the fundamental group of
$B_{\lambda_kg_{n(k)}}(x_{n(k)},d')$ is isomorphic to $\Zee$ and hence the metric sphere
$S_{\lambda_kg_{n(k)}}(x_{n(k)},d')$ is a $2$-torus. This $2$-torus is contained in the complement
of $V$ in the closure of $U_{n(k)}$ and separates the two boundary components of this region.
Since we have already seen that this region is homeomorphic to a product $T^2\times I$, it follows
that $S_{\lambda_kg_{n(k)}}(x_{n(k)},d')$ is isotopic in the closure of $U_{n(k)}$ to the boundary
of $U_{n(k)}$. Consequently, $B_{\lambda_kg_{n(k)}}(x_{n(k)},d')$ is a solid torus. Using the
regularity of the distance function from $x_{n(k)}$ see that $B_{\lambda_kg_{n(k)}}(x_{n(k)},a)$
is a solid torus for every $a\in [d',10^{-3}]$. The last statement in the proposition is immediate
from Lemma~\ref{generic2D}. This contradiction proves the result in the case when the limiting
$2$-dimensional space is flat.

Now we consider the case when the limiting cone angle $\alpha$ is less than $2\pi$. We rescale by
$10^3$ so that $r'$ in effect becomes $1$. In this case, according to Proposition~\ref{blowup} the
following holds for all $k$ sufficiently large. There is $x_{n(k)}'\in M_{n(k)}$ such that
$d_{\lambda_{k} g_{n(k)}}(x_{n(k)},x_{n(k)}')\rightarrow 0$ as $k\rightarrow\infty$ such that for
each $k$ sufficiently large, one of the following two alternatives holds: for
\begin{enumerate}
\item the distance function
from $x'_{n(k)}$ has no critical points on
$B_{\lambda_kg_{n(k)}}(x_{n(k)}',3/4)\setminus\{x'_{n(k)}\}$, or
\item there is $\delta_k\rightarrow 0$ such that the distance function from $x'_{n(k)}$
has no critical points
in $B_{\lambda_kg_{n(k)}}(x_{n(k)}',3/4)\setminus \overline
B_{\lambda_kg_{n(k)}}(x_{n(k)}',\delta_k)$ and has a critical point
 at distance $\delta_k$ from $x'_{n(k)}$.
\end{enumerate}
In Case 1 the level sets of the distance function are $2$-spheres and the metric balls are topological  $3$-balls.
Let us suppose that Case 2 holds. According to Proposition~\ref{blowup} after passing to
a subsequence the rescaled balls $\delta_k^{-1}B_{\lambda_kg_{n(k)}}(x_{n(k)}',3/4)$ converge in
the Gromov-Hausdorff topology to a complete $3$-dimensional Alexandrov space of curvature $\ge 0$.
By Proposition~\ref{smlimits} the limit is actually a smooth, orientable Riemannian manifold of
curvature $\ge 0$ and the convergence is $C^\infty$. Thus,  the limit has a soul which is either a
point, a circle, or a surface of non-negative curvature. We claim the soul is not a surface. For
if the soul is a surface, then either the limiting $3$-manifold or its double covering is a
Riemannian product of that surface with $\Ar$. The limit cannot be the product of a surface with
$\Ar$ because the complement of a small neighborhood about the soul is close to a connected
$2$-dimensional space and hence is connected. Thus, if the soul is a surface, the limiting
$3$-manifold is a non-orientable $\Ar$-bundle over that surface. It would then follow that given
any $\beta>0$ there is $R<\infty$ such that for all $k$ sufficiently large any triangle
$ax'_{n(k)}b$ with $|ax'_{n(k)}|=|bx'_{n(k)}|=R$ has comparison angle less than $\beta$ at
$x'_{n(k)}$. On the other hand, because the limit of the $B_{\lambda_kg_{n(k)}}(x'_{n(k)},1)$ is
$2$-dimensional, there is $\beta_0>0$ such that for all $k$ sufficiently large there are geodesics
from $x'_{n(k)}$ to points at a fixed positive distance that make a comparison  angle at
$x'_{n(k)}$ which is least $\beta_0$. This contradicts the monotonicity of the comparison angles.

This shows if Case 2 holds then  the soul of the limiting
manifold is either a circle or a point, and hence the level sets
$d(x'_{n(k)},\cdot)^{-1}(a)$ are either $2$-tori or $2$-spheres for
every $a$ with $\delta_k<a\le 3/4$ and these bound either solid tori
or $3$-balls in the metric ball. In the first case, the level sets
are topological $2$-spheres and they bound $3$-balls in the metric
ball.

Next, we shall show that in either case, provided that $\epsilon>0$ is sufficiently
small, the level sets of the distance function from $x'_{n(k)}$ must
be $2$-tori. Fix $\epsilon'>0$ small and let $\epsilon>0$ be such
that Proposition~\ref{s1glue} holds for these values of $\epsilon'$
and $\epsilon$. Consider the annular region
$A_k=d(x'_{n(k)},\cdot)^{-1}([1/4,3/4])$. This is a compact subset
and if $k$ is sufficiently large, then every point of this compact
set is within $\hat\epsilon$ of a point of $ B(\bar x_k,1)$ at which
$ B(\bar x_k,1)$ is interior flat of some fixed scale $s$. Having
taking $\hat\epsilon$ sufficiently small, by
Proposition~\ref{s1glue} there is an open subset $U_{n(k)}\subset
M_{n(k)}$ containing $A_k$ that is the total space of a circle
fibration where the fibers of the fibration make angle at most
$\epsilon'$ with the horizontal spaces of the $S^1$-product
neighborhoods with $\epsilon$-control at every point of $A_k$. Of
course, there is a compact subsurface $\Sigma_k$ contained in the
base of the fibration with the property that the pre-image, $W_k$,
of $\Sigma_k$ contains $A_k$. Each component of $\partial W_k$ is a
torus. Thus, for every $b\in (1/4,3/4)$ the level set
 $d(x_{n(k)}',\cdot)^{-1}(b)$ separates two boundary components of $W_k$.
 Since a $2$-sphere in the total
 space of a circle bundle cannot separate boundary components of
 that circle bundle, it follows that these level sets are  $2$-tori.

 This implies that for all $k$ sufficiently large,
 Case 2 holds and the soul of the limiting $3$-manifold is a circle. Thus,
  for every $k$ sufficiently large,
 for every $0<b\le 3/4$ the pre-image
 $d(x'_{n(k)},\cdot)^{-1}([0,b])$ is a solid torus. We fix $b\in
 (1/2,3/4)$ and set the pre-image of $[0,b]$ equal to $S$.
 Of course, provided that $k$ is sufficiently large
 $B(x_{n(k)},1/2)\subset S\subset B(x_{n(k)},3/4)$. This gives a contradiction
 and completes the proof of the result.
\end{proof}

The argument above actually proves more.

\begin{cor}\label{invtorus}
Fix $\epsilon'>0$ sufficiently small and let
$0<\epsilon<\epsilon_1(\epsilon')$, where $\epsilon_1(\epsilon')$ is
as in Proposition~\ref{s1glue}.  Under the hypothesis and notation
of the previous proposition, suppose that we have an open subset
$\widetilde U$ containing $B_{\lambda g_n}(x_n,3r'/4)\setminus
B_{\lambda g_n}(x_n,r'/4)$ with $\widetilde U$ being the total space
of an $S^1$-fibration with fibers making angle within $\epsilon'$ of
$\pi/2$ with the horizontal spaces of the $S^1$-product
neighborhoods with $\epsilon$-control at every point of $\widetilde U$. Then
there is a $2$-torus in $\widetilde U$ that is invariant under the
$S^1$-fibration structure, which is contained in $B_{\lambda
g_n}(x_n,r'/2)$, and which bounds a solid torus in $B_{\lambda
g_n}(x_n,3r'/4)$.
\end{cor}

There is a further result that is not actually  necessary for what follows
but which makes the picture clearer and also simplifies somewhat
several of the arguments.

\begin{prop}\label{seifertfib}
For $\epsilon'>0$ sufficiently small and let
$0<\epsilon<\epsilon_1(\epsilon')$, where $\epsilon_1(\epsilon')$ is
as in Proposition~\ref{s1glue}. Under the hypothesis of the previous
proposition, the $S^1$-factors in the local $S^1$-product structures
with $\epsilon$-control contained in $B_{\lambda
g_n}(x_n,3r'/4)\setminus B_{\lambda g_n}(x_n,r'/4)$ are
homotopically non-trivial in  $B_{\lambda g_n}(x_n,3r'/4)$.
\end{prop}

\begin{proof}
Suppose that the result does not hold for any $\epsilon'>0$. We take
a sequence of $\epsilon'_k$ tending to zero and $\epsilon'_k$
counter-examples $(B_k,x_k)$. After passing to a subsequence these
counter-examples converge to a $2$-dimensional Alexandrov ball
$(Z,z)$ with curvature $\ge -1$ which is interior good at the
limiting base point $z$ on some scale $\bar r'\ge r_0$. The
fundamental group $\Gamma_k$ of $B_k$ is infinite cyclic and the shortest
homotopically non-trivial loop through $x_k$ has a length that tends
to zero as $k\rightarrow\infty$. We
consider the universal coverings $\widetilde B_k$ of the $B_k$ and
let $\tilde x_k$ be a lifting of $x_k$. For any fixed $s>0$ and any fixed $N<\infty$
for all $k$ sufficiently large there are at least $N$ distinct preimages of $x_k$ in
$B(\tilde x_k,s)$. On the other hand, suppose that the circles in the product structure
contained in $B_{\lambda
g_k}(x_k,3r'/4)\setminus B_{\lambda g_k}(x_k,r'/4)$ are
homotopically trivial in  $B_{\lambda g_k}(x_k,3r'/4)$. Then there is $s>0$ such that for all
$k$ sufficiently large and any point $y_k\in \left(B_{\lambda
g_k}(x_k,3r'/4)\setminus B_{\lambda g_k}(x_k,r'/4)\right)$ the preimage in $\widetilde B_k$
of the ball $C_k$ of radius $s$ centered at $y_k$ is a disjoint union of components
mapping homeomorphically onto the ball. Fix $s\le r'/8$. This means that each of these preimages contains the ball of radius $s$ about the corresponding preimage of $y_k$. Fix a pre-image $\tilde x_k$ of $x_k$, and a preimage $\tilde y_k$ of $y_k$ within distance $3r'/4$ of $\tilde x_k$. For $k$ sufficiently large we have an arbitrarily large number of group elements of the fundamental group of $B_k$ that move  $\tilde x_k$ a distance at most $s$, but the balls of radius $s$ about the corresponding translates of $\tilde y_k$ are disjoint. Let $G(k)\subset \pi_1(\tilde B_k,x_k)$ be the set of elements moving $\tilde x_k$ a distance at most $s$ and let $N(k)$ be its cardinality. Since all of these points are contained in the ball of radius $r'$ about $\tilde x_k$, and the exponential mapping at the tangent space to $\tilde B_k$ at $\tilde x_k$ is defined out to distance at least $2r'$. In particular, there are geodesics from $\tilde x_k$ to each of the translates of $\tilde y_k$ by elements of $G(k)$, and consequently $N(k)$ geodesics of length $\le r'$, all of whose endpoints are separated by distances at least $2s$. Thus, the comparison angles at $\tilde x_k$ for the triples of points consisting of $\tilde x_k$ and two translates of $\tilde y_k$ are bounded away from zero independent of $k$. Since the exponential mapping is defined on the ball of radius $2r'$ in the tangent space to $\widetilde B_k$ at $\tilde x_k$, monotonicity holds for these triangles. This
is a contradiction.

\end{proof}

The topological import of this result about the fundamental group is
the following:

\begin{cor}\label{seifcor}
Under the notation and hypotheses of Corollary~\ref{invtorus}, the
$S^1$-fibration structure on $\widetilde U$  extends to a Seifert
fibration over $\widetilde U\cup B_{\lambda g_n}(x,3r'/4)$ with one
singular fiber.
\end{cor}

\begin{defn}
$B_{\lambda g_n}(x_n,r'/4)$ satisfying the conclusions of
Propositions~\ref{intcone1} and~\ref{seifertfib} and Corollary~\ref{invtorus}
is {\em an $\epsilon'$-solid torus
neighborhood near a $2$-dimensional interior cone point.}
\end{defn}

\begin{rem}
In fact, a strengthening of this argument (see Theorem 0.2 and the
material in Section 4 of \cite{SY2000}) proves that the order of the
exceptional fiber is bounded above by $2\pi/\alpha$ where $\alpha$
is the cone angle of the nearby interior $\mu$-good ball at its
central point. We shall not make use of this result.
\end{rem}

\subsection{Near almost flat boundary points}

Now let us turn to the parts of the $M_n$ close to flat boundary
points of a $2$-dimensional Alexandrov ball.

\begin{prop}\label{flatboundary} Given $\epsilon'>0$,
let $0<\epsilon<\epsilon_1(\epsilon')$, where
$\epsilon_1(\epsilon')$ is  the constant given in
Proposition~\ref{s1glue}. The following hold  for all $\xi>0$ less
than a positive constant $\xi_0(\epsilon)$ and  for all $\mu>0$ less
than a positive constant $\mu_4(\xi, \epsilon)$. For any $0<s_1\le
1/4$
 and, given $s_1$, for all $\hat\epsilon>0$ less than a positive constant
 $\hat\epsilon_2(\epsilon,\mu,s_1)$, suppose that, for some $n$, there is a point $x_n\in M_n$
with the property  that $B_{\lambda g_n}(x_n,1)$ is within
$\hat\epsilon$ of a $2$-dimensional Alexandrov ball $X=B(\bar x,1)$.
Suppose that $\gamma$
 is a $\xi$-approximation to $\partial X\cap B(\bar x,3/4)$ on scale $s_1$
 with $\mu$-control. Suppose that $\tilde\gamma$ is a geodesic in $M_n$ within
$\hat\epsilon$ of $\gamma$. Then the subspace
$\bar\nu_\xi(\widetilde\gamma)$ is homeomorphic to $D^2\times [0,1]$
where the (closed) disks in this (topological) product structure are
the level sets of $f_{\widetilde\gamma}$.  The subset
$\nu_\xi^0(\widetilde\gamma)$ is homeomorphic to $S^1\times
[0,1]\times (0,1)$ where each circle factor is the intersection of a
level set of $f_{\widetilde\gamma}$ with a level set of
$h_{\widetilde\gamma}$. (These intersections are called {\em level
circles}.)
\end{prop}

\begin{proof}
We work with the metric $\lambda g_n$, so that in particular,
$\ell(\widetilde \gamma)$ means the length of $\widetilde \gamma$
with respect to this metric. Provided that $\hat\epsilon$ is sufficiently small,
it follows from Lemma~\ref{goodrect} that $f_{\widetilde\gamma}$ is
$20\xi$-regular on $\bar\nu_\xi(\widetilde\gamma)$ so that, provided
that $\xi$ is sufficiently small, each level set $L$ is a Lipschitz
surface and these level surfaces foliate $\bar\nu_\xi(\widetilde\gamma)$.
The conditions on the restrictions of $f_\gamma$ and $h_\gamma$ to
$\nu^0_\xi(\gamma)$ imply by arguments identical to the ones given
in the proof of Lemma~\ref{semilocal} that the map
$F=(f_{\widetilde\gamma},h_{\widetilde \gamma})$ determines a
fibration of $\nu^0_\xi(\widetilde \gamma)$ with fibers circles.
Hence, $\nu^0_\xi(\widetilde \gamma)$ is homeomorphic to $S^1\times
[0,1]\times (0,1)$ where the circles are the level circles.

We shall show that provided that $\xi>0$ is sufficiently small, the level sets of
$f_{\widetilde\gamma}$ are homeomorphic to disks. From the immediately preceding  discussion, it
follows that the boundary of any level surface for $f_{\widetilde \gamma}$ is a single circle.
Since the level sets of $f_{\widetilde\gamma}$ are connected, to show these level sets are
homeomorphic to disks it suffices to show that  they have virtually nilpotent fundamental groups
and are orientable. The level sets are orientable since $M_n$ is and since they are the level sets
of a regular Lipschitz function so that there is a neighborhood of the level set in $M_n$ that is
homeomorphic to the product of the level set with $I$.

\begin{claim} For $\xi>0$ sufficiently small, the
fundamental groups of the level sets of $f_\gamma$ are virtually
abelian.
\end{claim}

Of course, the image in the fundamental group of the manifold of the fundamental group of a level
set is the same as the image of the fundamental group of the $\xi$-neighborhood of a point. The
argument we give here is a simplification of the argument in the appendix of \cite{FY} proving a
more general result.

\begin{proof}
We suppose that the claim does not hold. Then there is a sequence of $\xi_k$ and counter-examples
$\nu_\xi(\widetilde \gamma_k)$. Take as base points $p_k$ the midpoints of $\widetilde \gamma_k$.
According to what we just established every point of $\nu_{\xi_k}^0(\widetilde\gamma_k)$ is the
center of an $S^1$-product neighborhood with $\epsilon$-control. Thus, there are fixed constants
$\epsilon>0$ and $\delta>0$, independent of $k$, such that the ball of radius $\epsilon$ about
$p_k$ contains a ball of radius $\delta$ about some point $q_k$ with the property that the image
$\Gamma_k^0$ of $\pi_1(B(q_k,\delta))$ in $\Gamma_k=\pi_1(B(p_k,\epsilon))$ is either trivial or
infinite cyclic. Notice that, being the fundamental group of a surface with non-empty boundary,
$\Gamma_k$ is either cyclic or a free group
of rank at least $2$. We are supposing that for no $k$ is $\Gamma_k$ cyclic. Hence, for all $k$,
the group $\Gamma_k$ is a non-abelian free group. Thus, any non-abelian subgroup of $\Gamma_k$ is
automatically free. Also, the inclusion of  the ball $B(p_k,\xi_k)\to B(p_k,\epsilon)$ is
surjective on fundamental groups. Consider the universal covering $\widetilde B$ of
$B(p_k,\epsilon)$ and a lift $\tilde p_k$ of $p_k$. Also, fix a component $\widetilde B_0$ of the
preimage $B(q_k,\delta)$ containing a lift $\tilde q_k$ of $q_k$ at distance $d_k\le \epsilon$
from $\tilde p_k$. For any subgroup $F\subset \Gamma_k$ containing $\Gamma_k^0$, the set of
components of $\widetilde B_0\cdot F$ is in natural one-to-one correspondence with the cosets
$\Gamma_k^0\backslash F$. The group $\Gamma_k$ is generated by elements $\gamma\in \Gamma_k$ with
the property that $B(\tilde p_k,\xi_k)\cap B(\tilde p_k,\xi_k)\cdot \gamma\not=\emptyset.$ All
these generators are represented by loops based at $p_k$ of length at most $2\xi_k$. Since these
elements generate the non-abelian free group there are two of them, $\alpha,\beta$ that do not
commute, and hence themselves generate a free group $F=F\langle\alpha,\beta\rangle$ of rank $2$.
Each of these elements moves $\tilde p_k$ by a distance at most $2\xi_k$, and hence their inverses
move $\tilde p_k$ by distance at most $2\xi_k$. Thus, we have at least $2\cdot 3^{N-1}$ translates
of $\widetilde B_0$ within distance $2N\xi_0$ of $\widetilde B_0$. (Distinct irreducible words in
$\alpha^{\pm},\beta^{\pm}$ that end with a letter that does not cancel with the first letter of the two
inverse words generating the $\Gamma_k^0$ represent distinct cosets of $\Gamma_k^0$.) Choose
geodesics $\mu_i$ from $\tilde p_k$ to each the translates of $\tilde q_k$. Each of these $\mu_i$
has length at most $2N\xi_k+2\epsilon$. We fix $N=[\epsilon/\xi_k]$, so that this length is at
most $4\epsilon$. Since $\widetilde B_0$ contains the $\delta$ ball about $\tilde q_k$, it follows
that the endpoints of the $\mu_i$ are separated by at least $2\delta$. Since the curvature of the
$B(p_k,\epsilon)$ is bounded below by $-1$, by the comparison result and monotonicity (see Remark 3.5 of \cite{BGP}), this implies that there is a
constant $\delta'>0$ depending only on $\delta$ and $\epsilon$ so that the angles between the
$\mu_i$ at $\tilde p_k$ is at least $\delta'$. The number of these geodesics is $2\cdot
3^{([\epsilon/\xi_k]-1)}$, which tends to infinity as $k\rightarrow \infty$. But this is
impossible; given $\delta'>0$ there is an upper bound $N(\delta')$ to the number of geodesics
emanating from a point of a $3$-manifold mutually separated from each other by angle at least
$\delta'$.
\end{proof}
This completes the proof of the fact that the level sets of
$f_\gamma$ are $2$-disks, and hence completes the proof of the
proposition.
\end{proof}

We also need to understand the relationship of
$\nu^0_\xi(\widetilde\gamma)$ with the $S^1$-product neighborhoods.

\begin{lem}\label{flatboundary1}
With notation and assumptions as in the previous proposition the following hold. \begin{enumerate}
\item  Each point of $\nu_\xi^0(\widetilde\gamma)$ is the center of an $S^1$-product neighborhood
with $\epsilon$-control.
\item For any point $x\in \nu^0_\xi(\widetilde\gamma)$,
the $S^1$-product structure with $\epsilon$-control centered at $x$, $\varphi\colon
S^1\times B(0,\epsilon^{-1})\to M_n$,  can be chosen so that the following hold for
$R\le\epsilon^{-1}/2$ and any point $q\in \varphi(S^1\times B(0,R))$:
\begin{enumerate}
\item For any geodesic $\zeta$ from $\gamma$ to a point $q$,
$\varphi^{-1}$ of intersection of $\zeta$ with the neighborhood is within $\epsilon'$ of the
straight line starting at $q$ in the negative $y$-direction in the $\Ar^2$factor.
\item For any geodesics $\zeta_\pm$ from $e_\pm(\gamma)$ to $q$, $\varphi^{-1}$
of the intersection of $\zeta_\pm$ with the neighborhood are within $\epsilon'$ of straight lines
starting at $q$ in the $x_\pm$-directions.
\end{enumerate}
\item For any
$q\in\nu_\xi^0(\widetilde\gamma)$   and for any $S^1$-product neighborhood with $\epsilon$-control
containing $q$, the angle at
$q$ between the level circle $S(q)$ through
$q$ and the horizontal space of the $S^1$-product
neighborhood is within $\epsilon'$ of $\pi/2$. Furthermore, if $q$
is contained in $\varphi(S^1\times B(0,R))$, then
$S(q)$ is isotopic in the $S^1$-product neighborhood to an
$S^1$-factor.
\end{enumerate}
\end{lem}

\begin{proof}
For any  point $x\in\bar\nu_\xi^0(\widetilde\gamma)$ the ball
$B_{\lambda g_n}(x,1)$ is within $\hat\epsilon$ of an Alexandrov
ball $B(\bar x,1)$  that is interior $\mu$-flat at $\bar x$ on scale
$\xi^2s_1/20$. It follows from Lemma~\ref{generic2D} that for all
$\hat\epsilon$ sufficiently small, that every point of
$\nu_\xi^0(\widetilde\gamma)$ is the center of an $S^1$-product
neighborhood with $\epsilon$-control. It then follows from the last
statement in Lemma~\ref{generic2D} that we can choose the Euclidean
coordinates for this $S^1$-product structure so that any geodesics
from $e_\pm$ is within $\epsilon'$ of the straight line in the
horizontal space in the $x_\pm$-direction. Now consider any geodesic
from $\widetilde\gamma$ to a point of this product neighborhood. It
makes angle near to $\pi/2$ with any geodesic from $e_\pm$ and also
is almost horizontal since its length is much longer that the
diameter of the circle factors. It follows that (possibly after
reversing the $y$-coordinate) these geodesics are within $\epsilon'$
of the $y_-$-direction in the horizontal spaces. We consider
$F=(f_{\widetilde\gamma},h_{\widetilde\gamma})$ mapping
this $S^1$-product neighborhood to $\Ar^2$. The restriction of $f_{\widetilde\gamma}$
to any horizontal space  is strictly increasing in the $x$-direction
and in fact $|f_{\widetilde\gamma}(x_0,y)-f_{\widetilde\gamma}(x_1,y)|\ge (1-\epsilon')|x_0-x_1|$,
whereas $|f_{\widetilde\gamma}(x,y_0)-f_{\widetilde\gamma}(x,y_1)|< \epsilon'|y_0-y_1|$. The second
coordinate function $h_{\widetilde\gamma}$ satisfies analogous inequalities with the roles
of $x$ and $y$ reversed. It follows that the restriction of $F$ to
any horizontal space is one-to-one and hence the level sets of $F$
meet each the horizontal space in at most one point. Furthermore,
the fact that the geodesics from $\widetilde\gamma$ and from $e_\pm$
are nearly horizontal implies, by \S 11 of \cite{BGP} that the level
sets of $F$ are nearly orthogonal to the horizontal spaces in the
sense described in the statement of the proposition. Choosing
$\hat\epsilon$ sufficiently small, we can arrange that these level
sets make angle within $\epsilon'$ of $\pi/2$ with every horizontal
space. It follows that any level set of $F$ that meets
$\varphi\left(S^1\times B(0,R)\right)$ is contained in
$\varphi\left(S^1\times B(0,\epsilon^{-1})\right)$ and is a circle
meeting each horizontal space once. Such circles are
isotopic in the neighborhood to the circle factors.
\end{proof}

\begin{defn}
We call any neighborhood $\nu_\xi(\widetilde\gamma)$ for which there is
a geodesic $\gamma$ in a $2$-dimensional standard ball satisfying the
hypotheses of Proposition~\ref{flatboundary} (and hence $\nu_\xi(\widetilde \gamma)$
satisfies the conclusions
of the last two results) {\em an
$\epsilon'$-solid cylinder neighborhood at scale $s_1$ near a flat boundary,} or
simply an {\em $\epsilon'$-solid cylinder neighborhood at scale $s_1$} for short.
\end{defn}

For later use we need one final addendum about the $\nu_\xi(\widetilde \gamma)$. It follows
directly from the corresponding statement in the $2$-dimensional case (Lemma~\ref{goodrect}.)

\begin{lem}\label{flatboundary2}
With notation and assumptions as in the previous proposition the
following holds. For any $c\in [\xi^2,\xi]$ and for any level
surface $L$ of $f_{\widetilde\gamma}$ the distance from any point of
$L\cap h_{\widetilde\gamma}^{-1}(c\cdot\ell(\widetilde\gamma))$ to
$L\cap \widetilde\gamma$ is at most
$(1+4\xi)c\cdot\ell(\widetilde\gamma)$. Also, for any point $y\in
\nu_\xi^0(\widetilde \gamma)$ there is a geodesic $\zeta$ of length
$10\xi$ from $y$ to a point $z$ such that for any $w\in \zeta$ at
distance  at most $5\xi\ell(\widetilde \gamma)$ from $y$ the
comparison angle $\widetilde\gamma wz$ is greater than $\pi-2\xi$.
(Here all distances and $\ell(\widetilde\gamma)$ are measured with
respect to $\lambda g_n$.)
\end{lem}

We shall also need smooth vector fields well-adapted to
$\nu_\xi(\widetilde\gamma)$.

\begin{cor}\label{3dvf}
There is a smooth unit vector field $\tilde\tau$ on $\nu_\xi(\widetilde\gamma)$  such that,
setting $d_\pm$ equal to the distance function from the endpoints $e_\pm$ of $\widetilde \gamma$,
we have $d'_-(\widetilde\tau)>1-10\xi$, $d'_+(\tilde\tau)>-1+10\xi$, and
$h_{\widetilde\gamma}'(\tilde\tau)<c\xi^2$ for a universal constant $c$.
 Provided that $\xi$ is
sufficiently small, for any points $p,q$ on a flow line of the flow generated by $\tilde\tau$ we
have
$$\left|\frac{h_{\widetilde\gamma}(p)-h_{\widetilde\gamma}(q)}
{f_{\widetilde\gamma}(p)-f_{\widetilde\gamma}(q)}\right|<2c\xi^2.$$
In particular, for $\xi>0$ sufficiently small, any maximal flow line
of $\widetilde \tau$ that meets $\nu_{3\xi/4}(\widetilde\gamma)$ is
an interval with endpoints in the ends of
$\nu_\xi(\widetilde\gamma)$ and this interval meets each level set
of $f_{\widetilde\gamma}$ in a single point.
\end{cor}

\begin{proof}
The existence of $\tilde\tau$  as stated follows immediately from the definition of a $\xi$-approximation
and Lemma~\ref{smreg}. The last statements then follow easily.
\end{proof}

\begin{defn}
The metric $\lambda g_n$ that was used in the previous proposition
is called the {\em metric used to define the neighborhood}
$\nu(\widetilde \gamma)$. By $\ell(\widetilde \gamma)$ we always
mean the length of the geodesic $\widetilde \gamma$ with respect to
the metric used to define the neighborhood. By a {\em spanning disk} in an $\epsilon'$-solid cylinder we mean a $2$-disk with boundary contained in the side of the solid cylinder that separates the ends of the solid cylinder.
\end{defn}

\subsubsection{Intersections of the $\nu_\xi(\widetilde\gamma)$}

It is important to have control over how the various
$\epsilon'$-solid cylinder neighborhoods near a flat boundary
intersect.

\begin{lem}\label{nuint} The following hold for all $\xi>0$ sufficiently small,
for all $\mu>0$ less than a positive constant $\mu_5(\xi)$, for
 any $0<s_1\le 1/4$, and,
 for all $\hat\epsilon>0$ less than a positive constant
$\hat\epsilon_3(\xi,s_1)$. For $i=1,2$, let $X_i=B(\bar x_i,1)$ be
standard $2$-dimensional balls and let $\gamma_i\subset B(\bar
x_i,3/4)$  be $\xi$-approximations to $\partial X_i$ on scale $s_1$
with $\mu$-control with $\bar y_i\in B(\bar x_i,1)$ being the
$\mu$-control point for $\gamma_i$.
 Suppose we have points $x_i\in
M_n$  with the property that $B_{ g'_n(x_i)}(x_i,1)$ is within
$\hat\epsilon$ of $B(\bar x_i,1)$ for $i=1,2$.
 Furthermore, suppose that
$\widetilde\gamma_i\subset B_{g'_n(x_i)}(x_i,1)$ is within
$\hat\epsilon$ of $\gamma_i$.
 Denote by $\ell_i$ the length of $\widetilde\gamma_i$ as measured
in the metric $g'_n(x_i)$. Then the  intersection of
$\widetilde\gamma_2$ with $\nu_\xi(\widetilde\gamma_1)$ is contained
in $\nu_{\xi^2}(\widetilde\gamma_1)$ and $\widetilde\gamma_1$ meets each level set of $f_{\widetilde\gamma_2}$
in at most one point. In particular, for any $c\ge
\xi$ the intersection of $\widetilde\gamma_2$ with the boundary of
$\bar\nu_{c\xi}(\widetilde\gamma_1)$ is contained in the ends of
this neighborhood.
\end{lem}

\begin{proof}
Fix $\xi>0$ and $0<s_1\le 1/4$. Let $y_1\in B_{g'_n(x_1)}(x_1,1)$
and $y_2\in B_{g'_n(x_2)}(x_2,1)$ be points within $\hat\epsilon$ of
$\bar y_1$ and $\bar y_2$ respectively. Since
$\nu_\xi(\widetilde\gamma_{1})\cap
\nu_\xi(\widetilde\gamma_{2})\not=\emptyset$, it follows that
$\rho_n(x_1)/\rho_n(x_2)$ is between $1/2$ and $2$. The ball
$B_{g'_n(x_1)}(y_1,s_1)$ contains $\nu_\xi(\widetilde\gamma_1)$ and
hence contains a point of $\nu_\xi(\widetilde\gamma_2)$. The length
of $\widetilde\gamma_2$ with respect to $g'_n(x_2)$ is between
$s_1/10$ and $s_1$, so its length with respect to $g'_n(x_1)$ is
between $s_1/20$ and $2s_1$. If follows that
$\widetilde\gamma_2\subset B_{g'_n(x_1)}(y_1,4s_1)$. Similarly,
$B_{g'_n(x_2)}(y_2,s_1)\subset B_{g'_n(x_1)}(y_1,5s_1)$.

Now suppose that we have a sequence of $\mu_n\rightarrow 0$ and
$\hat\epsilon_n\rightarrow 0$ and counterexamples to the result for
each of these constants. Denote  the $2$-dimensional balls
associated to these counterexamples by $B(\bar x_{n,1},1)$ and
$B(\bar x_{n,2},1)$ and denote the $\mu_n$-control points by $\bar
y_{n,1}$ and $\bar y_{n,2}$. From the above we see that $B(\bar
y_{n,2},s_1)$ is within $3\hat\epsilon_n$ of a sub-ball of radius
between $s_1/2$ and $2s_1$ in $B(\bar y_{n,1},5s_1)$. Passing to a
subsequence so that limits $B(\bar y_{\infty,1},5s_1)$ and $B(\bar
y_{\infty,2},s_1)$ exist, we see that, taking the limit as $n$ tends
to infinity, we see that $B(\bar y_{\infty,2},s_1)$ is identified
with a sub-ball of $B(\bar y_{\infty,1},5s_1)$. On the other hand
since the $\mu_n\rightarrow 0$, both $B(\bar y_{\infty,1},5s_1)$ and
$B(\bar y_{\infty,2},s_1)$ are sub-balls of $[0,\infty)\times \Ar$
and the limiting geodesics $\gamma_{\infty,1}$ and
$\gamma_{\infty,2}$ are geodesics in the boundary. Hence, the
intersection of $\gamma_{\infty,2}$ with
$\nu_\xi(\gamma_{\infty,1})$ is contained in $\gamma_{\infty,1}$.
This is a contradiction, establishing the result.
\end{proof}

We also need estimates about the vector fields from Lemma~\ref{3dvf}
and also about the distances between the sides of the neighborhoods.

\begin{lem}\label{nuint1}
With notation and assumption as in the previous lemma the following
hold, provided that $\xi>0$ is sufficiently small.
\begin{enumerate}
\item For a unit vector field $\tilde \tau_1$ on
$\nu_\xi(\widetilde\gamma_1)$ satisfying Corollary~\ref{3dvf}, at
any point of $\bar\nu_\xi(\widetilde \gamma_1)\cap
\bar\nu_\xi(\widetilde\gamma_2)$ we have
$$\|f'_{\widetilde\gamma_2}(\tilde\tau_1)\|>1-20\xi.$$
\item For any constants $c_1,c_2$ with
$2\xi\le c_i\le 3/4$ and with
$$c_1\ell_1\rho_n(x_1)<(0.9)c_2\ell_2\rho_n(x_2)$$ each
level set of $f_{\widetilde\gamma_2}$ in $\bar\nu_{c_2\xi}(\widetilde\gamma_2)$ that meets
$\bar\nu_{\xi,[-.24\ell_1,.24\ell_1]}(\widetilde\gamma_1)$ meets
$\bar\nu_{c_1\xi}(\widetilde\gamma_1)$ in a disk whose boundary is contained in the side of
$\bar\nu_{c_1\xi}(\widetilde\gamma_1)$, a disk that separates the ends of
$\bar\nu_{c_1\xi}(\widetilde\gamma_1)$.
\end{enumerate}
\end{lem}

\begin{proof}
The first statement is immediate  from Corollary~\ref{flatint}.
 It follows immediately from this that
any level set of $f_{\widetilde\gamma_{n,2}}$ meets each flow line
for $\tilde\tau_1$ in at most one point.

Now let us establish the second statement. Let $y$ be a point in
$$\nu_\xi(\widetilde\gamma_2)\cap\nu_{\xi,[-(.24)\ell_1,(.24)\ell]}(\widetilde\gamma_1),$$
and consider the level surface $L$ for $f_{\widetilde\gamma_2}$
through $y$. It follows from Corollary~\ref{levelset} that, provided
that $\xi>0$ is sufficiently small, the variation of
$f_{\widetilde\gamma_1}$ on $L\cap \nu_\xi(\widetilde\gamma_1)$ is
less than $(0.001)\ell_1$. This implies that $L$ does not meet the
ends of $\bar\nu_\xi(\widetilde\gamma_1)$. Thus, under the given
assumptions on $c_1$ and $c_2$ we see that $L\cap
\left(h_{\widetilde\gamma_2}^{-1}([0,c_2\xi])\right)$ crosses the
side of $\bar\nu_{c_1\xi}(\widetilde\gamma_1)$.

Let us consider the intersection of $L$ with
$$U=\nu_{c_1\xi}(\widetilde\gamma_1)\setminus\nu_{\xi^2}(\widetilde\gamma_1).$$
On $U$ the functions $f_{\widetilde\gamma_2}$ and
$h_{\widetilde\gamma_1}$ satisfy Lemma~\ref{2dregular} and hence the
intersection of the level sets of these functions are circles that
are almost orthogonal to the horizontal spaces in $S^1$-product
neighborhoods with $\epsilon$-control, circles that meet each of these
horizontal spaces in a single point. This means that $L\cap U$ is
homeomorphic to $S^1\times (0,1)$ and is foliated by circles which
are the intersections of $L$ with level sets of
$h_{\widetilde\gamma_1}$. Now we fix the circle $C=L\cap
h_{\widetilde\gamma_{n,1}}^{-1}(c_1\ell_1/2)$ and let $D\subset L$
be the surface bounded by $C$. Let $z\in C$. We flow $D$ along the
flow lines of $\tilde\tau_1$ to a level set $L'=
f_{\widetilde\gamma_1}^{-1}(b)$ where $b$ is chosen so that
$|b-f_{\widetilde\gamma_1}(z)|=4\xi^2\ell_1$. By
Corollary~\ref{levelset} for all $n$ sufficiently large,
$f_{\widetilde\gamma_1}$ varies by less than $2\xi^2\ell_1$ on $L'$.
By Corollary~\ref{levelset} $L'\cap L=\emptyset$ and the maximal
difference between the values of $f_{\widetilde\gamma_1}$ on $L$ and
$L'$ is at most $4\xi^2\ell_1$. Thus, any flow line for $\tilde
\tau$ that starts on $D$ stays in $\nu_{c_1\xi}(\widetilde\gamma_1)$
as it flows from $D$ to $L'$. Thus, deforming along these flow lines
gives a topological embedding of $D$ as a subsurface of $L'$. Since
$\partial D=C$ is a single circle and since $L'$ is a disk, it
follows that $D$ is a topological disk and hence $L$ is also a disk.
Clearly, it separates the ends of
$\nu_{c_1\xi}(\widetilde\gamma_1)$.
\end{proof}

\begin{addendum}
In the previous two lemmas, we assumed the metrics were $g_n'(x_1)$ and
$g'_n(x_2)$. The reason for this was that if
$B_{g'_n(x_i)}B(x_i,1)$ have non-trivial intersection then the
metrics are within a multiplicative factor of $2$ of each other. We
also have analogous results when we use a fixed multiple $\lambda
g_n$ as the metric in two balls. The proofs are identical, since
this time the metrics agree.
\end{addendum}

\subsection{Boundary points of angle $\le \pi-\delta$}

\begin{prop}\label{corners}
For all $a>0$, given $a$ for all $\mu>0$ sufficiently small, for all
$0<r\le 10^{-3}$, and, given $a, \mu,$ and $r$, for all
$\hat\epsilon>0$ sufficiently small the following hold. Suppose
that, for some $n$ there is a point $x_n\in M_n$ with the property
that $B_{\lambda g_n}(x_n,1)$ is within $\hat\epsilon$ of a
$2$-dimensional Alexandrov space $X=B(\bar x,1)$ of area $\ge a$
that is boundary $\mu$-good at $\bar x$ on scale $r$. Then the ball
$B_{\lambda g_n}(x_n,7r/8)$ is a topological $3$-ball and the
distance function, $d(x_n,\cdot)$, is $(1-\beta)$-strongly regular
on $B_{\lambda g_n}(x_n,7r/8)\setminus B_{\lambda g_n}(x_n,r/8)$,
where $\beta=\beta(\mu)$ limits to zero as $\mu$ tends to zero.
\end{prop}

\begin{proof}
Fix $a>0$. Suppose that there is a sequence $\mu_k\rightarrow 0$ as
$k\rightarrow \infty$ for which the result does not hold, meaning
that for each $k$ there is $0<r_k\le 10^{-3}$ for which there is no
constant $\hat\epsilon$ as required. This implies that for each $k$
there is a sequence of constants $\hat\epsilon_{k,\ell}$ tending to
zero and counter-examples
$B_{\lambda_{k,\ell}g_{n(k,\ell)}}(x_{n(k,\ell)},1)$ with these
values of the constants. The  balls
$B_{\lambda_{k,\ell}g_{n(k,\ell)}}(x_{k,\ell},1)$ that are within
$\hat\epsilon_{k,\ell}$ of standard $2$-dimensional balls $B(\bar
x_{k,\ell},1)$ of  area at least $a$, balls that are boundary
$\mu_k$-good at $\bar x_{k,\ell}$ on scale $r_k$. The ball
$B_{(\lambda_{k,\ell}/(r_k)^2)g_{n(k,\ell)}}(x_{k,\ell},1)$ is
within $\hat\epsilon_{k,\ell}/r_k$ of the unit ball centered at
$\bar x_{k,\ell(k)}$ in $(1/r_k)B(\bar x_{k,\ell},r_k)$, and the
latter is boundary $\mu_k$-good at $\bar x_{k,\ell}$ on scale $1$.
For each $k$ choose $\ell(k)$ sufficiently large such that
$\hat\epsilon_{k,\ell(k)}/r_k\rightarrow 0$. (We shall add another
condition on how large $\ell(k)$ must be later in the argument.)
Re-index the constants by $k$ so that, for example,
$\lambda_{k,\ell(k)}$ is denoted $\lambda_k$. Passing to a
subsequence and taking a limit, the fact that the $\mu_k\rightarrow
0$ implies that the $2$-dimensional unit balls converge to a flat
cone $C$ in $\Ar^2$ of angle $\le \pi$. The area of $C$ is bounded
below by a positive constant $a'$ depending only on $a$. By our
choice of $\hat\epsilon_k$, it follows that the balls
$B_{(\lambda_k/(r_k)^2)g_{n(k)}}(x_k,1)$ also converge to $C$, with
the $x_k$ converging to the cone point. It follows that given any
$\zeta>0$ for all $k$ sufficiently large, the distance function
$d_k=d_{(\lambda_k/(r_k)^2)g_{n(k)}}(x_k,\cdot)$ is
$(1-\zeta)$-regular on
$B_{(\lambda_k/(r_k)^2)g_{n(k)}}(x_k,1)\setminus
B_{(\lambda_k/(r_k)^2)g_{n(k)}}(x_k,\zeta)$, and in particular this
annular region is homeomorphic to a product with an interval and the
slices of the product structure are the level sets of the distance
function. We shall achieve a contradiction by showing that these
level sets are $2$-spheres and that the metric balls that they bound
are homeomorphic to $3$-balls.

At this point in the argument, we fix $\epsilon'>0$ sufficiently
small and let $\epsilon<\epsilon_1(\epsilon')$ as in
Proposition~\ref{s1glue}. Then we fix $\xi>0$ sufficiently small so
that Proposition~\ref{flatboundary}. Lastly, we also assume that for
each $k$ we have chosen $\ell(k)$ sufficiently large so that
$\hat\epsilon_{k,\ell(k)}<\hat\epsilon_2(\epsilon,\mu_k,1/4)$ for
Proposition~\ref{flatboundary}.
We consider first the case when the cone angle at the cone point of
$C$ is $\pi$. In this case, $C$ is isometric to a unit ball centered
at a boundary point of $\Ar\times [0,\infty)$.
 Since
$\mu_k\rightarrow 0$  by Proposition~\ref{flatboundary} and our
choice of $\ell(k)$, there is a constant $\zeta>0$ such that for all
$k$ sufficiently large there neighborhood of $x_k$ containing
$B_{(\lambda_k/(r_k)^2)g_{n(k)}}(x_k,\zeta)$ and contained in
$B_{(\lambda_k/(r_k)^2)g_{n(k)}}(x_k,1)$ that is homeomorphic to
$D^2\times I$. The boundary of this neighborhood, which is a
$2$-sphere, separates the level set for $d_k$ at distance $\zeta$
from the level set at distance $1$. It follows that all the level
sets are $S^2$-spheres. Furthermore, since the level set at distance
$\zeta$ is a $2$-sphere contained in a neighborhood of $x_k$
homeomorphic to a $3$-ball, this level set bounds a $3$-ball in this
neighborhood. It follows immediately that for all $k$ sufficiently
large, all the metric balls $B_{(\lambda_k/(r_k)^2)g_{n(k)}}(x_k,t)$
for $\zeta\le t\le 1$ are homeomorphic to $3$-balls. This is a
contradiction, proving the result in this case.

We now examine the case when the cone angle of $C$ is strictly less
than $\pi$. For the rest of the argument we implicitly use the
metric $(\lambda_k/(r_k)^2)g_{n(k)}$ on $M_{n(k)}$ and on all its
subsets.
 According to
the Proposition~\ref{blowup} there is a sequence of points
$x_{n(k)}'\in M_{n(k)}$ with $d(x_{n(k)}',x_{n(k)})\rightarrow 0$
such that one of the two following cases holds:
\begin{enumerate}
\item the distance function from $x'_{n(k)}$ has no critical points on
$B(x'_{n(k)},1)\setminus\{x_{n(k)}'\}$, or
\item  there is a sequence $\delta_k\rightarrow 0$ such that the
distance function from $x_{n(k)}'$ has no critical points at
distances between $\delta_k$ and $1$ and has a critical point at
distance $\delta_k$.
\end{enumerate}
In the first case, all the level sets for the distance function from
$x_{n(k)}'$ at distance strictly between $0$ and $1$ are $2$-spheres
and the corresponding metric balls are homeomorphic to $3$-balls. In
the second case, rescaling by $\delta_k^{-2}$ we get a sequence of
$3$-manifolds with a subsequence converging to a $3$-dimensional
Alexandrov space of curvature $\ge 0$. By Proposition~\ref{smlimits}
the convergence is in fact a smooth convergence and the limit is a
smooth complete $3$-manifold of non-negative curvature. It follows
that for all $k$ sufficiently large, one of these two possibilities
holds for $X_k$.

\begin{claim}
The level sets of the distance function $f_k=d(x_{n(k)}',\cdot)$ at
distance between $\delta_k$ and $1$ are topological $2$-spheres.
\end{claim}

Let us assume this claim for a moment and complete the proof of the
lemma. It follows from this claim that the end of the limiting
manifold is homeomorphic to $S^2\times [0,\infty)$. The limiting
manifold has a soul which is a manifold of non-negative curvature.
Because  the neighborhood of infinity of the limit is diffeomorphic
to $S^2\times [0,\infty)$, the soul must be either a point or $\Ar
P^2$. The second case is not possible, since in this case the
original manifolds would converge to an interval not a
$2$-dimensional Alexandrov space of area $\ge a$. Since its soul is
a point, the limiting manifold is diffeomorphic to $\Ar^3$. The
result is then immediate.

It remains to prove the claim.

\begin{proof} (of the claim) We know that
$$f_k\colon B_{(\lambda_k/(r_k)^2)
g_{n(k)}}(x_{n(k)},3/4)\setminus B_{(\lambda_k/(r_k)^2)
g_{n(k)}}(x_{n(k)},1/4)\to (1/4,3/4)$$ is the projection mapping of a locally
trivial fibration.  Set $b^+=1/2$ and $b^-=(.46)$.  Then using the metric $(\lambda_k/(r_k)^2)g_{n(k)}$,
 by our choice of
constants and Proposition~\ref{flatboundary} there are good
$\xi$-approximations $\nu_\xi(\gamma)$ and $\nu_\xi(\gamma')$ to the
boundary of length $1/10$ centered at the two points of $\partial
B$ at distance $(.48)$ from $\bar x$.
 Let
$\nu_\xi(\widetilde\gamma_k)$ and $\nu_\xi(\widetilde\gamma_k')$ be
corresponding neighborhoods in $M_{n(k)}$. For $k$ sufficiently
large we can choose these geodesics within $\epsilon'$ of $\gamma$
and $\gamma'$. Every point of the open subset $U$ which is the
intersection of
$$B_{(\lambda_k/r^2_k) g_{n(k)}}(x_{n(k)},b^+)\setminus\bar B_{(\lambda_k/r_k^2)
g_{n(k)}}(x_{n(k)},b^-)$$ with the complement of the closure of
$\nu_{\xi^2}(\widetilde\gamma_k)\cup
\nu_{\xi^2}(\widetilde\gamma_k')$ is the center of an $S^1$-product
structure with $\epsilon$-control. Hence, this subset sits inside a
larger open subset that is the total space of an $S^1$-fibration
with fibers within $\epsilon'$ of orthogonal to the horizontal
spaces of the $S^1$-product structures with $\epsilon$-control. This
implies that there is an annulus in $U$ with boundary contained in
$\nu_\xi(\widetilde\gamma_k)\cup\nu_\xi(\widetilde\gamma_k')$ that
separates $f_k^{-1}(b^+)\cap U$ from $f_k^{-1}(b^-)\cap U$. Since the
boundary loops of this annulus are homotopically trivial in
$\nu_\xi(\widetilde\gamma_k)\cup \nu_\xi(\widetilde\gamma_k')$, it
follows that there is a map of $S^2$ into $B_{\lambda_k
g_{n(k)}}(x_{n(k)},b')\setminus B_{\lambda_k g_{n(k)}}(x_{n(k)},b)$
that is homologically non-trivial. The claim follows.
\end{proof}
\end{proof}

This argument actually proves more.

\begin{cor}\label{corners1}
Fix $\epsilon'>0$ and let $\epsilon>0$ be less than the constant
$\epsilon_1(\epsilon')$ as in Proposition~\ref{s1glue}. For all
$\xi>0$ sufficiently small and every $a>0$  the following holds
for all $\mu$ less than a positive constant $\mu_6(\xi,a)$, for
every $r>0$ and for all $\hat\epsilon$ less than a positive constant
$\hat\epsilon_4(\epsilon,a,\mu,r)$. With the notation and
assumptions of the previous proposition, fix $b\in (r'/8,7r'/8)$.
The level set $L_b=d(x_n,\cdot)^{-1}(b)$ is a topologically locally
flat $2$-sphere and the metric ball that it bounds is a topological $3$-ball.
Furthermore, there are two geodesics
$\widetilde\gamma_1$ and $\widetilde\gamma_2$ within $\epsilon$ of
geodesics $\gamma_1$ and $\gamma_2$ in $X$ that are
$\xi$-approximations to $\partial X$ on scale $s_1$ with the
property that every point of $L_b$ that is not the center of an
$S^1$-product neighborhood with $\epsilon$-control is contained in
union $\nu_{\xi^2}(\widetilde\gamma_1)\cup \nu_{\xi^2}(\widetilde
\gamma_2)$.
\end{cor}

Actually, we have more control over the intersections of the level
sets with the $\nu_\xi(\widetilde\gamma_i)$.

\begin{cor}\label{corners'}
With notation and assumptions as in the previous corollary, for any
$b\in (r'/8,7r'/8)$, and any the level set $L_b=d_{\lambda
g_n}(x_n,\cdot)^{-1}(b)$ meets $\bar\nu_\xi(\widetilde\gamma_i)$ in
a spanning $2$-disk in
$\bar\nu_\xi(\widetilde\gamma_i).$ Furthermore, for any $c\in
[\xi,1]$ the level set
$h_{\widetilde\gamma_i}^{-1}(c\xi\ell(\widetilde\gamma_i))$ crosses
$L_b$ topologically transversally and the intersection is a circle
bounding the disk $L_b\cap\bar\nu_{c\xi}(\widetilde\gamma_i)$.
\end{cor}

\begin{proof}
Let $f$ denote the distance function from $x_n$. By
Corollary~\ref{flat/ball}, and the fact that $B_{\lambda
g_n}(x_n,1)$ is within $\hat\epsilon$ of a standard $2$-dimensional
ball $B(\bar x,1)$ that is boundary $\mu$-good at $\bar x$, it
follows that $f$ is less than $b$ on one end of
$\nu_\xi(\widetilde\gamma_i)$ and greater than $b$ on the other end,
so that $L_b\cap \nu_\xi(\widetilde\gamma_i)$ separates the ends of
$\nu_\xi(\widetilde\gamma_i)$. Also, by the same lemma the function
$f$ is increasing on the flow lines of $\widetilde\tau$ is in
Corollary~\ref{3dvf} so that $L_b\cap \nu_\xi(\widetilde\gamma_i)$
is transverse to these flow lines. Furthermore, it follows from
Corollary~\ref{3dvf} that any flow line from a point of
$L_b\cap\nu_{\xi/2}(\widetilde\gamma_i)$  remains in
$\nu_\xi(\widetilde\gamma_i)$ until it crosses the end of this
region. Thus, flowing along these flow lines gives us an embedding
of $L_b\cap \nu_{\xi/2}(\widetilde\gamma_i)$ into a level disk for
$f_{\widetilde\gamma_i}$ and hence embeds this surface as a
subsurface $\Sigma$ of a disk. It follows from \S 11 of \cite{BGP}
applied to the functions $h_{\widetilde\gamma_i}$ and $d(x_0,\cdot)$
restricted to $\nu^0_\xi(\widetilde\gamma_i)$ that the intersection
$L_b\cap
\partial \nu^0_\xi(\widetilde\gamma_i)$ is homeomorphic to $S^1\times (0,1)$
and the intersections of $L_b$ with the level surfaces of
$h_{\widetilde\gamma_i}$ foliate this region by circles. If follows
that $\Sigma$ has a single boundary component and hence is
homeomorphic to a disk. From that it is immediate that for every
$c\in [\xi,1]$ the intersection
$L_b\cap\nu_{c\xi}(\widetilde\gamma_i)$ is a $2$-disk.
\end{proof}

\begin{defn}
We call any ball $B_{\lambda g_n}(x_n,r'/4)$ satisfying the
conclusions of Proposition~\ref{corners} Lemma~\ref{corners1}, and
Corollary~\ref{corners'} {\em a $3$-ball near a $2$-dimensional
boundary corner.}
\end{defn}

\subsection{Balls near open intervals}

The following results describe the parts of $M_n$ close to
$1$-dimensional Alexandrov balls.

\begin{lem}\label{1Dlem}
Given $\epsilon'>0$ the following holds for all $0<\epsilon$ less
than a positive constant $\epsilon_2(\epsilon')$.  If
$B_{g'_n}(x_n,1)$ is within $\epsilon$ of a standard $1$-dimensional
ball $J$,  then for any point $y\in B_{g'_n(x)}(x_n,24/25)$ whose
distance from the endpoints of $J$ (if any) is at least $1/25$
 there
is an open set $U=U(y)$, with $B_{g'_n(x_n)}(y,1/50)\subset
U\subset B_{g'_n(x_n)}(y,1/25)$, and an $\epsilon'$-approximation
$p_{x_n}\colon U\to J$, where $J$ is an open interval of length $3/50$ with central
point $p_{x_n}(y)$ such that
 the following hold:
\begin{enumerate}
\item There is a product structure on $U$ such that $p_{x_n}$ is the projection mapping onto
the interval factor. \item The fibers of $p_{x_n}$ are homeomorphic
to either $2$-spheres or $2$-tori.
\item There is a smooth unit vector unit field $\chi$ on $U$ such that
for any (minimal) geodesic $\gamma$  of length $\ge 1/400$, measured
in the metric $g'_n(x_n)$, ending at a point $z\in U$, the angle at
$z$ between $\chi(z)$ and $\gamma'(z)$ is within $\epsilon'$ of
either $0$ or $\pi$. \item If $z\in B_{g'_n(x_n)}(y,1/50)$
and if $d_{g'_n(x_n)}(w,z)\ge 1/400$, then the level surface of the distance
function $d(w,\cdot)$ through $z$ is contained in $U$ and is
isotopic in $U$ to a fiber of $p_{x_n}$.
\end{enumerate}
\end{lem}

\begin{proof}
Take a point $u$ at distance $7/200$ from $x_n$ and define
$$U=\bigl[B_{g'_n(x_n)}(u,13/200)\setminus
B_{g'_n(x_n)}(u,1/200)\bigr]\cap B_{g'_n(x_n)}(x_n,1/25)$$ and let
$p_{x_n}\colon U\to (-3/100,3/100)$ be $d(u,\cdot)-7/200$.
This is a $(1-\delta)$-strongly regular Lipschitz function for some
$\delta$ that depends on $\hat\epsilon$ and goes to zero as
$\hat\epsilon$ does. According to Lemma~\ref{smreg} there is a
smooth unit vector field $\chi$ on $U$ with the property that
$p_{x_n}'(\chi)>1-\delta'$ for some $\delta'$ that depends on
$\delta$ and goes to zero as $\delta$ does. The integral curves of
this vector field cross each level surface of $p_{x_n}$ exactly
once, and hence determine a product structure on $U$ with
$p_{x_n}$ being the projection onto one factor. Now let $\gamma$ be
any geodesic of length at least $1/400$ ending at $z\in U$. By
restricting $\gamma$ to a possibly shorter geodesic we can suppose
that $\gamma\subset B_{g'_n(x_n)}(x_n,1/25)$. Denote by $a$ its other endpoint.
Suppose that we have points $\bar u,\bar z,\bar a\in J$ such that under the
$\hat\epsilon$ approximation of between this ball and the open
interval $J$ that $d(\bar u,u)<\hat\epsilon$, $d(\bar z,z)<\hat\epsilon$ and
$d(\bar a,a)<\hat\epsilon$.  Suppose
that $\bar u$ is separated from $\bar a$ by $\bar z$. Then the
comparison angle $\tilde\angle azu$ is close to $\pi$, with an error that goes to zero as
$\hat\epsilon$ does. Hence, by monotonicity  the actual angle that $\gamma$ makes with any geodesic
from $u$ to $z$ is close to $\pi$ with an error
going to zero as $\hat\epsilon\rightarrow 0$. Since the angle at $z$
between $\chi(z)$ and any geodesic from $u$ to $z$ is close to
$\pi$, it follows that the angle between $\chi(z)$ and
$\gamma'(z)$ is less than an error term that goes to zero as
$\hat\epsilon$ goes to zero.

If $\bar u$ lies on the same side of $\bar z$ as $\bar a$, then we
choose $u'$ at distance $7/200$ from $x_n$ but on the `other side'
of $x_n$ (meaning the approximating points in the interval lie on
the other side). The same argument then shows that the angle between
any geodesic from $u'$ to $z$ and $\gamma$ is close to $\pi$,
implying that the angle between $\gamma$ and any geodesic from $u$
to $z$ is close to $0$. Since the angle between $\chi$ and any
geodesics from $u$ to $z$ is close to $\pi$, the angle between
$\chi$ and $\gamma'(z)$ is close to $\pi$, with an error that goes
to zero as $\hat\epsilon$ goes to zero.

Now let $z\in B_{g'_n(x_n)}(x_n,1/50)$ and let $w\in M_n$ be a point such that
$d(w,z)\ge 1/400$. It follows easily from the result just established
that if $\hat\epsilon>0$ is sufficiently small, then the level set of
$d(w,\cdot)$ through $z$ is contained in
$U$ and is transverse to $\chi$ and hence
isotopic in this open set to a fiber of $p_{x_n}$.

It remains to show that, provided that $\hat\epsilon>0$ is
sufficiently small, the fibers of $p_{x_n}$ are either $2$-spheres
or $2$-tori. If not we take a sequence of $\hat\epsilon_k\rightarrow
0$ and examples $p_k\colon U_k\to (-3/100,3/100)$ with fibers
$L_k=p_k^{-1}(t_k)$ that are not $2$-spheres or $2$-tori. Fix points
$z_k\in L_k$, let $d_k$ be the diameter of $L_k$ and rescale,
forming $\frac{1}{d_k}(U_k,z_k)$, and, after passing to a
subsequence take a limit. This limit is an Alexandrov space of
dimension $2$ or $3$ and splits as a product $\Ar\times Y$ where $Y$
has diameter $1$. If $Y$ is $2$-dimensional, then by
Proposition~\ref{smlimits} the convergence is smooth and $Y$ is a
surface of curvature $\ge 0$. Since $Y$ is orientable, it follows in
this case that $Y$ and hence the fibers $L_k$, for all $k$ sufficiently large,
are homeomorphic to either $2$-spheres or $2$-tori, which is a contradiction.

Suppose that $Y$ is $1$-dimensional. Then it is either a closed
interval or circle, and there are rescalings $\lambda_k$ such that
$\lambda_kU_k$ converge to the product $\Ar\times Y$. If $Y$
is a circle, we invoke Lemma~\ref{generic2D} and
Proposition~\ref{s1glue} to see that for all $k$ sufficiently large,
any level set of $p_k$ is contained in an open subset $V_k\subset
\lambda_kU_k$ that is the total space of a circle fibration.
We can take a slightly smaller compact fibration $W_k\subset V_k$
still containing the level set. The boundary components of $W_k$ are
tori and at least one of them separates the two ends of
$\lambda_kU_k$. On the other hand, the level set $L_k$
separates two of the boundary components of $W_k$. These two facts
together imply that for all $k$ sufficiently large, $L_k$ is a
$2$-torus, in contradiction to our assumption.

Lastly, suppose that $Y$ is a closed interval. Then invoking
Lemma~\ref{generic2D}, Proposition~\ref{s1glue} and
Proposition~\ref{flatboundary} we see that for all $k$ sufficiently
large every level set of $p_k$ is contained in the union of the
total space of an $S^1$-fibration and two sets of the form
$\nu_\xi^0(\widetilde\gamma_i)$ as in
Proposition~\ref{flatboundary}. Since the homotopy class of the
fiber of the $S^1$-fibration is trivial in $\nu_\xi^0(\widetilde \gamma_i)$, it
follows that the level set of $p_k$ is contained in an open subset
of $U_k$ whose fundamental group is the fundamental group of
a connected surface with non-empty boundary; that is to say a free
group. But the fundamental group of the level set maps
isomorphically onto the fundamental group of $U_k$ and is the
group of a surface. This means that the level set is the $2$-sphere.
\end{proof}

\begin{defn}
 A neighborhood $U$, a point $y\in U$, and a projection mapping
 $p\colon U\to J$ satisfying the conclusions of the above lemma
 is called {\em an interval product structure centered at $y$ with $\epsilon'$-control.}
 The content of the above lemma is that
 for $\epsilon<\epsilon_2(\epsilon')$ if $B_{g'_n(x)}(x,1)$ is within
 $\epsilon$ of a standard $1$-dimensional ball $J$ and if
 $y\in B_{g'_n(x)}(x,24/25)$ has distance at least $1/25$ from
 the endpoints (if any) of $J$, then there is an interval
 product structure centered at $y$ with  $\epsilon'$-control.
\end{defn}

Now we need to understand what happens near the endpoints of the
nearby interval. Unlike elsewhere in this section, here we do not assume that
$B_{g'_n(x_n)}(x_n,1)$ is disjoint from $\partial M_n$.

\begin{lem}\label{endpts}
There is $a_1>0$ such that the following holds for all $\epsilon>0$
and for all $\beta$ less than a positive constant
$\tilde\beta(\epsilon)$. Suppose that $B_{g'_n(x)}(x,1)$ is within
$\beta$ of a standard $1$-dimensional ball $J$, that $\bar x$ is an
endpoint of $J$ and that $d(x,\bar x)<1/25$. One of the following
two possibilities holds:
\begin{enumerate}
\item $\bar B(x,1/2)$ is diffeomorphic to $T^2\times [0,1/2)$, to a solid torus, to a twisted
$I$-bundle over the Klein bottle, to a $3$-ball, or to $\Ar P^3\setminus
B^3$.
\item There is a $\lambda\ge \epsilon^{-1}\rho_n(x)^{-2}$ such that $B_{\lambda g_n}(x,1)$
is within $\epsilon$ of a standard $2$-dimensional ball of area at
least $a_1$, and $\bar B_{g'_n(x)}(x,1/2)\setminus
B_{g'_n(x)}(x,1/\lambda)$ is a topological product of a surface with an interval
with $d(x,\cdot)$ being
the projection mapping to the interval of this product structure.
\end{enumerate}
\end{lem}

\begin{proof}
The first case to consider is when $B(x,1/2)$ meets the boundary of
$M_n$. Let $x'\in B(x,1/2)\cap \partial M_n$. According to
Proposition~\ref{bdrybehavior} $B(x',1)$ is diffeomorphic to
$T^2\times [0,1)$ and the result follows easily in this case. Thus,
we can suppose that $B(x,1)$ is disjoint from $\partial M_n$.
Suppose that there is no $\beta$ as required. We take a sequence of
$\beta_k\rightarrow 0$ and counter examples $x_k\in M_{n(k)}$. We can
assume that $x_k$ is within $\beta_k$ to the endpoint of the
interval. We apply the blow-up result (Proposition~\ref{blowup}).
This tells us that for all $k$ sufficiently large there is another
point $x'_k$, such that the sequence $x'_k$ also converges to the
endpoint, such that one of two
possibilities holds for the distance function from $x'_k$: either it
has no critical points within distance $1/2$ of $x'_k$ (except of
course $x'_k$) or there is a sequence $\delta_k\rightarrow0$ such
that all critical points within distance $1/2$ are within $\delta_k$
and there is a critical point at distance $\delta_k$ from $x'_k$. In
the first case, the ball in question is a topological $3$-ball. In
the second we rescale by multiplying the metric by $\delta_k^{-2}$.
The result is a complete Alexandrov space of non-negative curvature
and of dimension $2$ or $3$ with $\bar x$ being the limit of the $x'_k$.
We consider the case when the result is
$3$-dimensional. By Proposition~\ref{smlimits} it is a complete
$3$-manifold of non-negative curvature, and as such it has a soul.
If the soul is a point, then the limit is diffeomorphic to $\Ar^3$
and level sets of the distance function from $\bar x$ are $2$-spheres. If the soul
is a circle, then the limit is a solid torus and the level sets of the distance function
from $\bar x$ are
$2$-tori. If the soul is a Klein bottle, then the level sets of the distance function
from $\bar x$ are $2$-tori. If the soul is $\Ar P^2$, then the limit is a punctured
$\Ar P^3$ and the level sets are $2$-spheres. Thus, in these cases
the original $B(x_k,1/2)$ is diffeomorphic to the limiting complete manifold
and the level sets of the distance function for $x_k$ away from the
end point are topologically isotopic to the level sets of the
distance function from $x'_k$ at distances more than $\delta_k'$.
This establishes by contradiction that Case 1 holds under these
assumptions.

Suppose that the limit of the rescalings is
$2$-dimensional $(X,\bar x')$.
Consider points $q_k\in B_{g'_{n(k)(x_k)}}(x_k,1/2)$ that converge to a point $\bar q\in J$
at distance $1/4$ from $\bar x$. The point $x'_k$ is chosen as the unique local maximum for the distance function
from $q_k$ near the endpoint of $J$. Let $\gamma_k$ be a geodesic from $x'_k$ to $q_k$, and let $q'_k$ be the point of $\gamma_k$
at distance $2/\delta_k$ from $x'_k$. Let $p_k$ be the critical point for $d(x'_k,\cdot)$ at distance
$\delta_k$ from $x'_k$. In the rescaled ball we have $|x'_kq'_k|=2$ and $x'_kp_k|=1$.
The fact that $x'_k$ is the unique local maximum for the distance function from $q_k$ near the endpoint of $J$, this implies that $|x'_kq'_k|\ge |p_k,q'_k|$.
Since $p_k$ is a critical point for the distance function from $x'_k$, the comparison angle $\tilde x'_kp_kq'_k$
is at most $\pi/2$. This facts together imply that the area of the unit ball centered at $\bar x'$ in $X$
has area at least $a_1$ for some universal constant $a_1$. This shows that Case 2 holds under these assumptions, which is a contradiction.
\end{proof}

\subsection{Determination of the Constants}

We fix $\epsilon'>0$ a  universally small constant. Then
$\epsilon>0$ is chosen to be less than the minimum of the constants
$\epsilon_0(\epsilon')$ in Proposition~\ref{semilocal},
$\epsilon_1(\epsilon')$ in Proposition~\ref{s1glue}, and
$\epsilon_2(\epsilon')$ in Lemma~\ref{1Dlem} and sufficiently small
so that Lemma~\ref{2dregular} holds. Then $\beta$ is chosen less
than $\tilde\beta(\epsilon)$ in Lemma~\ref{endpts} and also less
than $\epsilon/2$. Now we fix $0<\xi\le 10^{-3}$ with $\xi$
sufficiently small so that Theorem~\ref{summary}, Lemma~\ref{nuint},
Lemma~\ref{nuint1}, and Corollary~\ref{corners1}, all hold. We also
fix $\xi>0$ less than the constant $\xi_0(\epsilon)$ in
Proposition~\ref{flatboundary}. Next, we fix $a>0$ less than the
constants $a_2(\beta/2)$ in Lemma~\ref{smarea1D} and $a_1$ in
Lemma~\ref{endpts}. We now fix $\mu>0$ less than the minimum of
$\{\mu_1(\xi),\mu_2(\epsilon),\mu_3(\epsilon,a),
\mu_4(\xi,\epsilon),\mu_5(\xi),\mu_6(\xi,a)\}$ where these are
the constants given Theorem~\ref{summary}, Lemma~\ref{generic2D},
Proposition~\ref{intcone1}, Proposition~\ref{flatboundary},
Lemma~\ref{nuint}, and Corollary~\ref{corners1}.  Now we fix
$\delta,r_0$ positive constants as in Theorem~\ref{summary} for the
given values of $\xi,\mu$ and $a$. We fix $s_1>0$ less than $\tilde
s_1(\xi,\mu,a)$ in Theorem~\ref{summary}. Then we choose $s_2>0$
less than the constants $\tilde s_2(\xi,\mu,a,s_1)$ in
Theorem~\ref{summary}. With all of these constants determined, we
are ready to fix $0<\hat\epsilon<\beta/2$. We choose this constant
less that the minimum of
$$\{\hat\epsilon_0(\mu,s_2),\hat\epsilon_1(\epsilon,a,r_0),\hat\epsilon_2(\epsilon,\mu,s_1),
\hat\epsilon_3(\xi,s_1),\hat\epsilon_4(\epsilon,a,\mu,r_0)\}$$ as given
in Lemma~\ref{generic2D}, Proposition~\ref{intcone1},
Proposition~\ref{flatboundary}, Lemma~\ref{nuint},
Corollary~\ref{corners1}. We also choose
$\hat \epsilon<10^{-3}\xi^2s_1/C$ where $C$ is the constant in Lemma~\ref{generic2D}.
Now we pass to a subsequence of
the $M_n$ so that $\epsilon_n\le {\rm min}(\hat\epsilon,\epsilon)$
for all $n$, and also so that Proposition~\ref{bdrybehavior} holds
for all $n$.

\section{The global result}

At this point we have fixed all the constants appearing in the last
two sections in such a way that the conclusions of all the results from these two sections
hold. This gives us
complete control over the local nature of the $(M_n,g_n)$, in the
sense that we have complete control over the
$B_{g'_n(x)}(x,1/2)\subset B_{g'_n(x)}(x,1)$ for every $x\in M_n$.
The purpose of this section is to globalize these results
establishing Theorem~\ref{1Dthm'}.

\begin{defn}
Given a ball $B_{\lambda g_n}(x,r)$ we say that $r$ is its {\em rescaled radius} and $r/\sqrt{\lambda}$ is its {\em unrescaled radius}.
\end{defn}

\subsection{Regions of $M_n$ close to open intervals}

We begin the globalization by studying the generic
``$1$-dimensional'' regions of the $M_n$. We shall construct an open
set $U_{n,1}'\subset M_n$ which is a first approximation to the
submanifold $V_{n,1}\subset M_n$ referred to in
Theorem~\ref{1Dthm'}. The manifold $U'_{n,1}$ will be an open submanifold.
Eventually, when we  define $V_{n,1}$ as
follows: For each end of $U_{n,1}'$ either we truncate it by
removing an open collar neighborhood of that end, or we
extend it  by adding compact external collar neighborhood. Also, we
shall add disjoint compact $3$-balls to $U'_{n,1}$ in creating $V_{n,1}$.

\begin{prop}\label{Un1}
Consider the subset $X_{n,1}\subset M_n$ consisting of all points
$x_n\in M_n$ for which $B_{g'_n(x_n)}(x_n,1)$  is within $\hat\epsilon$
of a standard $1$-dimensional ball $J$ and the distance from $x_n$
to the endpoints (if any) of $J$ is at least $1/50$. Then there is
an open subset $U_{n,1}\subset M_n$ containing $X_{n,1}$ with the
following properties:
\begin{enumerate}
\item Each component of $U_{n,1}$ is either a $2$-torus bundle over
the circle, or diffeomorphic to a product of either $S^2$ or $T^2$
with an open interval.
\item For each non-compact end ${\mathcal E}$ of $U_{n,1}$ there is a point $x_{\mathcal E}\in
X_{n,1}$, and an interval product structure centered at $x_{\mathcal E}$ with $\epsilon'$-control,
$p_{x_{\mathcal E}}\colon U(x_{\mathcal E})\to J(x_{\mathcal E})$,
where $J(x_{\mathcal E})$ is an interval of length $\ge 1/100$,
with the property that $U(x_{\mathcal E})$  is a neighborhood of the end ${\mathcal E}$.
\item For distinct non-compact ends ${\mathcal E}$ and ${\mathcal
E}'$ the neighborhoods $U(x_{\mathcal E})$ and $U(x_{\mathcal E}')$
are disjoint.
\item For each point $x\in X_{n,1}$, the ball  $B_{g'_n(x)}(x,1/400)$
is contained in $U_{n,1}$.
\end{enumerate}
\end{prop}

\begin{proof}
Suppose that we have an open set $V\subset M_n$ satisfying the first
three conclusions and a point $x\in X_{n,1}$ for which the
fourth conclusion does not hold.  Consider the open set $U(x)$ and
projection $p_x\colon U(x)\to J'$ associated to $x$ by
Lemma~\ref{1Dlem}. Recall that $J'$ is of length $3/50$.
Let $J''\subset J'$ be an open interval of length
$1/25$ centered at $p_x(x)$ and let $W=p^{-1}(J')$. If $W$ is
disjoint from $V$ we replace $V$ by $V\cup W$. The result satisfies
the first three conclusions. Suppose that $W$ meets a component
of $V$. Since $W$ is close to an open interval, it has two ends.
Suppose that there is a level set of $p$ near each end of $W$ that
is contained in $V$. Then by Lemma~\ref{1Dlem} these level sets are
isotopic to the fibers of the product structure of these components
of $V$, and hence the union of $V\cup W$ still satisfies the first
three conclusions of this proposition. Similarly, if one end of
$W$ has such a level surface and the other end is disjoint from $V$,
then the union $V\cup W$ satisfies the first three conclusions
of this proposition.

Now suppose that one of the ends of $W$ (say the end corresponding
to $-1/50)$) meets $V$ but no level surface near this end of $W$ is
contained in $V$. Let $x'$ be the point as in the second item for
the corresponding end of $V$. Then $\rho_n(x')$ and $\rho_n(x)$ are
within a multiplicative factor of $2$ of each other. We extend
the end of $W$ by taking $\hat W=p^{-1}(-3/100,1/50)$. According to
Lemma~\ref{1Dlem} and the fact that $\rho_n(x')\le 2\rho_n(x)$,
there is a level surface of $\hat W$ near the negative end of $\hat
W$ that is contained in $V$. Arguing as before shows that in all
cases we can extend $V$ by taking its union with a set of the form
$W$ in such a way that the first three conclusions still hold
but also so that $B_{g'(x)}(x,1/400)\subset V\cup W$. Since $M_n$ is
compact and $\rho_n$ is bounded below by a positive constant on
$M_n$, it follows easily that after a finite number of such
extensions we have arrived at a situation where the all four
conclusions hold.
\end{proof}

We fix $U_{n,1}\subset M_n$ as in the above proposition. For each
non-compact end ${\mathcal E}$ of $U_{n,1}$ we fix a point
$x_{\mathcal E}$ producing the neighborhood $U(x_{\mathcal E})$ of
the end together with a projection mapping $p_{x_{\mathcal E}}\colon
U(x_{\mathcal E})\to J(x_{\mathcal E})$ as in Conclusion 2 of
Proposition~\ref{Un1}. In particular, $p_{x_{\mathcal E}}$ is an
$\epsilon'$-approximation  and $J(x_{\mathcal E})$ is an interval of
length at least $1/50$ centered at $p_{x_{\mathcal E}}(x_{\mathcal
E})$.

\subsection{Balls close to half-open intervals}

Now suppose that $x\in M_n\setminus U_{n,1}$ is in the closure of
$U_{n,1}$. Since $\epsilon_n<\hat\epsilon$ for all $n$, there are
three possibilities for $B_{g'_n(x)}(x,1)$:
\begin{enumerate}
\item It is within $\hat\epsilon$ of a standard $2$-dimension ball
$\bar B$ of area $\ge a$. \item It is within $\hat\epsilon$ of a
$2$-dimensional standard ball $\bar B$ of area $<a$.
\item It is within $\hat\epsilon$ of a standard $1$-dimensional ball
$J$.\end{enumerate}

In the second case, if follows from Lemma~\ref{smarea1D} and the
fact $a<a_2(\beta/2)$ that  $\bar B$ is within $\beta/2$ of a
standard $1$-dimensional ball $J$ and since $\hat\epsilon<\beta/2$,
if follows that $B_{g'_n(x)}(x,1)$ is within $\beta$ of $J$. Thus,
in the second and third cases, $B_{g'_n(x)}(x,1)$ is within $\beta$
of a standard $1$-dimensional ball $J$. Suppose that either Case
2 or 3 above holds, and consider two further possibilities: (i)
 If the endpoints of
$J$, if any, are at distance at least $1/25$ from $x$, and (ii)
there is an endpoint of $J$ within distance $1/25$ of $x$. The first
possibility contradicts the fact that $x\not\in U_{n,1}$: Since we
have chosen $\beta<\epsilon/2$ it follows from the definition if
(i) holds then that $x\in X_{n,1}\subset U_{n,1}$. This contradicts
our assumption that $x\not\in U_{n,1}$. Thus, we conclude that $x$
is within distance $1/25$ of an endpoint of $J$. Since
$\beta<\tilde\beta(\epsilon)$ from Lemma~\ref{endpts}, the
conclusions of that lemma hold for $B= B_{g'_n(x)}(x,1)$. That is to
say: there exists an open subset $V=V(x)\subset B$ containing
$B_{g'_n(x)}(x,1/2)$ such that one of the following hold:
\begin{enumerate}
\item $V$ is diffeomorphic to $T^2\times [0,1)$ and contains a
boundary component of $M_n$.
\item $V$ is an open $3$-ball or is homeomorphic to (a) the complement
of a closed $3$-ball in $\Ar P^3$,
(b) an open solid torus, or (c) an open twisted $I$-bundle
over the Klein bottle.
\item  There is a constant $\lambda>\epsilon^{-1}$
and a point $x'\in M_n$ such that $B'=B_{\lambda g'_n(x)}(x',1)$
contains $B_{\lambda g'_n(x)}(x,1/2)$ and is within $\epsilon$ of a
$2$-dimensional Alexandrov ball of radius $1$ and area at least
$a_0$.
\end{enumerate}
Furthermore, in all cases the end of $V$ is contained in $U_{n,1}$
and there is a level set $\Sigma$ for the distance function from $x$
with $\Sigma\subset U_{n,1}$ and with $\Sigma$ isotopic in $V\cap
U_{n,1}$ to a fiber of the fibration structure of $U_{n,1}$ (of
course $\Sigma$ is either a $2$-sphere or a $2$-torus). Thus, in the
first three cases the union of $V$ with the component of $U_{n,1}$
containing the end of $V$ is diffeomorphic to $V$. In the last case,
the distance function from $x'$ has no critical points in
$B\setminus \overline{B_{\lambda g'_n(x')}(x',1)}$, and in
particular, the region between $B'$ and $\Sigma$ is a topological
product.

 This completes the proof of the following:

\begin{lem}\label{dichotomy}
Let $A$ be a connected component of $M_n\setminus U_{n,1}$. Then one
of the following holds.
\begin{enumerate}
\item  For every point $x\in A$ the ball $B_{g'_n(x)}(x,1)$ is within
$\epsilon_n$ of a standard $2$-dimensional ball $\bar B$ of area
$\ge a$.
\item There is a point $x\in A$ such that $A\subset
B_{g'_n(x)}(x,1/2)$ is diffeomorphic to $T^2\times [0,1]$,
a solid torus, a twisted
$I$-bundle over the Klein bottle, a closed $3$-ball, or $\Ar
P^3\setminus B^3$. In all these cases the metric sphere
$S_{g'_n(x)}(x,1)$ is either a $2$-torus or a $2$-sphere and is isotopic in $U_{n,1}$ to a fiber in its
fibration structure.
\item There is $\lambda>\epsilon^{-1}$ and a point $x\in A$ such
that:
\begin{enumerate}
\item[(a)]  $A\subset  B_{g'_n(x)}(x,1/2)$.
\item[(b)] $\bar B_{g'_n(x)}(x,1)\setminus  B_{g'_n(x)}(x,1/\lambda)$ is a
topological product with an interval and the distance function from
$x$ is the projection mapping of this product structure.
\item[(c)] $B_{g'_n(x)}(x,9/10)\setminus B_{g'_n(x)}(x,1/10)\subset
U_{n,1}$ and  $S_{g'_n(x)}(x,1/2)$ is isotopic in
$U_{n,1}$ to a fiber of its fibration structure.
\item[(d)] $B_{\lambda g'_n(x)}(x,1)$ is within $\epsilon$ of a standard
$2$-dimensional ball of area $\ge a$.
\end{enumerate}
\end{enumerate}
\end{lem}

\begin{defn}
We call a component of $M_n\setminus U_{n,1}$ satisfying the
Conclusion 3 above {\em a component which is close to an interval but which
expands to be close to a standard $2$-dimensional ball} and we call
a component satisfying Conclusion 1 above {\em a component close to
a $2$-dimensional space}. For a component which is close to an interval but which
expands to be close to a standard $2$-dimensional ball, we use the
metric $\lambda g'_n(x)$ as described in Part 3 of the previous lemma on
the entire component. For components close to a $2$-dimensional
space we use varying metrics $g'_n(x_i)$ as described in Part 1 of
the previous lemma.
\end{defn}

At this point we add to $U_{n,1}$ every component of $M_n\setminus
U_{n,1}$ of Type 2 in Lemma~\ref{dichotomy}. Call the result $U'_{n,1}$. Some of the
components of $U'_{n,1}$ are components of $U_{n,1}$. Let us
consider the others. Fix  a component $C'$ of $U'_{n,1}$ that is not
a component of $U_{n,1}$. It contains a component $C$ of $U_{n,1}$.
The component $C$ has at most two ends and $C'$ is the union of $C$
with either one or two neighboring components  of $M_n\setminus C$
(neighboring in the sense that their  closures meet $C$). Let $A$ be
a component of $M_n\setminus C$  neighboring $C$ that is contained
in $C'$. Then $A$ is diffeomorphic to one of the four manifolds list
in Conclusion 2 of Lemma~\ref{dichotomy}. Furthermore, any fiber of
the fibration structure on $C$, it divides $C'$ into two components
one of which contains $A$ and is the union of $A$ and a collar
neighborhood of the boundary of $A$. Hence, this closed
complementary component is homeomorphic to $A$. If $C'=C\cup A$,
then it follows that $C'$ is homeomorphic to ${\rm int}\, A$ and
hence to the interior of one of the four manifolds listed in
Conclusion 2 of Lemma~\ref{dichotomy}. If $C'=C\cup A_1\cup A_2$ for
distinct components $A_1$ and $A_2$ of $M_n\setminus U_{n,1}$, then
the same argument shows that $C'$ is the union of
two manifolds homeomorphic to one of the four listed in Conclusion 2
of Lemma~\ref{dichotomy} along their common boundary. Any such
manifold is a component of $M_n$, and every one of its prime factors
is geometric. (The manifold is prime unless it is $S^3$ or $\Ar
P^3\#\Ar P^3$.)

Invoking the hypothesis that no closed component of $M_n$ admits a
Riemannian metric of non-negative sectional curvature, allows us to
conclude the following:

\begin{prop}\label{componenttype}
The open subset $U'_{n,1}\subset M_n$ constructed in the previous
paragraph satisfies the following:
\begin{enumerate}
\item Every component of $U'_{n,1}$ is diffeomorphic to one of the following:
\begin{enumerate}
\item[(a)] a $T^2$-bundle or an $S^2$-bundle over either the circle
or an  interval with the fiber(s) over the endpoint(s)
being boundary component(s) of $M_n$,
\item[(b)] a twisted $I$-bundle over the
Klein bottle whose boundary is a boundary component of $M_n$,
\item[(c)] an open solid torus, an open twisted $I$-bundle over the
Klein bottle, an open $3$-ball, the complement of a closed $3$-ball
in $\Ar P^3$, or
\item[(d)] the union of two twisted $I$-bundles over the Klein bottle along their common boundary.
\end{enumerate}
\item Each  non-compact end of $U'_{n,1}$ has a neighborhood that is
a component of $U_{n,1}$, and hence there is a non-compact end
${\mathcal E}$ of $U_{n,1}$ such that $U(x_{\mathcal E})$ is a
neighborhood of this end.
\item Every complementary component $M_n\setminus U'_{n,1}$
either is a component close to a $2$-dimensional space  or is a component
which is close to an interval but which expands to be close to a standard
$2$-dimensional ball.
\end{enumerate}
\end{prop}

\subsection{A decomposition into compact sets}\label{compactext}

For each non-compact end ${\mathcal E}$ of $U'_{n,1}$ we have the
neighborhood $U(x_{\mathcal E})$ that fibers over an interval
$J(x_{\mathcal E})$ by an $\epsilon'$-approximation. Denote by $J^+(x_{\mathcal E})$
the closed half-ray with endpoint the central point of $J(x_{\mathcal E})$ whose preimage is also
a neighborhood in $U'_{n,1}$ of the end
${\mathcal E}$. Let
$\Sigma({\mathcal E})\subset U(x_{\mathcal E})$ be the fiber over
the central point of $J(x_{\mathcal E})$, and set $U^+({\mathcal E})$
equal to the preimage of  $J^+(x_{\mathcal E})$. We form the union
$W_{n,2}$ of $M_n\setminus U'_{n,1}$ with the $U^+({\mathcal E})$ as
${\mathcal E}$ varies over the ends of $U'_{n,1}$. Then $W_{n,2}$ is
compact and $\partial W_{n,2}$ is a disjoint union of the
$\Sigma({\mathcal E})$ as ${\mathcal E}$ varies over the non-compact
ends of $U'_{n,1}$. In particular $\partial W_{n,2}$ consists of a
disjoint union of $2$-tori and $2$-spheres. We set $W_{n,1}$ equal
to the complement in $M_n$ of the interior of $W_{n,2}$. It is the
compact manifold with boundary obtained from $U'_{n,1}$ by deleting the collar neighborhoods $U^+(x_{\mathcal E})$ as ${\mathcal E}$ ranges over the non-compact
ends of $U'_{n,1}$. Its boundary consists of the
boundary of $W_{n,2}$ disjoint union the boundary of $M_n$. Recall that the latter is a
disjoint union of incompressible tori.

Let us recap our progress to date.

\begin{prop}\label{ct1} We have a decomposition $M_n=W_{n,1}\cup W_{n,2}$.
The intersection $W_{n,1}\cap W_{n,2}$ is the boundary of $W_{n,2}$ and it is the union of the
boundary components of $W_{n,1}$ that are not boundary components of $M_n$.
For each end ${\mathcal E}$
of $U'_{n,1}$ there is one component of $W_{n,1}\cap W_{n,2}$. This component is denoted $\Sigma({\mathcal E})$.
Each of these components is either a $2$-torus or a $2$-sphere, and each $\Sigma({\mathcal E})$ is  a fiber
of the projection mapping $p_{x_{\mathcal E}}\colon U(x_{\mathcal E})\to J(x_{\mathcal E})$.
Each component of $W_{n,1}$ is homeomorphic to one of the following:
\begin{enumerate}
\item a $T^2$-bundle  over either a circle or a
compact interval,
\item an $S^2$-bundle over a compact interval,
\item a twisted $I$-bundle over the Klein bottle,
\item a compact solid torus,
\item a compact $3$-ball,
\item the complement in $\Ar P^3$ of an open $3$-ball, or
\item the union of two twisted $I$-bundles over the Klein bottle
along their common boundary.
\end{enumerate}
\end{prop}

\begin{proof}
This is immediate from Proposition~\ref{componenttype} and the
construction.
\end{proof}

For each component $A$ of $M_n\setminus U'_{n,1}$, we set $\widehat
A$ equal to the union of $A$ together with $U^+({\mathcal E})$ as
${\mathcal E}$ varies over the ends of $U'_{n,1}$ whose closures
meet $A$. Then $W_{n,2}$ is the disjoint union of the $\widehat A$
as $A$ ranges over the components of $M_n\setminus U'_{n,1}$.

\subsection{Covering of $W_{n,2}$}

Next we must study the structure of components $\widehat A$ of
$W_{n,2}$. The crucial ingredient is to construct chains of
$\epsilon'$-solid cylinder neighborhoods that together with
$U_{2,{\rm generic}}$, the $\epsilon'$-solid torus neighborhoods near interior cone
points, and the  $3$-balls near a $2$-dimensional
boundary corner cover $W_{n,2}$.

The following two results are immediate consequences of
Theorem~\ref{summary} and the results of Section~\ref{4}

\begin{lem}\label{Acover2}
Suppose that $A$ is a  component of $M_n\setminus U'_{n,1}$ that is
close to a $2$-dimensional space. Then $\widehat A$ is contained in
the union of:
\begin{enumerate}
\item $U_{2,{\rm generic}}$,
\item the open subset $U_{\rm cyl}$ consisting of all points that are in the center of
cores of $\epsilon'$-solid cylinders $\nu_{\xi^2}(\widetilde \gamma)$
at scale $s_1$ near flat $2$-dimensional boundary points,
\end{enumerate}
and a finite number of
\begin{enumerate}
\item[3.] $\epsilon'$-solid tori $B(z_i)=B_{g'_n(z_i)}(z_i,r(z_i)/4)$, for $i=1,\ldots, N_t$, near interior cone points, and
\item[4.]  $3$-balls $B(x_i)=B_{g'_n(x_i)}(x_i,r(x_i)/4)$, for $i=1,\ldots N_c$, near  $2$-dimensional
boundary corners.
\end{enumerate}
\end{lem}

For each $3$-ball $B(x)$  near a $2$-dimensional
boundary corner,
we denote by $\widehat B(x)$ the ball
$B_{g'_n(x)}(x,7r(x)/8)$ and call it the {\em expanded version of the ball}.

\begin{addendum}\label{1add}
We can choose the neighborhoods in Lemma~\ref{Acover2} so that in addition to the fact that they cover $\widehat A$ we have:
\begin{enumerate}
\item[(i)] The  $3$-balls $B(x_i), 1\le i\le N_c$, are disjoint.
\item[(ii)] Any $3$-ball $B_{g'_n(x)}(x,r(x)/4)$ near a $2$-dimensional
boundary corner that meets one of the $B(x_i), 1\le i\le N_c$, is contained in one of the  expanded versions $\widehat B(x_i)$.
\item[(iii)] The $\epsilon'$-solid tori $B(z_i)), 1\le i\le N_t$, are disjoint.
\item[(iv)] Each $\epsilon'$-solid torus in (3) of Lemma~\ref{Acover2} is disjoint from each  $3$-ball in (4) of Lemma~\ref{Acover2} and is also disjoint from $U_{\rm cyl}$.
\end{enumerate}
Also, every $3$-ball and every $\epsilon'$-solid torus
 in the collection above meets  $\widehat A$.
\end{addendum}

\begin{proof}

We consider collections of  disjoint $3$-balls near $2$-dimensional boundary corners.
For any member of such a collection we have its unrescaled radius $\rho_n(x_i)r(x_i)$. If there is a disjoint
 $3$-ball near a $2$-dimensional boundary corner then we add it to the collection. If there is a $3$-ball $B$ near a $2$-dimensional boundary corner that meets one of the $3$-balls in the collection
but is not contained in any of the expanded version of the $3$-balls in the collection, then we
add $B$ to the collection and remove all the $3$-balls in the collection that meet it.
Since all these balls are contained in $\widehat B$ which itself is a union of $B$, $U_{2,{\rm gen}}$ and $U_{\rm cyl}$,
removing these balls does not destroy the fact that we have a covering of $\widehat A$.
In this case, it follows that the unrescaled radius of $B$ is at least $1.1$ times the unrescaled radius of each ball that we deleted. Since $\widehat A$ is compact and thus the unrescaled radius of any ball is bounded above and below by positive constants, starting with the empty collection we can only repeat these two operations only finitely many times. When we can no longer repeat the operation we arrive at a collection of $3$-balls near $2$-dimensional boundary corners satisfying the first two conditions.

Consider the collection of $\epsilon'$-solid tori. If two of these meet, say
$B(z_i)$ and $B(z_j)$, then without loss of generality we can suppose that
$$\rho_n(z_i)r(z_i)\ge \rho_n(z_j)r(z_j).$$ This implies that $B(z_j))\subset B_{g'_n(z_i)}(z_i,r(z_i)/2)$. Since
$$B_{g'_n(z_i)}(z_i,r(z_i)/2)\setminus B_{g'_n(z_i)}(z_i,r(z_i)/4)$$ is contained in $U_{2,{\rm gen}}$, we can remove $B(z_j)$ from the collection and still have a covering.
This allows us to make the $\epsilon'$-solid tori disjoint.

Now suppose that an $\epsilon'$-solid torus in the collection meets one of the $3$-balls near a $2$-dimensional boundary
corner in our collection. If the unrescaled radius of the $3$-ball is no greater than  the unrescaled radius of the $\epsilon'$-solid torus, then the $3$-ball is contained in the expanded version of the $3$-solid torus, where, as before, in the expanded version we replace the radius $r(x_i)/4$ by $7r(x_i)/8$. But this is impossible, since the generic circle fibers in the $3$-ball are isotopic in the solid torus to generic fibers of its Seifert fibration, but the circle fibers in the $3$-ball are homotopically trivial in the $3$-ball whereas the circle fibers in the solid torus are homotopically non-trivial in the solid torus. If the unrescaled radius of the  $3$-ball is greater than that of the solid torus, then the solid torus is contained in the expanded $3$-ball and hence, by the same reasoning as above, it can be removed from the collection without destroying the fact that the collection covers $\widehat A$. This shows that we can make the $\epsilon'$-solid tori disjoint from the $3$-balls near boundary corner points.

Lastly, suppose that an $\epsilon'$-solid torus in the collection meets $U_{\rm cyl}$. Then there is an $\epsilon'$-solid cylinder that is contained in the expanded version of the $\epsilon'$-solid torus. This is a contradiction for it implies that the
generic fiber of the Seifert fibration on the expanded $\epsilon'$-solid torus is homotopically trivial in the $\epsilon'$-solid cylinder contained in the expanded $\epsilon'$-solid torus.

This completes the proof that there is a covering satisfying the listed properties. From this collection we simply remove any of the sets in the collection that does not meet $\widehat A$.
\end{proof}

We have analogues of these results for components which are near intervals but
that expand to be near $2$-dimensional components.

\begin{lem}\label{Acover1}
Suppose that $A$ is a component of $M_n\setminus U'_{n,1}$ that is
close to an interval but that expands to be close to a standard
$2$-dimensional ball. Then for some $\lambda>\epsilon^{-1}$, using
the metric $\lambda g'_n(x)$ for an appropriate $x\in A$, we have that  $\widehat A$ is contained in
the union of
\begin{enumerate}
\item $U_{2,{\rm generic}}$,
\item the open subset $U_{\rm cyl}$ of points in the center of cores of $\epsilon'$-solid cylinders $\nu_{\xi^2}(\widetilde \gamma)$ at scale $s_1$ near flat
$2$-dimensional boundary points,
\end{enumerate}
and a finite union of
\begin{enumerate}
\item[3.] $\epsilon'$-solid tori $B(z_i)=B_{\lambda g_n}(z_i,r(z_i)/4)$ near interior cone points, and
\item[4.] $3$-ball components near $2$-dimensional boundary
corners $$B(x_i)=B_{\lambda g_n}(x_i,r(x_i)/4).$$
\end{enumerate}
\end{lem}

The same argument as in the proof of Addendum~\ref{1add} shows:

\begin{addendum}\label{2add}
We can choose the neighborhoods in Lemma~\ref{Acover1} so that in addition to the fact that they cover $\widehat A$ we have
\begin{enumerate}
\item[(i)] The  $3$-balls $B(x_i)$ in (4) of Lemma~\ref{Acover1} are disjoint.
\item[(ii)] Any $3$-ball $B(x)$ that meets one of the $B(x_i)$ is contained in one of the expanded versions $\widehat B(x_i)$.
\item[(iii)] The $\epsilon'$-solid tori $B(z_i)$ in (3) of Lemma~\ref{Acover1} are disjoint.
\item[(iv)] Each $\epsilon'$-solid torus in (3) of Lemma~\ref{Acover1} is disjoint from each  $3$-ball in (4) of Lemma~\ref{Acover1} and from $U_{\rm cyl}$.
\end{enumerate}
Also, every $3$-ball in (4) of Lemma~\ref{Acover1} and every $\epsilon'$-solid torus in (3)
of Lemma~\ref{Acover1} meets  $\widehat A$.
\end{addendum}

\subsubsection{The circle fibration}

Above we constructed a covering of
the compact submanifold $W_{n,2}$  by:
\begin{enumerate}
\item $U_{2,{\rm gen}}$,
\item $U_{\rm cyl}$,
\item a finite number of $\epsilon'$-solid tori $B(z_i)$, and
\item a finite number of $3$-balls $B(x_i)$ near $2$-dimensional boundary corners.
\end{enumerate}
The constants $r(z_i)$ and $r(x_i)$ are bounded between $r_0$ and $10^{-3}$.
The constants $\lambda_j$ or $\lambda_i$ multiplying the metric $g_n$ is the same for all the balls and all the solid tori that meet a given component of $W_{n,2}$ that is near to an interval but expands to be near to a $2$-dimensional component.
For balls and solid tori that meet any other component  of $W_{n,2}$ the constant multiplying the metric $g_n$ is the value of $\rho_n^{-2}$ at the central point.

Proposition~\ref{s1glue} and Lemma~\ref{generic2D} imply that there is
an open subset $U_2'\subset U_{2,{\rm gen}}$ that contains the
complement in $W_{n,2}$ of  the union of the open sets in (3) and (4) in Lemmas~\ref{Acover2} and~\ref{Acover1} and the complement of $U_{\rm cyl}$.
Furthermore, $U'_2$ admits circle fibration whose fibers are
$\epsilon'$-orthogonal to the $S^1$-product neighborhoods centered
at each point of $U'_2\subset U_{2,{\rm gen}}$ and also the fibers have length less than $C\hat\epsilon$ where $C$ is the universal constant from Lemma~\ref{generic2D}.

\subsubsection{Removing $3$-balls from  $W_{n,2}$}

At this point we modify $W_{n,1}$ and $W_{n,2}$ by removing the
$3$-balls $B(x_i)$ near $2$-dimensional boundary corners  from $W_{n,2}$ and adding their closures as
disjoint components of $W_{n,1}$. The results are denoted $W'_{n,1}$
and $W'_{n,2}$, respectively. A slight modification of Proposition~\ref{ct1} holds
for these subsets.

\begin{cor} The conclusions of Proposition~\ref{ct1} hold
for the compact submanifolds $W'_{n,1}$ and $W'_{n,2}$ with one
change. The intersection $W'_{n,1}\cap W'_{n,2}$ is equal to the
disjoint union of $W_{n,1}\cap W_{n,2}$ and the metric spheres
$S(x_i,r(x_i)/4)$ that are the frontiers of the $B(x_i)$ are topological $2$-spheres.
\end{cor}

By doing this we have gained one thing: namely, $W'_{n,2}$ is
covered by  $U'_2$, $U_{\rm cyl}$, and the $\epsilon'$-solid tori
$B(z_i)$. The $\epsilon'$-solid tori do not meet $U_{\rm cyl}$,
 and $\epsilon'$-solid tori are pairwise disjoint.

\subsection{Deforming the splitting surfaces}

At this point we have constructed a decomposition $M_n=W'_{n,1}\cup
W'_{n,2}$ where the $W'_{n,i}$ are compact submanifolds meeting along
their boundary. We must modify $W'_{n,1}$ and $W'_{n,2}$ in order to
form $V_{n,1}$ and $V_{n,2}$ as required by Theorem~\ref{1Dthm'}.
There are two steps in this modification. The first involves changing the
boundary surfaces between $W'_{n,1}$ and $W'_{n,2}$ slightly so that they are
well-positioned with respect to the circle fibration on $U'_2$. It is carried out in this
section. The other involves removing $\epsilon'$-solid tori and chains of $\epsilon'$-solid cylinders from $W'_{n,2}$. It is carried out in the two subsections after this one.

\subsubsection{Interface with the $3$-balls near boundary corners}

Let us  deform the boundaries of the
$B(x_i)$  slightly until they
are the (overlapping) union of an annulus in $U'_2$ saturated under the $S^1$-fibration and
two disks, each disk spanning an $\epsilon'$-solid cylinder.

\begin{lem}\label{ballinterface}
Let $B(x_i)$ be a $3$-ball near a $2$-dimensional boundary corner contained in the collection
given in Lemma~\ref{Acover2} or Lemma~\ref{Acover1}. Then
there are disjoint $\epsilon'$-solid cylinders $\nu(1)=\nu(1,i)$ and $\nu(2)=\nu(2,i)$ of scale $s_1$, of length $s_1/2$
and of width $\xi s_1/2$ such that the centers of their cores meet the metric sphere $S(x_i,r(x_i)/4)$.
Furthermore, there is a $2$-sphere
$S(x_i)\subset B(x_i,3r(x_i)/8)\setminus \bar B(x_i,r(x_i)/4)$ that is
the (overlapping) union of an annulus $A(x_i)$ and two $2$-disks, $D_1$ spanning  $\bar \nu(1)$, and  $D_2$ spanning in $\bar \nu(2)$. These satisfy the following:
\begin{enumerate}
\item The annulus $A(x_i)$ is contained in $U'_2$ and is saturated
under  $S^1$-fibration on $U'_2$.
\item One of the boundary circles of $A(x_i)$ is contained in  $\nu(1)$ and the other is contained in $\nu(2)$.
\item For $j=1,2$, the intersection of $D_j$ with $A(x_i)$ is an
annulus which is a collar neighborhood in $D_j$ of $\partial D_j$
and is a collar neighborhood in $A(x_i)$ of one of its boundary
components.
\item Every point of $S(x_i)\setminus A(x_i)$ is contained in the sub-cylinder of one of the  $\nu(i)$  of width $\xi s_1/8$.
    \item For $j=1,2$ and for any $t\in [3\xi s_1/8,\xi s_1/2]$ the intersection of $S(x_i)$ with $h_{\widetilde\gamma_j}^{-1}(t)$ is a circle separating the ends of the level set $h_{\widetilde\gamma_j}^{-1}(t)$ in $\bar\nu_j$.
\end{enumerate}
 The $2$-sphere $S(x_i)$ is isotopic in $B(x_i,3r(x_i)/8)$ to the metric sphere
$S(x_i,r(x_i)/4)$. In particular, the $2$-sphere $S(x_i)$ separates
the metric sphere  $S(x_i,3r(x_i)/8)$ from the metric sphere $S(x_i,r(x_i)/4)$. As a result it bounds a closed
topological $3$-ball $\bar B'(x_i)\subset B(x_i,3r(x_i)/8)$.
\end{lem}

\begin{proof}
Since $B(x_i)$ is near a $2$-dimensional boundary corner, there is a $2$-dimensional Alexandrov ball
$\bar B$ of radius $1$ that is boundary $\mu$-good at some $\bar x$ on scale $r(x_i)$ of angle $\le \pi-\delta$ and $(B(x_i,1),x_i)$ is within $\hat\epsilon$ of $(\bar B,\bar x)$.
The metric sphere $S(\bar x,r(x_i)/4)$ is a topological interval that meets the boundary of $\bar B$ in its endpoints.
Let $\gamma_1$ and $\gamma_2$ be geodesics of length $s_1$ in $\bar B$ whose endpoints lie in the boundary of $\bar B$
and whose central points lie in $S(\bar x,r(x_i)/4)$ near to the two boundary points of this metric sphere.
For $j=1,2$, let $\widetilde \gamma_j$ be geodesics in $B(x_i,r(x_i))$ of length $s_1$ within $\hat\epsilon$ of the $\gamma_j$.
We can arrange that the central points of the $\gamma_j$ lie on the metric sphere $S(x_i,r(x_i)/4)$
For $j=1,2$ let $\nu(j)=\nu(j,i)$ be the $\epsilon'$-solid cylinders associated with the $\widetilde\gamma_j$ of length $s_1/2$ and width $\xi s_1/2$. By
construction the intersection of $S(x_i,r(x_i)/4)$ with $\nu(j)$ contains the central point of $\widetilde\gamma_j$. Consider the saturated open subset $U'_2(x_i)$ of $U'_2$
consisting of all fibers of the  $S^1$-fibration on $U'_2$ that meet
$E=B(x_i,(.002)s_1+r(x_i)/4)\setminus \bar B(x_i,(.001)s_1
+r(x_i)/4)$. This open subset contains the complement in $E$ of the
cores of $\nu(1)$ and $\nu(2)$ and is contained in
$B(x_i,3r(x_i)/8)\setminus \bar B(x_i,r(x_i)/4)$. For $j=1,2$ fix a point
$y_j$ in the intersection of the level set $h_{\widetilde\gamma_j}^{-1}((.11)\xi s_1)$ and the central disk of $\nu(j)$. According to
Lemma~\ref{flatboundary} here is a geodesic $\zeta_j$ from $y_j$ to
a point $z$ at distance $(1.1)\xi s_1$ from
$\widetilde \gamma_j$ with the property that for any $w\in \zeta\cap h_{\widetilde\gamma_j}^{-1}((.51)\xi s_1$ the comparison angle
$\tilde\angle\widetilde\gamma_j w z\ge \pi-2\xi$. In particular,
this geodesic meets each level set of $h_{\widetilde\gamma_j}^{-1}(t)$
for $(.11)\xi s_1\le t\le \xi s_1/2$ in a single point. Of course, $\zeta\subset U'_2(x_i)$. Let
$\hat A_j$ be the annulus that is the saturation of $\zeta\cap \bar h_{\widetilde\gamma_j}^{-1}([0,(.51)\xi s_1])$
under the circle fibration on $U'_2(x_i)$. This annulus crosses
each level set of $h_{\widetilde\gamma_i}^{-1}(t)$ for $(.12)\xi s_1\le t\le \xi s_1/2 $ in $\bar\nu(j)$ in a single circle, a circle that separates the ends of that level set.

 The base space of the circle fibration of $U'_2(x_i)$ is a connected surface $\Sigma$. Consider
the saturated open subset $V'_2(x_i)\subset U'_2(x_i)$ which is the union
of all fibers that meet the complement of $\bar\nu(1)\cup
\bar\nu(2)$. It is also connected as is its quotient surface
$\Sigma'\subset \Sigma$. It follows that there is a saturated
annulus $\widehat A_0\subset \Sigma'$  that connects orbits over the
intersection of $\zeta_j\cap h_{\widetilde\gamma_j}^{-1}((.51)\xi s_1)$, for $j=1,2$
and which is disjoint from the union of the orbits over points of $\zeta_j\cap h_{\widetilde \gamma_j}^{-1}([0,(.51)\xi s_1)$.
 The union $\widehat A_1\cup\widehat A_0\cup \widehat A_2$ is
an annulus $\widehat A$. Since in the rescaling of the metric giving
the $S^1$-product structure the fiber circles of the fibration
structure lie within $\epsilon'$ of an $S^1$-factor in an $\epsilon$-product neighborhood, it follows
immediately that  the boundary circles of $\widehat A$ are contained $\cup_{j=1,2}h_{\widetilde\gamma_j}^{-1}([0,(.121)\xi s_1])$. For $j=1,2$, consider
the level set $L(j)=
h_{\widetilde\gamma_j}^{-1}((.122)\xi s_1)$ and the intersection of $c_j=\widehat A \cap L(j)$.
The circle $c_j$
separates (in $\widehat A$) the boundary component of $\widehat A$ contained in $
\nu(j)$ from the intersection of $\widehat A$ with the side of $\bar \nu(j)$.
 We define
$A'(x_i)$ as the subannulus of $\widehat A$ bounded by $c_1$ and $c_2$.
There is a disk $D'_j\subset h_{\widetilde\gamma_j}^{-1}([0,(.122)\xi s_1])$
with boundary $c_j$. We define $D_j$ to be the union of $D'_j$ and the intersection
  of $A'(x_i)\cap \bar\nu(j)$. We set  $A(x_i)$ equal to the sub-annulus of $A'(x_i)$ bounded
  by the $S^1$-fibers of $U'_2$ passing through the point on $\zeta\cap h_{\widetilde\gamma_{j}}^{-1}((.123)\xi s_1)$. Then $A(x_i)$,  the disks $D_j$ and the union $S(x_i)=D_1\cup A(x_i)\cup D_2$
  are as required by the lemma.
\end{proof}

\subsubsection{Other interfaces}

\begin{lem}\label{prodinterface}
Suppose that $\lambda_n\ge \rho_n^{-2}(x_n)$ and that $B_{\lambda_ng_n}(x_n,1)$
is within $\hat\epsilon$ of a standard $2$-dimensional Alexandrov ball of area $\ge a$. Suppose also that for some $\alpha$ with $1/100\le \alpha\le 1/2$ the function $d(x_n,\cdot)$ is $(1-\epsilon)$-regular on
$C(x_n,\alpha,\alpha+1/100)=B_{\lambda_ng_n}(x_n,\alpha+1/100)\setminus
\bar B_{\lambda_ng_n}(x_n,\alpha)$ and determines a fibration of
this open set over the interval $(\alpha,\alpha+1/100)$ with fiber either
$T^2$ or $S^2$.
\begin{enumerate}
\item If the fibers of the restriction of $d(x_n,\cdot)$ to
$C(x_n,\alpha,\alpha+1/100)$ are $2$-tori, then there is a $2$-torus
$T\subset U'_2 \cap A(x_n,\alpha,\alpha+1/100)$ that is saturated
under the  $S^1$-fibration on $U'_2$ and that separates the metric spheres
$S_{\lambda_ng_n}(x_n,\alpha)$ and
$S_{\lambda_ng_n}(x_n,\alpha+1/100)$. This $2$-torus is isotopic to
the fibers of $d(x_n,\cdot)$ on $C(x_n,\alpha,\alpha+1/100)$.
\item  If the fibers of the restriction of $d(x_n,\cdot)$ to
$C(x_n,\alpha,\alpha+1/100)$ are $2$-spheres, then there is a
$2$-sphere $S$ in $C(x_n,\alpha,\alpha+1/100)$ that is the union of
an annulus $A(x_n)$ in $U'_2$, an annulus saturated under the
$S^1$-fibration, and disks $D_1$ and $D_2$ in two
$\epsilon'$-solid cylinders, $\bar\nu(1)$ and $\bar \nu(2)$.
The $D_j$, $A(x_n)$, and $S(x_n)=D_1\cup A(x_n)\cup D_2$ satisfy Properties 1 -- 5
listed in Lemma~\ref{ballinterface}.
 The $2$-sphere $S(x_n)$ separates the metric
spheres $S_{\lambda_ng_n}(x_n,\alpha)$ and
$S_{\lambda_ng_n}(x_n,\alpha+1/100)$.
\end{enumerate}
\end{lem}

\begin{proof}
Let $W$ be the union of $C'(x_n)=C(x_n,\alpha+10^{-3},\alpha+(.0099))$ together with all $\epsilon'$-solid cylinders and $3$-balls near $2$-dimensional corner points that meet $C'(x_n)$. Then $W\subset C(x_n)$.
First, we consider the case when the fibers of $d(x_n,\cdot)$ in
$C(x_n)=C(x_n,\alpha,\alpha+1/100)$ are $2$-tori.

\begin{claim}
In this case, there is no $3$-ball near a $2$-dimensional boundary corner and no $\epsilon'$-solid
cylinder that meets $C'(x_n)$.
\end{claim}

\begin{proof}
 Suppose that there is at least one
$\epsilon'$-solid cylinder or $3$-ball meeting $C'(x_n)$. Then by
Van Kampen's theorem, the fundamental group of $W$ is the quotient of
the fundamental group of a Seifert fibration over a non-compact,
connected surface when the class of the generic fiber is set equal
to the trivial element. Hence, the fundamental group is a free
product of cyclic groups. On the other hand, since
$S(x_n,\alpha+1/200)\subset W\subset C(x_n,\alpha,\alpha+1/100)$, it
follows that the fundamental group of $S(x_n,\alpha+1/200)$ is
identified with a subgroup of the fundamental group of $W$. Since
$S(x_n,\alpha+1/200)$ is a $2$-torus, this is a contradiction.
\end{proof}

This proves that $C'(x_n)$ is contained in the total space $X\subset U'_2$ of a Seifert fibration, and $X\subset C(x_n)$. In fact $X$ is the union of a saturated subset of $U'_2$ and a finite union of $\epsilon'$-solid tori $\hat\tau(z_i)$.
 Since the complement of a compact subset
of each $\hat \tau(z_i)$ is contained in $U'_2$, it follows that the
complement of a compact subset of $X$ is contained in $U'_2$ and
hence there is a compact submanifold $X'\subset X$, containing
$S(x_n,\alpha+1/200)$ whose boundary is contained in $U'_2$ and is
saturated under the  $S^1$-fibration. Since $X'$ contains
$S(x_n,\alpha+1/200)$, if follows that $X'$ separates the ends of
$C(x_n,\alpha,\alpha+1/100)$ and hence one of the boundary
components of $X'$ also separates these ends. This boundary
component is a $2$-torus contained in $U'_2$,  saturated under the
$S^1$-fibration. It separates the boundary components of
$C(x_n,\alpha,\alpha+1/100)$. Consequently, it is isotopic in
$C(x_n,\alpha,\alpha+1/100)$ to any fiber of $d(x_n,\cdot)$. This
completes the proof in the case when $C(x_n)$ is fibered by $2$-tori.

Next, we consider the case when the fiber $S(x_n,\alpha+1/200)$ is a
$2$-sphere.
 The first thing to observe is that this $2$-sphere is
not isotopic in $X$ to a $2$-sphere that is disjoint from the
$\epsilon'$-solid cylinders and $3$-ball neighborhoods. The reason is that any
$2$-sphere contained in a Seifert fibered $3$-manifold  with base
a connected, non-compact surface is homotopically trivial in that manifold,
but $S(x_n,\alpha+1/200)$ is not homotopically trivial in $X$.  Since
$d(\bar x,\cdot)$ is $(1-2\epsilon)$-regular on the metric annulus $\bar C=C(\bar x,\alpha,\alpha+/100)$, it follows that the level sets of this function on $\bar C$ are intervals.

Any level set of $d(x_n,\cdot)$
meeting the center of an $\epsilon'$-solid cylinder,  $C'(x_n)$,  intersects the solid cylinder in a spanning $2$-disk. We fix such two such $\epsilon'$-solid cylinders $\bar\nu(1)$ and $\bar\nu(2)$ with central geodesics $\widetilde\gamma_1$ and $\widetilde\gamma_2$ where the corresponding geodesics in $\bar B(\bar x,1)$ are near the two boundary components of $\bar C$. We construct a saturated annulus $\widehat A_j$ extending these by disks $D_j$  as in Lemma~\ref{ballinterface} so that the intersection of $\widehat A_j\cup D_j$
with $\bar\nu(j)$ is a spanning disk.
Since removing the $3$-balls near $2$-dimensional corner points and $\epsilon'$-solid tori from $W$ cannot disconnect it, it follows that there is a saturated annulus $A_0$ in $U'_2\cap C(x_n)$ connecting
the outer boundaries of $\widehat A_1$ and $\widehat A_2$. The union $D_1\cup A_0\cup D_2$ is the $2$-sphere as required.
\end{proof}

This leads immediately to the following two results.

\begin{cor}\label{T2split}
Let $A$ be a component of $M_n\setminus U'_{n,1}$ that is close to a
$2$-dimensional space. Let ${\mathcal E}$ be an end of $U'_{n,1}$
neighboring $A$. Suppose that  the end $U(x_{\mathcal E})$ is fibered by $2$-tori.
Then there is a $2$-torus $T({\mathcal E})\subset U^+(x_{\mathcal E})\cap
U'_2$ that is saturated under the $S^1$-fibration and separates
the ends of $U(x_{\mathcal E})$. In particular, $T({\mathcal E})$ is
isotopic in $U^+(x_{\mathcal E})$ to the boundary $2$-torus $\Sigma({\mathcal E})$ of
$U^+(x_{\mathcal E})$.
\end{cor}

\begin{cor}\label{S2interface}
Suppose that $A$ is a component of $M_n\setminus U'_{n,1}$ that is
close to a $2$-dimensional space, and suppose that ${\mathcal E}$ is
an end of $U'_{n,1}$ neighboring $A$ and that $U(x_{\mathcal E})$
is fibered by $2$-spheres.
Then there is a $2$-sphere $S({\mathcal E})\subset U^+(x_{\mathcal E})$
that is the union of an annulus $E\subset U'_2$ saturated under the
 $S^1$-fibration and two $2$-disks, $D_1$ and $D_2$, contained in
 two $\epsilon'$-solid cylinders.
These satisfy the following:
\begin{enumerate}
\item Each boundary circle of $E$ is contained in the interior of one of the  $\nu(j)$ and in fact lies in the sub solid cylinder of width $\xi s_1/4$.
\item $S({\mathcal E})$ is isotopic in $U^+({\mathcal E})$ to $\Sigma({\mathcal E})$.
\end{enumerate}
\end{cor}

\subsubsection{Case of components near $1$-dimensional spaces
but which expand to be near $2$-dimensional spaces}

Now we need to perform a similar construction for each component
$A$ of $M_n\setminus \hat U'_{n,1}$ which is close to an
interval but which expands to be near a $2$-dimensional space.
Then there is exactly one
neighboring component of $U'_{n,1}$ and the end ${\mathcal E}$ of
this component neighboring $A$ is either fibered by
$2$-tori or fibered by $2$-spheres. Furthermore, there is
 a point $x\in \widehat A$ and a constant $\lambda>\rho^{-2}_n(x)$
 such that $B_{\lambda g_n}(x,1)$ is
close to a standard $2$-dimensional ball $\bar B=B(\bar x,1)$ and on
the region between the metric sphere $S_{\lambda g}(x,1/2)$ and
$\Sigma({\mathcal E})$ distance function from $x$ is
regular. In particular, the region
bounded by $S_{\lambda g}(x,1/2)$ and $\Sigma({\mathcal E})$ is
homeomorphic to a product of a closed surface (either a $2$-sphere or a $2$-torus)
with a closed interval, and the distance function from $x$
is regular on this region and gives this product structure.

Let us consider first the case when the annular region $C_{\lambda
g_n}(x,1,2)$ is fibered by $2$-tori. Applying
Lemma~\ref{prodinterface} to the annular region $C_{\lambda
g_n}(x,1/2,1)$ we see that in this case there is a $2$-torus
$T({\mathcal E})$ in $C(x,1/2,1)\cap U'_2$ that is saturated under
the circle fibration on $U'_2$ and separates the metric
spheres $S_{\lambda g_n}(x,1/2)$ and $S_{\lambda g_n}(x,1)$. It
follows that $T({\mathcal E})$ is isotopic in $C_{\lambda
g_n}(x,1/2,1)$ to either end and consequently the region between
$T({\mathcal E})$ and $\Sigma({\mathcal E})$ is a product region.

Let us consider now the case when the annular region $C_{\lambda
g_n}(x,1,2)$ is fibered by $2$-spheres. Again applying
Lemma~\ref{prodinterface} we see that there is a $2$-sphere
$S({\mathcal E})\subset C_{\lambda g_n}(x,1,2)$ that separates the
ends and has the properties stated in the second part of that lemma.
The region between $S({\mathcal E})$ and $\Sigma({\mathcal E})$ is
a product region.

\subsubsection{Redefinition of the boundary between $W'_{n,1}$ and
$W'_{n,2}$}

Now we deform slightly $W'_{n,1}$ and $W'_{n,2}$ so as to replace
the splitting surfaces $W'_{n,1}\cap W'_{n,2}$ by the surfaces
constructed in the previous sections.

For each end ${\mathcal E}$ of $U'_{n,1}$ we have a surface either a
$2$-torus $T({\mathcal E})$ that is contained in $U'_2$ and is
saturated under the  $S^1$-fibration or a $2$-sphere
$S({\mathcal E})$ that is the union of an annulus in $U'_2$
saturated under the  $S^1$-fibration and two spanning disks
in  $\epsilon'$-solid cylinders. In all cases this
surface is contained in $W'_{n,2}$ and is parallel to the surface
$\Sigma({\mathcal E})$ which is a splitting surface for $W'_{n,1}$
and $W'_{n,2}$. For each end ${\mathcal E}$ we remove the product
region between either $T({\mathcal E})$ or $S({\mathcal E})$
 and $\Sigma({\mathcal E})$ from $W'_{n,2}$ and add it to
$W'_{n,1}$. For each $3$-ball component $\bar B(x_i,r(x_i)/4)$ of
$W'_{n,1}$ we have a surface $S(x_i)$ as constructed in
Lemma~\ref{ballinterface}. It is contained in $W'_{n,2}$ and the
region between it and the metric sphere $S(x_i,r(x_i)/4)$ is a
product. We remove this product region from $W'_{n,2}$ and add it to
$W'_{n,1}$

After making all these changes we relabel the results $W'_{n,1}$ and
$W'_{n,2}$. What we have achieved is to make each component of the intersection
either a $2$-torus contained in $U'_2$ and saturated
under the  $S^1$-fibration or a $2$-sphere that is the union
of an annulus in $U'_2$ that is saturated under the
$S^1$-fibration and two $2$-disks spanning
$\epsilon'$-solid cylinder neighborhoods.
Also, each component of $W'_{n,1}$ satisfies the properties required
of $V_{n,1}$ in Theorem~\ref{1Dthm'}.
{\bf At this point we define
 $V_{n,1}$ to be $W'_{n,1}$.}

  Here is our progress to date.

\begin{cor}\label{progress}
We have a decomposition $M_n=W'_{n,2}\cup V_{n,1}$. The intersection
$W'_{n,2}\cap V_{n,1}$ is the boundary of $W'_{n,2}$ and is the
union of all boundary components of $V_{n,1}$ that are not boundary
components of $M_n$. Each component of $V_{n,1}$ is as listed in
Theorem~\ref{1Dthm'}. Furthermore, the intersection $W_{n,2}'\cap
V_{n,1}$ consists of $2$-tori contained in $U'_2$ and saturated
under the $S^1$-fibration structure on $U'_2$ and $2$-spheres that are unions of annuli
contained in $U'_2$ and saturated under the  $S^1$-fibration
and two spanning $2$-disks in  $\epsilon'$-solid
cylinders.
\end{cor}

\subsection{Overlaps of $\epsilon'$-solid cylinders}

Our next step is to show that we can arrange that the complement of
the union of $U'_2$, the $\epsilon'$-solid tori near interior cone points and
 the $3$-balls near $2$-dimensional corner points referred to in Lemmas~\ref{Acover2}
and~\ref{Acover1} is contained in the union of  a finite set of cores of $\epsilon'$-solid cylinders
and these cylinders have  good intersections with each
other and with the $3$-balls and with the $2$-spheres $S({\mathcal E})$ associated to ends
${\mathcal E}$. To do
this we introduce the notion of chains of these neighborhoods.

\subsubsection{Chains of solid cylinders}

Let us begin with the local structure, namely, two solid cylinders meeting
nicely.
Let $y\in M_n$ and  $\lambda\ge \rho_n(y)^{-2}$ be given. Suppose
that $B_{\lambda g_n}(y,1)$ is within $\hat\epsilon$ of a standard
$2$-dimensional ball $X=B(\bar x,1)$ of area at least $a$, and
suppose that $\gamma\subset B(\bar x,1/2)$ is a $\xi$-approximation
to $\partial X$ on scale $s_1$ and that $\widetilde\gamma\subset
M_n$ is an $\hat\epsilon$-approximation to $\gamma$. We use the
metric $\lambda g_n$ to measure things, so that in particular,
$\ell=\ell(\widetilde \gamma)$ means the length of $\widetilde
\gamma$ in the metric $\lambda g_n$. Recall from
Proposition~\ref{flatboundary} that for any constant $c\in
[\xi^2,\xi)$ and any $-\ell/2\le a<b\le \ell/2$, the region
$\bar\nu_{c,[a,b]}(\widetilde \gamma)$ is homeomorphic to $D^2\times
I$ and is foliated by its intersections with the level sets of
$f_{\widetilde \gamma}$, each intersection being a $2$-disk.

\begin{defn}
Suppose that we have points $y_1,y_2\in M_n$
 such that $B_{g'_n(y_i)}(y_i,1)$
is within $\hat\epsilon$ of a standard $2$-dimensional ball $B(\bar
x_i,1)$. Suppose that we have geodesics $\gamma_i\subset B(\bar
x_i,1/2)$ that are $\xi$-approximations to the boundary on scale
$s_1$ and suppose that we have geodesics $\widetilde\gamma_i$ that
are $\hat\epsilon$ approximations to $\gamma_i$. We denote $\ell_i$
the length of $\widetilde \gamma_i$ with respect to the metric
$g_n'(y_i)$. Suppose that we have constants $c_i\in [100\xi^2,\xi]$, and
intervals $[a_i,b_i]\subset [-\ell_i/4,\ell_i/4]$. We say that the
$\bar\nu(i)=\bar\nu_{c_i,[a_i,b_i]}(\widetilde\gamma_i)$ have {\em good
intersection} if, possibly after reversing the directions either or
both of the $\widetilde \gamma_i$, the following hold:
\begin{enumerate}
\item The function $f_{\widetilde \gamma_1}$ is an increasing function along $\widetilde\gamma_2$
at any point of $\widetilde \gamma_2\cap\bar\nu_\xi(\widetilde
\gamma_1)$.
\item There is a point in the negative end of
$\bar\nu_{c_2,[a_2,b_2]}(\widetilde\gamma_2)$ that is contained in
$f_{\widetilde \gamma_1}^{-1}(b_1-(.02)\ell_1,b_1-(.01)\ell_1)$ in
$\bar\nu_{c_1,[a_1,b_1]}(\widetilde\gamma_1)$, and the positive end of $\bar\nu(2)$ is disjoint from $\bar\nu(1)$.
\item $c_1\ell_1\rho_n(y_1)$ is either at least
$(1.1)c_2\ell_2\rho_n(y_2)$ or is at most
$(0.9)c_2\ell_2\rho_n(y_2)$.
\end{enumerate}
\end{defn}

\begin{lem}
With the notation above, suppose that for $i=1,2$ the sets
$\bar\nu(i)=\bar\nu_{c_i,[a_i,b_i]}(\widetilde\gamma_i)$ have good
intersection. Then that intersection is homeomorphic to a $3$-ball.
If
\begin{equation}\label{lengthcompar}
c_1\ell_1\rho_n(y_1)<c_2\ell_2\rho_n(y_2),
\end{equation}
then that $3$-ball meets the boundary of $\bar\nu(2)$ in a $2$-disk
contained in the negative end of $\bar\nu(2)$ and the rest of the
boundary consists of an annulus in the side of $\bar\nu(1)$ together
with the positive end of $\bar\nu(1)$. If the reverse inequality holds
in~\ref{lengthcompar}, the similar statements hold with the roles of
$\bar\nu(1)$ and $\bar\nu(2)$ and `positive' and `negative' reversed.
\end{lem}

\begin{proof}
We suppose that Inequality~\ref{lengthcompar} holds. It follows from
Lemma~\ref{flatint} that for all $n$ sufficiently large, the sides
of $\bar\nu(1)$ and of $\bar\nu(2)$ do not intersect and in fact the side of
$\bar\nu(2)$ is disjoint from $\bar\nu(1)$.  Thus, the intersection of
$\bar\nu(1)$ and $\partial \bar\nu(2)$ is contained in the negative
end of $\bar\nu(2)$. By Part 3 of Lemma~\ref{nuint}, this intersection
is a $2$-disk. Hence, it cuts off a $3$-ball in $\bar\nu(1)$.

The other case is analogous.
\end{proof}

\begin{cor}
With notation and assumptions above, suppose that
Inequality~\ref{lengthcompar} holds. Then the boundary of
$\bar\nu(1)\cup\bar\nu(2)$ consists of the union of two subsets:
(i) the disjoint union of the negative end of $\bar\nu(1)$ and the
positive end of $\bar\nu(2)$ and (ii) an annulus $A$. These two
subsets are glued together along their boundaries. The annulus $A$
consists of the union of three annuli glued together along their
boundaries. The first is the intersection of the side of $\bar\nu(1)$
with the complement of the interior of $\bar\nu(2)$. The second is the
negative end of $\bar\nu(2)$ minus its intersection with the interior of
$\bar\nu(1)$ and the third is the side of $\bar\nu(2)$. If the opposite
inequality to Inequality(~\ref{lengthcompar}) holds, then there are
similar statements  with the roles of $\bar\nu(1)$ and $\bar\nu(2)$ and
`positive' and `negative' reversed.
\end{cor}

There are also an analogous definition and results when the balls
are $B_{\lambda g_n}(y_1,1)$ and $B_{\lambda g_n}(y_2,1)$. (Notice
that we are using the same multiple of the metric on the two balls.)
Since the only place in these arguments where we used the fact that
we were dealing with $g'_n(y_i)$ rather than arbitrary multiples of
$g_n$ was when we compared $\rho_n(y_1)$ with $\rho_n(y_2)$. If we
are using the same multiple of the metric for both balls, then the
comparison factor is $1$. We leave the explicit formulations to the
reader.

\subsubsection{Chains}

Now suppose that we have a sequence of $\epsilon'$-solid cylinders
$\{\bar\nu(1),\ldots,\bar\nu(k)\}$, with
$\bar\nu(i)=\bar\nu_{c_i,[a_i,b_i]}(\widetilde\gamma_i)$ as in
Proposition~\ref{flatboundary} with the geodesics $\widetilde
\gamma_i$ oriented. We say that these form a {\em linear chain of
$\epsilon'$-solid cylinders with good intersections} if:
\begin{enumerate}
\item For each $1\le i<k$ the open sets $\bar\nu(i)$ and $\bar\nu(i+1)$ have a good
intersection with the given orientations.
\item If $\bar\nu(i)\cap \bar\nu(j)\not=\emptyset$ for some $i\not= j$, then
$|i-j|=1$.
\end{enumerate}

In addition to linear chains there are circular chains.

\begin{defn}
A {\em circular chain of $\epsilon'$-solid cylinder neighborhoods
with good intersections} is a sequence
$\{\bar\nu(1),\ldots,\bar\nu(k)\}$ of $\epsilon'$-solid cylinder
neighborhoods, indexed by the integers modulo $k$, such that for
each $i,\ 1\le i\le k$, the pair $\{\bar\nu(i),\bar\nu(i+1)\}$ has good
intersections and $\bar\nu(i)\cap \bar\nu(j)\not=\emptyset$ implies
that $j\cong i-1,i\ \text{\rm or}\ i+1\pmod k$.
\end{defn}

\begin{lem}
Suppose that $\{\bar\nu(1),\cdots,\bar\nu(k)\}$ is a linear chain of
$\epsilon'$-solid cylinder neighborhoods with good intersections.
Then $\bar\nu(1)\cup\cdots\cup\bar\nu(k)$ is homeomorphic to a
$3$-ball and its boundary is the union of the negative end of
$\bar\nu(1)$, the positive end of $\bar\nu(k)$ and an annulus $A$.
\end{lem}

\begin{proof}
This is proved easily by induction.
\end{proof}

The same arguments establish the analogue for circular chains.

\begin{lem}
Let $\{\bar\nu(1),\ldots,\bar\nu(k)\}$ be a circular chain of
$\epsilon'$-solid cylinder neighborhoods with good intersections
contained in $M_n$. Then $\bigcup_i\bar\nu(i)$ is a solid torus.
\end{lem}

{\bf From now on {\em a chain of $\epsilon'$-solid cylinders} means
a chain with good intersections.}

\begin{defn}
By a {\em complete} chain of $\epsilon'$-solid cylinders we mean either:
\begin{enumerate}
\item a circular chain
contained in ${\rm int}\,W'_{n,2}$,
or
\item a linear chain with the property that each of the extremal solid cylinders in the chain
meets a $2$-sphere boundary component of $W'_{n,2}$ in a spanning disk and all of the non-extremal solid cylinders in the chain are contained in ${\rm int}\,W'_{n,2}$.
\end{enumerate}
\end{defn}

\begin{prop}
Let $A$ be a component of $M_n\setminus U'_{n,1}$ and let $\widehat
A$ be the associated compact submanifold as defined in
Section~\ref{compactext}. In the covering of $\widehat A$ given in
either Lemma~\ref{Acover2} or Lemma~\ref{Acover1} we can replace the open subset $U_{\rm cyl}$
 by a finite set of cores of $\epsilon'$-solid cylinders. The $\epsilon'$-solid cylinders
  in this collection form a disjoint union of complete chains with good intersection.
 For each boundary component of $\widehat A$ one of the following holds.
\begin{enumerate}\item The boundary component is a $2$-torus  and is  disjoint from all the
$\epsilon'$-solid cylinders.
\item The boundary component is a $2$-sphere and there are exactly two $\epsilon'$-solid cylinders in the collection meet this boundary component. Each meets it in a spanning disk for the solid cylinder.
    \item The width of each $\epsilon'$-solid cylinder in the collection is between $(.40)\xi s_1$ and $\xi s_1/2$.
\end{enumerate}
\end{prop}

\begin{proof}
 For each $2$-sphere boundary component $S$ of $W'_{n,2}$ there are two geodesics $\widetilde \gamma_1$ and $\widetilde \gamma_2$ of length $s_1$ and associated $\epsilon'$-solid cylinders
$\bar\nu_\xi(\widetilde\gamma_1(S))$ and $\bar\nu_\xi(\widetilde\gamma_2(S))$ of length $s_1/2$
that meet $S$ in a spanning disk that itself meets the central $2$-disk of the core of the $\epsilon'$-solid cylinder.
The complement of the cores of these in $S(x_i)$, is contained in $U'_2$
and there is an annulus $A\subset S$ that is saturated under the $S^1$-fibration of $U'_2$
with the property that $S$ is the union of $A$ with the intersection of $S$ with these two $\epsilon'$-solid cylinders.
As we run over all the boundary components $S$ of $W'_{n,2}$ these $\epsilon'$-solid cylinders are disjoint.

Suppose by induction that we have a disjoint collection ${\mathcal D}$ of such
chains of $\epsilon'$-solid cylinders with good intersection
containing all the solid cylinders constructed in the last
paragraph. We also suppose that each extremal solid cylinder in each
chain which is a member of ${\mathcal D}$ either has a free end, as defined
below, or is one of the $\epsilon'$-solid cylinders meeting a boundary component of $W'_{n,2}$ constructed in the last paragraph. By a {\em free
end} of an $\epsilon'$-solid cylinder we mean an end that is contained in  a sub-cylinder  $\bar\nu_{[a,b]}$ of length at least
$s_1/6$ with the sub-cylinder being disjoint from all the other $\epsilon'$-solid cylinders in ${\mathcal D}$ and is also disjoint from the boundary components of $W'_{n,2}$.
Lastly, we assume by induction that each cylinder in ${\mathcal D}$ is of the form
$\bar\nu_{c,[a,b]}(\widetilde\gamma)$ where $(.40)\xi s_1\le c\le \xi s_1/2$.

Suppose that the complement $Y$ in $\widehat A$ of the union of
$U_{2,{\rm gen}}$, the $\epsilon'$-solid tori $B(z_i)$ and
the $3$-ball neighborhoods $B(x_i)$ near $2$-dimensional corners
in the given collection is not
contained in the union of the $\epsilon'$-solid cylinders in the
family ${\mathcal D}$. First, we consider the possibility that one of the
chains in ${\mathcal D}$ has a free end. Let $\bar\nu$ be an extremal member of a
chain of ${\mathcal D}$ with a free end. We take a point $x\in M_n$ that is
disjoint from all the chains of ${\mathcal D}$ and within $\hat\epsilon$ of the
free end of the core of $\bar\nu$. We know that $B_{g'_n(x)}(x,1)$ is
within $\hat\epsilon$ of a $2$-dimensional Alexandrov ball $\bar
B=B(\bar x,1)$. We examine the possibilities for $\bar B$ near $\bar
x$.

Suppose that $\bar B$ was interior $\mu$-good at $\bar x$ on
scale $r$. In this case $\bar\nu$ would be contained in a solid torus
neighborhood near the interior cone point. The level circles
of the end of $\bar\nu$  are almost orthogonal to the horizontal
spaces, and hence these level circles are homotopically non-trivial
in the solid torus, which is absurd since they bound disks in $\bar\nu$
and $\bar\nu$ is contained in the solid torus.

Next, suppose that the free end of $\bar\nu$ is contained in a expanded version
of a ball of
the form $B_{g'_n(x_i)}(x_i,7r(x_i)/8))$ for one of the $3$-balls in our
collection. In the annular region $B_{g'_n(x_i)}(x_i,7r(x_i)/8))\setminus
B_{g'_n(x_i)}(x_i,r(x_i)/4))$ we can form two chains of $\epsilon'$-solid cylinders
with one end of each chain being one of the solid cylinders meeting $S_{g'_n(x_i)}(x_i,r(x_i)/4)$ constructed above. Then the free end of $\bar\nu$
must meet one of these chains.
We can arrange that $\bar\nu$ has good intersection with this chain.
In this way we extend the chain until it ends in the ball, keeping the intersections good.

Next suppose that $\bar B$ is boundary $\mu$-good at a point $\bar
y$ on scale $r(\bar y)$ with $r_0\le r(\bar y)\le 10^{-3}$ and with
angle $\le \pi-\delta$ and with $\bar x\in B(\bar y,r(\bar y)/4)$.
Then there is $y\in M_n$ within $\hat \epsilon$ of $y$ and the
neighborhood $B_{g'_n(x)}(y,r(\bar y))$ that is a $3$-ball
containing the end of $\bar\nu$. But this $3$-ball is contained in one of the
extended balls in our collection, so we have already seen in this case
how to extend the chain, with good intersections, until it meets
one of the boundary components of $W'_{n,2}$.

Suppose that there is a point within $\hat\epsilon$ of the free end of $\bar\nu$
that is in the center of the core of an $\epsilon'$-solid cylinder near
the flat boundary of a $2$-dimensional Alexandrov ball. In this case we can
take such a $\epsilon'$-solid cylinder, $\bar\nu'$ of width $\xi s_1/2$. By arranging its width
(decreasing it by a factor of at most $1.1$) and cutting it off appropriately on the end contained in $\bar\nu$, we can arrange it have good intersection with $\bar\nu$. One possibility is $\bar\nu'$ it meets no other $\epsilon'$-solid cylinder in the set we have constructed so far. In this case the other end of $\bar\nu'$ is a free end, which we can assume has length at least $s_1/6$. The other possibility is that $\bar\nu'$ meets some other $\epsilon'$-solid cylinder, $\bar\nu(1)$, in the given set. Then, as we move along the geodesic $\widetilde \gamma'$ away from $\bar\nu$, there is a first such $\epsilon'$-solid cylinder $\bar\nu(1)$ that the geodesic meets. Again cutting off $\bar\nu'$ appropriately, and possibly decreasing its width by a factor of at most $1.1$, we can arrange that the intersection of $\bar\nu'$ and $\bar\nu(1)$ is good, without destroying the fact that the intersection of $\bar\nu$ and $\bar\nu'$ is good. In this case we have extended the chain so that it together with the other $3$ sets contains a neighborhood of a fixed unrescaled size of the free end of $\bar\nu$. The width of $\bar \nu'$ is between $(.40) \xi s_1$ and $\xi s_1/2$.

This shows that in all cases we can extend the chain if it has a free end.
Now let us consider the case when the chains that we have constructed have no free ends yet the cores of the chains do not cover the complement of the other three sets.
Then we simply take a point not in the union of the other $3$ open sets. It lies at the center of the core of an $\epsilon'$-solid cylinder of rescaled length $s_1/2$ near a flat boundary. We simply add this $\epsilon'$-solid cylinder to the collection. Since there are no free ends of the pre-existing chains, this solid cylinder is disjoint from all the previous ones.

After a finite number of repetitions of these two constructions, we have created chains of $\epsilon'$-cylinders as required that cover the complement of the other three open sets.

 Since at each step we need only make the thickness of the $\epsilon'$-cylinder differ by a factor of $1.1$
 from two given numbers, we can arrange that all the intersections are good by taking widths of the $\epsilon'$-solid cylinders to be between $(.40)\xi s_1$ and $\xi s_1/2$.
\end{proof}

\subsubsection{Refinements of chains}

At this point we have covered $W'_{n,2}$ by $U'_2$, by $\epsilon'$-solid tori
and by chains of $\epsilon'$-solid cylinders with good intersection. The union of $U'_2$ and the $\epsilon'$-solid tori has the structure of a Seifert fibration with the (possible) exceptional fibers along the cores of the $\epsilon'$-solid tori. A circular chain of $\epsilon'$-solid cylinder neighborhoods
is homeomorphic to a solid torus, and a linear chain is homeomorphic to $D^2\times I$  meeting the boundary of $W'_{n,2}$ in spanning disks contained in $2$-sphere boundary components. The frontiers of these chains in $W'_{n,2}$ are contained in $U'_2$. The next step is
is to perturb these chains slightly to isotopic embeddings so that their frontiers in $W'_{n,2}$ are saturated under the $S^1$-fibration structure on $U'_2$. For this we first construct slightly smaller versions of these chains, called refinements.

\begin{defn}
Let
$\bar\nu(i)=\bar\nu_{c_i,[a_i,b_i]}(\widetilde\gamma_i)$, for
$i=1,\cdots,k$ be a chain of $\epsilon'$-solid cylinder neighborhoods with good intersection.
  Consider a consecutive pair $\nu(i),\nu(i+1)$. If
Inequality~\ref{lengthcompar} holds then we set
$$\bar\nu'(i)=\bar\nu_{(c_i/2),[a_i,b_i]}(\widetilde\gamma_i)$$ and
$$\bar\nu'(i+1)=\bar\nu_{(c_{i+1}/2),[a_{i+1}+(.001)\ell_i,b_{i+1}]}(\widetilde\gamma_{i+1}).$$
If the opposite inequality holds then we set
$$\bar\nu'(i)=\bar\nu_{(c_i/2),[a_i,b_i-(.001)\ell_{i+1}]}(\widetilde\gamma_i)$$
and
$$\bar\nu'(i+1)=\bar\nu_{(c_{i+1}/2),[a_{i+1},b_{i+1}]}(\widetilde\gamma_{i+1}).$$
Thus, we halve the width of both the solid cylinders and the truncate the end of the larger one.
We perform an analogous operation for each pair of successive solid cylinders, so that it is possible that both ends of $\nu(i)$ are truncated, only one end is truncated, or neither end is truncated.
In all cases the width of $\nu(i)$ is halved so as to become $c_i/2$.
The result is called a
{\em refinement} of the chain $\{\bar\nu(1),\cdots,\bar\nu(k)\}$. The
boundary of $\bar\nu'(1)\cup\cdots\cup \bar\nu'(k)$ consists of the negative
end of $\bar\nu'(1)$ union the positive end of $\bar\nu'(k)$ union
an annulus $A'$ analogous to $A$.
\end{defn}

It is easy to establish the following by induction.

\begin{lem}\label{prodstructure}
 Suppose that $\{\bar\nu(1),\cdots,\bar\nu(k)\}$ is either a linear chain
or a circular chain of $\epsilon'$-solid cylinder neighborhoods with
good intersections where
$\bar\nu(i)=\bar\nu_{c_i,[a_i,b_i]}(\widetilde\gamma_i)$. Then there
is a refinement $\{\bar\nu'(1),\ldots,\bar\nu'(k)\}$ of this chain.
Furthermore:
\begin{enumerate}
\item If the chain is a linear chain, then
$\left(\cup_{i=1}^k\bar\nu(i)\right)$ is homeomorphic to a $3$-ball
and the $2$-sphere $\partial \left(\cup_{i=1}^k\bar\nu(i)\right)$ is
made up of the negative end, $D_-(1)$, of $\bar\nu(1)$, the positive
end, $D_+(k)$, of $\bar\nu(k)$ and an annulus $A$ with $A$ meeting each of
$D_-(1)$ and $D_+(k)$ along its boundary circle. Similarly,
$\partial \left(\cup_{i=1}^k\bar\nu'(i)\right)$ is a $2$-sphere
consisting of the negative end $D'_-(1)$ of $\nu'_1$, the positive
end $D'_+(k)$ of $\nu'(k)$ and an annulus $A'$ meeting each of
$D'_-(1)$ and $D'_+(k)$ along its boundary circle. Furthermore,
there is a homeomorphism
$$\left(\cup_{i=1}^k \bar\nu(i)\setminus
{\rm int}\,\left(\cup_{i=1}^k\nu'(i)\right),A,A'\right)\equiv
\left(A\times I,A\times \{0\},A\times \{1\}\right).$$
\item If $\{\nu(1),\cdots,\nu(k)\}$ is a circular chain, then
$$\partial \left(\cup_{i=1}^k \nu(i)\right)$$ is homeomorphic to a
$2$-torus and
$$\bigcup_{i=1}^k\bar\nu(i)\setminus{\rm int}\,\left(\bigcup_{i=1}^k\nu'(i)\right)$$
is homeomorphic to $T^2\times I$.
\end{enumerate}
Lastly, for each $i$ let $g_{n,i}$ be the multiple of $g_n$ that is
used in defining $\nu(i)$, let $\widetilde \gamma_i$ be the central
geodesic of $\nu(i)$ and let $\ell_i$ be the length of $\widetilde
\gamma_i$ in the metric $g_{n,i}$. In each case the distance,
measured in $g_{n,i}$, from any point of $A'\cap \nu(i)$ to $A$ is at
least $\xi^2\ell(\widetilde\gamma_i)$.
\end{lem}

\subsection{Removing solid tori and solid cylinders from $W'_{n,2}$}

At this point we have constructed the compact submanifold $V_{n,1}$
and shown that it has all the properties  required by
Theorem~\ref{1Dthm'} but we still must modify $W'_{n,2}$ in order to
product $V_{n,2}$. To produce $V_{n,2}$ as required by
Theorem~\ref{1Dthm'} we shall remove  solid tori and solid cylinders
from $W'_{n,2}$.

The compact set $W'_{n,2}$ is covered by  (i) $U'_2$, (ii) a finite
number of $\epsilon'$-solid tori $B(z_i)=B(z_i,r(z_i)/4)$ near interior
singular points, and (iii) a finite number of chains of refinements
of $\epsilon'$-solid cylinders. Our first approximation to
$V_{n,2}$, we call it $V'_{n,2}$, is to remove from $W'_{n,2}$ the
union of:
\begin{enumerate}
\item the interiors of all the refinements of $\epsilon'$-solid cylinders in the
given chains and
\item the interiors of the $\epsilon'$-solid tori $B(z_i)$ in the
collection.
\end{enumerate}

This does not produce $V_{n,2}$ because, even though the frontiers of these components are contained in $U'_2$, they are not saturated under the $S^1$-fibration on $U'_2$.
In order to define $V_{n,2}$, we must  deform the solid tori and
solid cylinders that we remove from $W'_{n,2}$ slightly so as to
arrange that their frontiers in $W'_{n,2}$ are saturated under this action.

\subsubsection{Removing solid tori near interior cone points}

 For the solid tori near interior cone
points it is clear what to do. According to
Corollary~\ref{invtorus}, for each $\epsilon'$-solid torus
$B(z_i)$ near an interior cone point the neighborhood of
the boundary of this solid torus contains a $2$-torus $T(z_i)\subset
U'_2\cap B(z_i,3r(z_i)/8)$ that is saturated under the
$S^1$-fibration and bounds a solid torus $\widehat T (z_i)$ in
$B(z_i,3r(z_i)/8)$.  Instead of removing $B(z_i,r(z_i)/4)$ from
$W'_{n,2}$ we remove the interior of $\widehat T(z_i)$. We are
removing slightly larger solid tori, but this does not change the
topological type since the region between $T(z_i)$ and the metric
sphere $S(z_i,r(z_i)/4)$ is a product region, a region homeomorphic
to $T^2\times I$. The boundary component created by removing the
interior of $\widehat T(z_i)$ is a $2$-torus saturated under
the $S^1$-fibration structure on $U'_2$.

\subsubsection{Removing perturbations of chains of $\epsilon'$-solid cylinders}

We wish to make an  analogous removals of perturbations
of the chains of $\epsilon'$-solid cylinders, perturbed so that their frontiers in $W'_{n,2}$
are saturated under the $S^1$-fibration structure on $U'_2$. To do so  requires
more argument.

Let ${\mathcal C}$ be a circular chain of $\epsilon'$-solid
cylinders contained in $W'_{n,2}$ and let ${\mathcal C'}$ be the
given refinement of it. Then by Lemma~\ref{Acover1}, the  union
$\widehat T_{\mathcal C}$ of the solid cylinders in ${\mathcal C}$
minus the union of the interiors ${\rm int}\, \widehat T_{{\mathcal
C}'}$ of the solid cylinders in ${\mathcal C}'$ is homeomorphic to
$T^2\times I$. Furthermore, $\widehat T_{\mathcal C}\setminus {\rm int}\,
\widehat T_{{\mathcal C}'}$ is contained in $U'_2$.
We consider the union of all fibers of the $S^1$-fibration on
$U'_2$ that either meet the complement of $\widehat T_{\mathcal C}$
or are closer to this complement than they are to $\widehat
T'_{\mathcal C}$. This is an open subset $\Omega$ of $U'_2$ that is
saturated under the $S^1$-fibration. Since the distance between any point
$x\in \widehat T'_{\mathcal C}\cap \nu(i)$ and the complement of $\widehat
T_{\mathcal C}$, when measured in the multiple of the metric $g_n$ used to define $\nu(i)$, is at least
$\ell(\widetilde\gamma_i)\xi^2\ge s_1\xi^2/10$, it follows from
the fact that $\hat \epsilon<10^{-3}\xi^2s_1/C$, Lemma~\ref{generic2D}
and Proposition~\ref{s1glue} that  $\Omega$ contains $\partial
\widehat T_{\mathcal C}$ and is disjoint from $\widehat T'_{\mathcal
C}$. It then follows that we can find a compact $3$-manifold $\Omega_0$ with
boundary contained in $\Omega$ that is saturated under the
$S^1$-fibration on $U'_2$. Then $\Omega_0$  contains $\partial \widehat
T_{\mathcal C}$ and of course is disjoint from $\widehat
T_{{\mathcal C}'}$. One of the boundary components, $T({\mathcal C})$,
of $\Omega_0$ must then separate $\partial \widehat T_{\mathcal C}$
from $\partial \widehat T_{{\mathcal C}'}$. Since this boundary
component  is fibered by circles and is orientable, it is
diffeomorphic to a $2$-torus. Since the region $\widehat T_{\mathcal
C}\setminus {\rm int}\, \widehat T_{{\mathcal C}'}$ is homeomorphic to
$T^2\times I$. Any $2$-torus contained in this region that separates
the boundary components, e.g., $T({\mathcal C})$, is topologically
isotopic in $\widehat T_{\mathcal C}$ to either boundary component.
It then follows that $T({\mathcal C})$ bounds a solid torus
$\tau_{\mathcal C}$ contained in $T_{\mathcal C}$.

Now let us consider a linear chain ${\mathcal C}$ of
$\epsilon'$-solid cylinders. Let $\widehat C({\mathcal C})$ be the union of the
closed $\epsilon'$-solid cylinders in this chain. By Lemma~\ref{prodstructure} the submanifold
$\widehat C({\mathcal C})$ is homeomorphic to $D^2\times I$. Let $\widehat
C'({\mathcal C})$ be the union of the solid cylinders in the
refinement. Denote by $X({\mathcal C})$ the complement $\widehat
C({\mathcal C})\setminus {\rm int}\,\widehat C'({\mathcal C})$ and
by $E(X)$ the ends of $X$, i.e., the intersection of $X$ with the
ends of $\widehat C({\mathcal C})$. According to
Lemma~\ref{prodstructure}, the pair $(X,E(X)$ is homeomorphic to
$(S^1\times I\times I, S^1\times
\partial I\times I)$.
 Also, since the distance between  any point $x\in \widehat C'({\mathcal C})\cap \nu(i)$
and the complement of   the $\widehat C({\mathcal C})$, when measured in the multiple of the metric
$g_n$ used to define $\nu(i)$,
is at least $\xi^2\ell(\widetilde\gamma_i)\ge \xi^2s_1/10$. Since $\hat \epsilon<10^{-3}\xi^2s_1/C$,
it follows from Lemma~\ref{generic2D} and Proposition~\ref{s1glue}
that any fiber of the  $S^1$-fibration on $U'_2$
that meets the side of $ \widehat C({\mathcal C})$ is disjoint from
$\widehat C'({\mathcal C})$. Thus, the an open subset
$\Omega\subset U'_2$ consisting of all $S^1$-fibers that either meet the complement of $\widehat C({\mathcal C})$ or are closer to $\widehat C({\mathcal C})$ than to $\widehat C({\mathcal C}')$
contains the side of $\widehat C({\mathcal C})$ and  is disjoint from
$\widehat C'({\mathcal C})$. We can replace $\Omega$ by a compact
$3$-manifold $\Omega_0\subset \Omega$ that is saturated under the $S^1$-fibration structure and contains the side of $\widehat C({\mathcal
C})$. There is a
boundary component $T$ of $\Omega_0$ that separates the side of
$\widehat C({\mathcal C})$ from $\widehat C'({\mathcal C})$. Being
orientable and fibered by circles $T$ is diffeomorphic to a
$2$-torus.

 Each end of the chain crosses a
$2$-sphere boundary component of $W'_{n,2}$. Let us denote the these boundary components by $S_{\pm}({\mathcal C})$. By
construction $S_\pm({\mathcal C})$ is the union of an annulus
$E_\pm({\mathcal C})$ contained in $U'_2$ and saturated under the
$S^1$-fibration on $U'_2$ and two $2$-disks contained in the
 of the extremal $\epsilon'$-solid cylinders in the chain
${\mathcal C}$. Furthermore, the annulus $E_\pm({\mathcal C})$
contains all points of $S_\pm({\mathcal C})$ contained in the
 $\widehat C({\mathcal C})\setminus\widehat C'({\mathcal C})$. Thus, the intersection of $T$
with
$$S_\pm({\mathcal C})\cap\left(\widehat C({\mathcal C})\setminus\widehat C'({\mathcal C})\right)$$ is a union of fibers of the
$S^1$-fibration on $U'_2$. Hence, there is an annulus $P({\mathcal C})$ in $T\cap \left(\widehat C({\mathcal C})\setminus\widehat C'({\mathcal C})\right)$ that
is saturated under the $S^1$-fibration on $U'_2$ that has
one boundary circle in $E_+({\mathcal C})$ and the other boundary
circle in $E_-({\mathcal C})$ and is otherwise disjoint from the $S_\pm({\mathcal C})$.
The intersection of $P({\mathcal C})$ with $E_\pm({\mathcal C})$ is a circle bounding a disk $D_\pm$ in the $\pm$ end of $\widehat C({\mathcal C})$. The union of $D_-\cup P({\mathcal C})\cup D_+$ is a $2$-sphere contained in the interior
of the $3$-ball $\widehat C({\mathcal C})$. As such, this union is the boundary of a $3$-ball $\Gamma({\mathcal C})$
in $\widehat C({\mathcal C})$. It follows that there is a homeomorphism
$$(\Gamma({\mathcal C}),P({\mathcal C}))\cong (D^2\times I,\partial D^2\times I).$$

Now we are ready to define $V_{n,2}$. We begin with the compact
$3$-manifold $W'_{n,2}=M_n\setminus {\rm int}\, V_{n,1}$. From this
we remove the interiors of solid tori and of solid cylinders to form
$V_{n,2}$. For each $\epsilon'$-solid torus $B(z_i)$ near an interior
cone point, we remove the interior of the solid torus $\widehat
T(z_i)$ as described in above. For each circular chain of
$\epsilon'$-solid tori ${\mathcal C}$ we remove the interior of the
solid torus $\tau_{\mathcal C}$ described above. For each linear
chain of $\epsilon'$-solid cylinders ${\mathcal C}$ we remove the
sub region $\Gamma({\mathcal C})$ as described above. The boundary
of $V_{n,2}$ consists of $2$-tori contained in $U'_2$ and saturated under the $S^1$-fibration.
They are of three types:
\begin{enumerate}
\item components that are disjoint from $V_{n,1}$,
\item components that are unions of annuli meeting along their boundaries, annuli saturated under the
$S^1$-fibration on $U'_2$; the annuli alternate between the annuli
$P({\mathcal C})$ associated to linear chains of $\epsilon'$-solid
cylinders whose interiors are disjoint from $V_{n,1}$, and  annuli $E$ that are  contained in the
$2$-spheres boundary components of $W'_{n,2}$, and
\item components that are boundary components of $V_{n,1}$.
\end{enumerate}

Since $V_{n,2}\subset U'_2$ is a compact submanifold and its boundary is saturated under the
 $S^1$-fibration on $U'_2$, it follows that $V_{n,2}$ is a compact $3$-manifold that is saturated
under the $S^1$-fibration on $U'_2$. Hence,
$V_{n,2}$ is the total space of a locally trivial circle bundle.
Clearly from the construction the intersection of $V_{n,2}$ with
$V_{n,1}$ consists of the union of (i) all the $2$-torus boundary
components of $V_{n,1}$ that are not boundary components of $M_n$
and (ii) an annulus in  each $2$-sphere boundary component of
$V_{n,1}$. Lastly, the complement $M_n\setminus {\rm int}\,
\left(V_{n,1}\cup V_{n,2}\right)$ consists of a finite disjoint
union  of compact solid tori and of compact solid cylinders
(manifolds homeomorphic to $D^2\times I$). The boundaries of the
solid tori are boundary components of $V_{n,2}$, and the boundary of
each solid cylinder meets $V_{n,2}$ in $\partial D^2\times I$ and
this intersection is saturated under the $S^1$-fibration  on
$V_{n,2}$. This completes the proof that the submanifolds $V_{n,1}$ and $V_{n,2}$
satisfy all the properties given in the conclusion of
Theorem~\ref{1Dthm'}. This establishes Theorem~\ref{1Dthm'} and as a consequence Theorem~\ref{7.4}.

\end{document}